\documentclass[12pt, reqno]{amsart}
\usepackage[T1]{fontenc}
\usepackage[utf8]{inputenc}
\usepackage[english]{babel}
\usepackage{amsmath,amsfonts,amssymb,amsthm,mathrsfs}
\usepackage{accents,mathtools,extarrows}
\usepackage{graphicx}
\usepackage{rotating}
\usepackage{wasysym}
\usepackage{textcomp}
\usepackage{hyperref,cleveref}
\usepackage{xfrac}
\usepackage[all]{xy}
\usepackage{enumitem}
\usepackage{yhmath}
\usepackage{tikz-cd}									
\usetikzlibrary{matrix,arrows,calc,positioning,decorations.pathmorphing}
\tikzset{
  shift left/.style ={commutative diagrams/shift left={#1}},
  shift right/.style={commutative diagrams/shift right={#1}}
}

\theoremstyle{plain}
\newtheorem{theorem}{Theorem}[section]
\newtheorem{lemma}[theorem]{Lemma}
\newtheorem{proposition}[theorem]{Proposition}
\newtheorem{corollary}[theorem]{Corollary}
\newtheorem*{theorem*}{Theorem}
\newtheorem*{corollary*}{Corollary}

\theoremstyle{definition}
\newtheorem{definition}[theorem]{Definition}
\newtheorem{example}[theorem]{Example}
\newtheorem{remark}[theorem]{Remark}
\newtheorem{construction}[theorem]{Construction}
\newtheorem{observation}[theorem]{Observation}

\newenvironment{xreftheorem}[1]{\def\thexref{\ref{#1}}
\begin{thexreftheorem}\itshape}{\end{thexreftheorem}}
\newtheorem*{thexreftheorem}{Theorem \thexref}

\newcommand{\exend}{\unskip\nobreak\hfill$\circ$}
\newcommand{\defend}{\unskip\nobreak\hfill$\triangleleft$}

\newcommand{\N}{\mathbb N}
\newcommand{\Z}{\mathbb Z}
\newcommand{\Q}{\mathbb Q}

\renewcommand{\epsilon}{\varepsilon}
\renewcommand{\phi}{\varphi}
\renewcommand{\tilde}{\widetilde}

\begin{document}

\title[Reductive Borel--Serre and algebraic K-theory]{The reductive Borel--Serre compactification as a model for unstable algebraic K-theory}
\author{Dustin Clausen and Mikala {\O}rsnes Jansen}
\date{\today}
\address{Department of Mathematical Sciences, University of Copenhagen, 2100 Copenhagen, Denmark.}\email{mikala@math.ku.dk}
\thanks{Both authors were supported by the Danish National Research Foundation through the Centre for Symmetry and Deformation (DNRF92) and the Copenhagen Centre for Geometry and Topology (DNRF151). MJ was also supported by the European Research Council (ERC) under the European Unions Horizon 2020 research and innovation programme (grant agreement No 682922).}

\begin{abstract}
Let $A$ be an associative ring and $M$ a finitely generated projective $A$-module. We introduce a category $\operatorname{RBS}(M)$ and prove several theorems which show that its geometric realisation functions as a well-behaved unstable algebraic K-theory space.  These categories $\operatorname{RBS}(M)$ naturally arise as generalisations of the exit path $\infty$-category of the reductive Borel--Serre compactification of a locally symmetric space, and one of our main techniques is to find purely categorical analogues of some familiar structures in these compactifications. 
\end{abstract}

\maketitle

\tableofcontents

\section{Introduction}

\subsection{Unstable algebraic K-theory}
 
Let $A$ be a ring.  The algebraic K-theory space $K(A)$ is an invariant of $A$ which is built from the concrete linear algebra of finitely generated projective modules over $A$.  But $K(A)$ has a subtle nature.  In fact, there are several different ways of defining $K(A)$ as a CW-complex, and they are all different up to homeomorphism; however, they are nonetheless canonically homotopy equivalent.  Thus the true $K(A)$ is this common homotopy type, or \emph{anima}.  An anima is elusive and difficult to grasp, but it anchors itself to reality via concrete invariants such as homotopy groups.  The homotopy groups of $K(A)$ are abelian groups known as the higher K-groups, and they have myriad connections to other invariants of $A$ arising in different contexts.

\medskip

To every finitely generated projective module $M$ corresponds a point in $K(A)$.  Moreover, this association is functorial for isomorphisms, so one obtains a map
$$BGL(M)\rightarrow K(A).$$
This is very far from being an isomorphism, for two reasons: first, these anima have very different nature (one is a $K(\pi,1)$ for a generally non-abelian $\pi$ and the other is a simple space), and second, $K(A)$ takes into account all finitely generated projective modules, not just $M$.  We would like to mitigate the first reason while keeping the second.  More precisely, we want to define an intermediary anima $\overline{BGL(M)}$, a sort of ``closure'' of $BGL(M)$ in $K(A)$, which is similar to $K(A)$ in terms of its nature and properties, but whose definition only uses linear algebra internal to $M$.  Such an intermediary anima is called an \emph{unstable algebraic K-theory}.

\medskip

There have already been several proposed definitions for unstable algebraic K-theory in the literature, mostly in the the special case $M=A^n$.  In contrast to the stable situation of $K(A)$, all of these definitions are in general pairwise inequivalent, even as anima.  Our definition will be yet another one which is generally inequivalent to the others, see below for more remarks on the comparisons.  It will be denoted
$$\vert \operatorname{RBS}(M)\vert.$$

\noindent The notation foreshadows that this anima arises as the geometric realisation of an explicit category $\operatorname{RBS}(M)$ built from linear algebra internal to $M$.  We will say more about the definition and origins of this category later (see \Cref{RBS(M) defined in introduction}).  But first let's state the main results, which all concern the question of how close the natural maps
$$BGL(M)\rightarrow \vert \operatorname{RBS}(M)\vert\rightarrow K(A)$$
are to being isomorphisms.

\medskip

Our arguments are based on an inductive strategy, and for carrying many of them out it is at the very least convenient to impose the following condition on our module $M$:

\begin{definition}
We say that a finitely generated projective $A$-module $M$ is \emph{split noetherian} if every increasing chain of splittable submodules of $M$ stabilises.\defend
\end{definition}

If the ring $A$ is either noetherian or commutative with connected spectrum, then every finitely generated projective $A$-module is split noetherian.

\medskip

Concerning the map $BGL(M)\rightarrow \vert \operatorname{RBS}(M)\vert$, both anima are connected, so the first question is what happens on $\pi_1$.  Let $E(M)\subset GL(M)=\pi_1 BGL(M)$ denote the subgroup generated by those automorphisms of $M$ which induce the identity on the associated graded of some splittable flag of submodules, where we stress the word some, i.e. that we are running through all choices of such flags. We think of $E(M)$ as the subgroup of those elements which map to zero in $K_1(A)$ for reasons purely internal to $M$.  It is a variant of the usual subgroup $E_n(A)\subset GL_n(A)$ generated by elementary matrices; there is a containment $E_n(A)\subset E(A^n)$ which is in general strict, but often an equality, for example $E_n(A)=E(A^n)$ if $n\geq 2 + \operatorname{sr}(A)$ so that $E_n(A)=\operatorname{ker}(GL_n(A)\rightarrow K_1(A))$, see \cite{V}.

\medskip

Our first result is actually fairly straightforward to prove from the definition, but it already gives a good indication of the nature of $\vert \operatorname{RBS}(M)\vert$.

\begin{theorem}
Let $A$ be a ring and $M$ a split noetherian finitely generated projective $A$-module.  The map $GL(M)=\pi_1 BGL(M)\rightarrow \pi_1\vert \operatorname{RBS}(M)\vert$ is surjective with kernel $E(M)$, so
$$\pi_1 \vert \operatorname{RBS}(M)\vert = GL(M)/E(M).$$

\end{theorem}

Our next result says that for a large class of rings, this $\pi_1$ calculation completely captures the difference between $BGL(M)$ and $\vert \operatorname{RBS}(M)\vert$.  It is based on the work of Nesterenko-Suslin \cite{NS} who found a broadly satisfied hypothesis on a ring which guarantees that one can ignore the difference between block upper-triangular and block diagonal matrices when calculating group homology.

\begin{theorem}\label{plus}
Let $A$ be a ring with many units in the sense of \cite{NS}, and let $M$ be a split noetherian finitely generated projective $A$-module. Then the comparison map
$$c\colon BGL(M)\rightarrow \vert\operatorname{RBS}(M)\vert$$
is a $\mathbb{Z}$-homology isomorphism.

Suppose furthermore that every summand of $M$ is free.  Then $c$ is an isomorphism on homology with all local coefficient systems.  In particular, $E(M)$ is a perfect group, and
$$\vert \operatorname{RBS}(M)\vert \simeq BGL(M)^+,$$
the plus-construction taken with respect to $E(M)\subset \pi_1 BGL(M)$.  Equivalently, $\vert \operatorname{RBS}(M)\vert$ is the initial anima with a map from $BGL(M)$ which kills $E(M)\subset\pi_1 BGL(M)$.
\end{theorem}

Thus, for such rings $\vert \operatorname{RBS}(M)\vert$ provides an explicit linear-algebraic model for the plus-construction, which is otherwise a slightly esoteric homotopy-theoretic construction.   There are lots of rings with  many units, for example any algebra over a commutative local ring with infinite residue field.  A commutative local ring also satisfies the hypothesis that every finitely generated projective module is free.

\medskip

The simplest non-example is a finite field, and our third theorem analyses this case to see the difference with the plus construction.  As we will explain below, the resulting theorem should properly be attributed to Jesper Grodal, since in \cite{G} he proved a more general result in the context of arbitrary finite groups which specialises to this result for $G=GL(M)$.  However, we do give an independent proof based on the general machinery for analysing $\operatorname{RBS}(M)$ categories that we develop.  For purposes of comparison, it is also worth noting that for a finite ring which is not a product of fields, our categories $\operatorname{RBS}(M)$ do not fit into Grodal's framework (nor do we prove any results about that situation).

\begin{theorem}\label{finite}
Let $k$ be a finite field of characteristic $p$ and $V$ a finite-dimensional $k$-vector space.  Then:
\begin{enumerate}
\item $\vert \operatorname{RBS}(V)\vert$ is a simple space;
\item The map $\vert\operatorname{RBS}(V)\vert\rightarrow \ast$ is an $\mathbb{F}_p$-homology isomorphism;
\item The map $BGL(V)\rightarrow \vert\operatorname{RBS}(V)\vert$ is a $\mathbb{Z}[1/p]$-homology isomorphism.
\end{enumerate}
In particular, $\vert \operatorname{RBS}(V)\vert$ is the $\mathbb{Z}[1/p]$-homology localisation of $BGL(V)$.
\end{theorem}

We recall that the $\mathbb{Z}[1/p]$-homology of $BGL(V)$ was completely calculated by Quillen in the early days, \cite{Qfinite}.  On the other hand, the $\mathbb{F}_p$-homology is nontrivial, rather complicated, and still largely unknown, \cite{MP}, \cite{LS}.  However, Quillen in \cite{Qfinite} also showed that in the stable range the $\mathbb{F}_p$-homology vanishes, so the complicated part does not contribute to $K(\mathbb{F}_q)$.  Thus, compared to existing models such as the plus-construction, our new model for unstable algebraic K-theory exactly removes the complicated unknown part which anyway dies on stabilisation.  Actually we can rephrase the above theorem as giving an identification
$$\vert \operatorname{RBS}(V)\vert\simeq (BU(n)')^{h\psi^q}$$
as the homotopy fixed points for the unstable $q$-Adams operation on the prime-to-$p$ completion of $BU(n)$ for $n=\operatorname{dim}(V)$.  This is the evident unstable analogue of Quillen's identification of the $0$-component of the K-theory space
$$K(\mathbb{F}_q)_0\simeq (BU')^{h\psi^q}.$$

\begin{remark}
The crucial point is part 2, that $\vert \operatorname{RBS}(V)\vert$ has the $\mathbb{F}_p$-homology of a point.  This can also be deduced from a more general theorem of Jesper Grodal, \cite{G}.  Indeed, Grodal's Theorem 4.3 says that for any finite group $G$ and prime $p$, if $\mathcal{C}$ denotes the \emph{$p$-radical orbit category} of $G$, then $\vert \mathcal{C}\vert$ has the $\mathbb{F}_p$-homology of a point.  For $G=GL(V)$ it is a matter of comparing definitions and invoking a theorem of Borel-Tits (\cite{BorelTits71,BurgoyneWilliamson}), to see that $\mathcal{C}=\operatorname{RBS}(V)^{op}$, see the discussion in \cite{OrsnesJansen}, and hence our theorem follows from Grodal's.
\exend
\end{remark}

Now we turn to the relation between $\vert \operatorname{RBS}(M)\vert$ and $K(A)$.  Our last theorem gives a sense in which the $\vert \operatorname{RBS}(M)\vert$ stabilise to $K(A)$.

\begin{theorem}\label{gpcomplete}
Let $A$ be a ring.  Let $\mathcal{M}$ denote a set of representatives for the isomorphism classes of finitely generated projective $A$-modules.  Then there is a natural structure of a monoidal category on $\coprod_{M\in\mathcal{M}} \operatorname{RBS}(M)$ and an identification
$$K(A) \simeq \bigg\vert \coprod_{M\in\mathcal{M}} \operatorname{RBS}(M)\,\bigg\vert^{gp}$$
of $K(A)$ with the group completion of the realisation of this monoidal category.
\end{theorem}

We also prove a more general version of this theorem which describes in similar terms the K-theory of an arbitrary exact category in the sense of Quillen, \cite{QQ}.

\medskip

The version of \Cref{gpcomplete} with $BGL(M)$ instead of $\operatorname{RBS}(M)$ is essentially Segal's definition of algebraic K-theory, \cite{Segal}.  However, there is a very important technical difference between the two situations, in that $\coprod BGL(M)$ forms a symmetric monoidal category, whereas $\coprod \operatorname{RBS}(M)$ really only forms a monoidal category.  This means that as it stands we cannot use the ``group completion theorem'' of \cite{MS} to relate this group completion to the more naive procedure of taking the limiting object
$$\varinjlim_n \vert \operatorname{RBS}(A^n)\vert$$
along the natural stabilisation maps. Indeed, the group completion theorem requires some commutativity hypothesis which we don't know whether is satisfied for $\vert \coprod \operatorname{RBS}(M)\vert$ for general $A$ and $M$.

\subsection{The reductive Borel--Serre category}\label{RBS(M) defined in introduction}

Perhaps the most important aspect of our model is that it is given as the geometric realisation of an explicit category $\operatorname{RBS}(M)$.  Although we were led to this category by other means which we will discuss below, one can motivate it in terms of the following key property of algebraic K-theory:  if $M$ is a finitely generated projective $A$-module and
$$\mathcal{F}= (M_1\subsetneq\ldots \subsetneq M_{d-1})$$
is a splittable flag in $M$, so that each graded piece $M_i/M_{i-1}$ is nonzero and finitely generated projective (we set $M_0=0$ and $M_d=M$), then there is a canonically determined path
$$[M]\sim [\oplus_{i=1}^d M_i/M_{i-1}]$$
in $K(A)$.  Thus, in the eyes of K-theory, every filtration is split. One can also say this in a different way.  Let $P_\mathcal{F}\subset GL(M)$ denote the stabiliser of the flag $\mathcal{F}$ and $U_{\mathcal{F}}\subset P_{\mathcal{F}}$ the subgroup consisting of those elements which induce the identity on associated graded.  Then the restriction of $BGL(M)\rightarrow K(A)$ to $BP_\mathcal{F}$ naturally factors through $B(P_\mathcal{F}/U_\mathcal{F})$.  There are also a host of compatibilities satisfied by these canonical paths relating their functoriality under automorphisms and their behaviour under refinement of flags.  This leads to the following.

\begin{definition}
Let $A$ be a ring and $M$ a finitely generated projective $A$-module.  Define the category $\operatorname{RBS}(M)$ to have:
\begin{enumerate}
\item objects the splittable flags of submodules of $M$
$$\mathcal{F} = (M_1\subsetneq\ldots \subsetneq M_{d-1});$$
\item morphisms $\mathcal{F}\rightarrow \mathcal{F}'$ the set
$$\{g\in GL(M) : g\mathcal{F}\leq \mathcal{F}'\}/U_\mathcal{F},$$
where the partial order $\leq$ is the relation of refinement: $\mathcal{F}\leq \mathcal{F}'$ when the modules occurring in $\mathcal{F}'$ are a subset of those occurring in $\mathcal{F}$;
\item composition induced by multiplication in $GL(M)$.\defend
\end{enumerate}
\end{definition}

The empty flag $[\emptyset]$ has automorphism group $GL(M)$ in $\operatorname{RBS}(M)$. This produces a map $BGL(M)\rightarrow \vert\operatorname{RBS}(M)\vert$, and the preceding discussion hopefully makes it plausible that the natural map $BGL(M)\rightarrow K(A)$ factors through it:
$$BGL(M)\rightarrow \vert\operatorname{RBS}(M)\vert \rightarrow K(A).$$
However, this is a non-trivial claim which of course requires proof.  To explain this factoring as well as the more refined \Cref{gpcomplete}, it's useful to look at $\operatorname{RBS}(M)$ from a more intrinsic perspective. For a splittable flag $\mathcal{F}$, the automorphism group of $\mathcal{F}$ in $\operatorname{RBS}(M)$ identifies not with the automorphisms of $\mathcal{F}$ as a flag, but with the automorphisms of its associated graded.  Thus one should think that the objects of $\operatorname{RBS}(M)$ are not really flags, since giving a flag over-specifies the object. Rather the objects should be some abstract ordered list
$$(N_1,\ldots, N_d)$$
of nonzero finitely generated projective modules, which we imagine as the associated graded of some undetermined flag.  The flags themselves only really come in to play when describing the morphisms: namely a map $(N_1,\ldots,N_d)\rightarrow (N'_1,\ldots, N'_e)$ can only exist when $d\geq e$, and then it is the data of a flag on each $N'_j$ together with an isomorphism of the total associated graded of this list of flags with the $N_i$, in order.  There is an equivalent model for $\operatorname{RBS}(M)$ of exactly this form, and it is this model that is the most useful for giving the comparison with $K(A)$.

\medskip

Going in a different direction, when $A$ is a commutative ring there is yet another interpretation of $\operatorname{RBS}(M)$, this time in terms of the group $GL(M)$ viewed now as a reductive group scheme over $A$ instead of just the abstract group of its $A$-valued points.  This is simplest to state when $\operatorname{Spec}(A)$ is connected.  Then splittable flags $\mathcal{F}$ in $M$ are in bijection with \emph{parabolic subgroups} of $GL(M)$ via assigning to $\mathcal{F}$ its stabiliser $P_\mathcal{F}$.  Moreover the subgroup $U_\mathcal{F}\subset P_\mathcal{F}$ of those automorphisms of the flag inducing the identity on associated graded is recovered as the \emph{unipotent radical} of $P_\mathcal{F}$.  Thus one can also describe $\operatorname{RBS}(M)$ in reductive group terms: the objects are the parabolic subgroups, and the maps are the transporters of these subgroups taken modulo the unipotent radical of the source parabolic subgroup.

\medskip

This ties in to our initial motivation for defining $\operatorname{RBS}(M)$. Let $G$ be a connected reductive linear algebraic group defined over $\Q$, and $X=K\backslash G(\mathbb{R})/A_G$ the usual associated contractible symmetric space. For a neat arithmetic group $\Gamma\leq G(\Q)$, the locally symmetric space $\Gamma\backslash X$ is a model for the classifying space $B\Gamma$ --- unfortunately, it is very rarely compact. The Borel--Serre and reductive Borel--Serre compactification are two important compactifications of such locally symmetric spaces. The Borel--Serre compactification $\Gamma\backslash \widehat{X}$ is a compact smooth manifold with corners with the same homotopy type as $\Gamma\backslash X$. It was introduced in 1973 by Borel and Serre (\cite{BS}) and it was used crucially in Borel's calculation of the ranks of the K-groups $K_i(O_F)$ of the ring of integers $O_F$ in a number field $F$ (\cite{B}). It was also used by Quillen to show that these same K-groups are finitely generated (\cite{Qfg}). The reductive Borel--Serre compactification $\widehat{Y}_\Gamma$ was introduced by Zucker in 1982 as a quotient of the Borel--Serre compactification (\cite{Z}). Zucker was originally motivated by an interest in $L^2$-cohomology, but the reductive Borel--Serre compactification has since come to play a prominent and diverse role in the theory of compactifications.

\medskip

The Borel--Serre compactification is naturally stratified as a manifold with corners, and this stratification descends to define a natural stratification of the reductive Borel--Serre compactification. In \cite{OrsnesJansen}, the exit path $\infty$-category (or stratified homotopy type) of the reductive Borel--Serre compactification $\widehat{Y}_\Gamma$ is identified as a $1$-category $\operatorname{RBS}_\Gamma$ whose objects are the rational parabolic subgroups of $G$ and whose morphisms are given by transporters of these subgroups by elements in $\Gamma$ modulo an action of the unipotent radicals. The category $\operatorname{RBS}(M)$ introduced in this paper is a direct generalisation of the category $\operatorname{RBS}_\Gamma$, cf. the reductive group approach above.

\medskip

In fact, we also provide a proof of the identification of the exit path $\infty$-category of the reductive Borel--Serre compactification in this paper. Our proof uses entirely different methods to the one given in \cite{OrsnesJansen}, and we find that the two different proofs complement each other nicely, as they provide very different insights into the structure of the reductive Borel--Serre compactification.  Moreover, both methods are quite general in nature and have the potential to be useful for studying the exit path $\infty$-categories of other stratified spaces, so we think it is worthwhile to have them both explained.  Whereas the method in \cite{OrsnesJansen} is based on the idea of calculating mapping spaces in the exit path $\infty$-category in terms of the homotopy-theoretic data embodied in the links of the strata, the method in this paper is based on the idea of finding a way to glue our stratified space from simpler pieces whose exit path $\infty$-categories are equivalent to posets.  If the gluing is robust enough, this reduces the determination of the exit path $\infty$-category of our space to the calculation of a colimit in the $\infty$-category of $\infty$-categories.  We also provide a toolkit for computing such colimits.

\medskip

As we will see, the proof we present here has advantages with respect to the broader aim of this paper, namely in comparing with unstable K-theory, as the proof strategy by gluing can be transported over to the context of $\operatorname{RBS}(M)$ and exploited to make the necessary homology calculations.  In order to make the various calculations and identifications in this paper, we start out by developing a variety of tools for identifying and calculating colimits of $\infty$-categories. This allows us to exploit the inductive nature of the reductive Borel--Serre compactification, namely the fact that its boundary admits a closed cover by ``smaller'' reductive Borel--Serre compactifications, and we can mimick this when working with the generalisations $\operatorname{RBS}(M)$.

\medskip

We would like to note that the existence of a relationship between compactifications of locally symmetric spaces and algebraic K-theory is not original to this article.  As Dan Petersen pointed out to us, Charney and Lee wrote an article \cite{CL} in 1982 in which they established such a relationship for the Satake compactification of the Siegel modular variety.  They show that the homotopy type of the Satake compactification is rationally equivalent to the geometric realisation of a category $W_n$ whose stable version $W$ fits into a fibre sequence
$$K(\mathbb{Z})\rightarrow K^{sympl}(\mathbb{Z})\rightarrow \vert W\vert,$$
and therefore describes the difference between K-theory and symplectic K-theory of the integers.  What we have, then, is an analogue of the Charney--Lee result for the reductive Borel--Serre compactification and plain algebraic K-theory.  Moreover, the modern notion of exit path $\infty$-category lets us make a much more refined statement of the relationship, showing that not just the (rational) homotopy type, but the whole stratified homotopy type, as well as the theory of constructible sheaves, are determined by the associated category.

\medskip

Let us sum up and reiterate our main point: we provide an explicit category $\operatorname{RBS}(M)$ whose homotopy type is a reasonable unstable algebraic K-theory anima.  This indicates that unstable algebraic K-theory can naturally be viewed not as a bare anima or homotopy type, but rather as a stratified homotopy type, with structure very much akin to those stratified homotopy types arising from familiar compactifications of locally symmetric spaces.

\subsection{Comparison with previous approaches}

There have been several previous approaches to unstable algebraic K-theory.  Here we'd like to point out the ones we know about and say what we can about how our definition compares.

\medskip

First, there is the plus construction definition.  If $n\geq 3$, the subgroup $E_n(A)\subset GL_n(A)$ generated by elementary matrices is perfect, \cite{W} Lemma 1.3.2, so one can form the plus construction on $BGL_n(A)$ which kills the normal subgroup generated by $E_n(A)$.  By \Cref{plus} above, this agrees with our $\vert \operatorname{RBS}(A^n)\vert$ provided that $A$ is commutative and local with infinite residue field.   On the other hand, our \Cref{finite} shows that for finite fields, the two definitions differ, and ours yields an unstable algebraic K-theory space which is much simpler and closer in nature to the stable K-theory.

\medskip

Second, there is the Volodin definition, see \cite{S}.  At first glance this looks quite similar, since it is based on the same idea of contracting away unipotent matrix groups.  But the contraction happens in a very different way in Volodin's model: one considers all of the $\Sigma_n$-conjugates of the strict upper-triangular group and simultaneously collapses them, compatibly along their various intersections.  Already in unstable $K_1$ one sees a difference, in that the Volodin $K_1$ is the quotient $GL_n(A)/E_n(A)$, which is not necessarily a group in general but just a pointed set.  It also seems from our (albeit limited) experience that arguments which work for Volodin K-theory do not work for our model and vice-versa, so the nature of the two models really is quite different.

\medskip

Finally, there is Allen Yuan's quite recent \emph{partial K-theory}, \cite{Y}.  This had not yet appeared when we were proving our results, but it indeed seems very similar to our proposed model.  Partial K-theory is defined essentially so as to make the analogue of our \Cref{gpcomplete} a tautology (whereas for us the proof takes many pages of simplicial manipulations!).  That is, Yuan takes Waldhausen's S-dot construction, and instead of freely making a group-like $E_1$-anima out of it, which produces $K(A)$, he freely makes an $E_1$-anima without the group-like condition, and this is the definition of $K^\partial(A)$.  It is clear that partial K-theory should be similar to our $E_1$-anima $\vert\coprod \operatorname{RBS}(M)\vert$, because the S-dot construction exactly encodes filtrations and their associated gradeds with all compatibilities, and this was the essence of our $\operatorname{RBS}$ categories as well.  But it turns out that when Yuan unravels $K^\partial$ into something concrete, it ends up being slightly more combinatorially intricate, in that the basic objects are not lists of finitely generated projective modules, but lists of lists of finitely generated projective modules.  The two models for unstable K-theory unwind to the same thing when all flags on $M$ have length $\leq 2$, but in other cases they are a priori different and it's not clear whether or not the anima are nonetheless equivalent.  This would be interesting to investigate, because Yuan shows by an Eckmann--Hilton argument that $K^\partial(A)$ actually is $E_\infty$, which means the group completion theorem does apply to it.  Yuan also proves the analogue of our \Cref{finite} part 2 for $K^\partial$ of finite fields, and crucially uses this result in his work giving a new model for unstable homotopy theory.  Moreover, his proof has the same rough outline as ours: after some combinatorial shuffling one reduces to the fact that the $\mathbb{F}_p$-homology of the Steinberg representation of $GL_n(k)$ vanishes.

\medskip

We'd also like to make a small remark comparing the monoidal category $\coprod_{M\in \mathcal{M}} \operatorname{RBS}(M)$ with the symmetric monoidal category $\coprod_{M\in \mathcal{M}} BGL(M)$, where $\mathcal{M}$ is the set of isomorphism classes of finitely generated projective $A$-modules for a fixed ring $A$. Their realisations are naturally $\mathbb{E}_1$-spaces and we can compare their $\mathbb{E}_1$-homology, an $\mathbb{E}_1$-version of André-Quillen homology for simplicial commutative rings (see \cite{GKRWCellular}). Note that $\coprod_M BGL(M)$ is an $\mathbb{E}_\infty$-space, but because of the lack of commutativity in $\coprod_M \operatorname{RBS}(M)$, we do not at this point know if $\coprod_M|\operatorname{RBS}(M)|$ is an $\mathbb{E}_k$-space for $k>1$. For clarity, we will consider the case of a field $A=K$ (see the comment below about the generality in which the following calculations do or should hold). Then $\mathcal{M}\cong \N$ and the (bigraded) $\mathbb{E}_1$-homology of $\coprod_i BGL_i(K)$ is
\begin{align*}
H^{\mathbb{E}_1}_{n,d}\big(\textstyle\coprod_i BGL_i(K); \Z\big)\cong \bigoplus_n H_{d-n+1}(GL_n(K); \operatorname{St}_n^{\operatorname{split}}(K)),
\end{align*}
where $\operatorname{St}_n^{\operatorname{split}}(K)=\tilde{H}_{n-2}(S_n(K); \Z)$ is the split (or $\mathbb{E}_1$-) Steinberg module defined as the degree $n-2$ reduced homology of Charney's split building $S_n(K)$ (see the calculation in \cite[Section 17.2]{GKRWCellular}; see also \cite[Section 3.3]{GKRWFiniteFields} and \cite{Charney}). The following should be true and should follow by more or less directly generalising the calculation in \cite[Section 17.2]{GKRWCellular}:
\begin{align*}
H^{\mathbb{E}_1}_{n,d}\big(\textstyle\coprod_i | \operatorname{RBS}(K^i) |; \Z\big)\cong \bigoplus_n H_{d-n+1}(GL_n(K); \operatorname{St}_n(K)),
\end{align*}
where $\operatorname{St}_n(K)=\tilde{H}_{n-2}(T_n(K); \Z)$ is the usual (non-split) Steinberg module defined as the degree $n-2$ reduced homology of the Tits building $T_n(K)$. The calculation should reflect the fact that $\coprod_n \operatorname{RBS}(K^n)$ admits a filtration by monoidal categories whose associated graded is given by the $GL_n(K)$-coinvariants of the Steinberg-modules $\operatorname{St}_n(K)$ for varying $n$ (cf. the proof of \Cref{finite field Fp coeff}). The calculations mentioned here will only work when the reduced homology of the split (respectively non-split) Tits building is concentrated in one degree (for example by the Solomon--Tits Theorem (eg. \cite{Solomon, AbramenkoBrown}) or Charney's version for the split building (\cite{Charney})).  If $\mathcal{M}\ncong \N$, then the direct sums in the calculations above should just be replaced by the direct sum over $\mathcal{M}$ with bidegree $(n,d)$ on the left hand side corresponding to the direct sum over the rank $n$ modules of the degree $d-n+1$ homology of the corresponding $GL(M)$.

\subsection{Conventions and notation}

We let $\mathcal{S}$ denote the $\infty$-category of anima, and $\operatorname{Cat}_\infty$ the $\infty$-category of (small) $\infty$-categories.  We often view $\mathcal{S}$ as the full subcategory of $\operatorname{Cat}_\infty$ consisting of the $\infty$-groupoids.  For a topological space $X$, if we write $\operatorname{Sh}(X)$ or talk about sheaves on $X$ without specifying further, we mean to consider sheaves of anima, i.e.\ sheaves with values in the $\infty$-category $\mathcal{S}$.  The same goes for presheaves on a category or $\infty$-category.  We view posets as categories with at most one morphism between any two objects: $x\leq y$ means there is a map $x\rightarrow y$.

\subsection{Acknowledgements}

We would like to thank Ko Aoki, Dan Petersen, Allen Yuan, Joshua Hunt and S{\o}ren Galatius for helpful discussions.  We also thank the referee for numerous pertinent comments and questions which led to an improvement of the text.

\section{Colimits in \texorpdfstring{$\operatorname{Cat}_\infty$}{Cat-infinity}}

In this section we will describe how to calculate certain colimits in the $\infty$-category of small $\infty$-categories.  We note right away that there is a general description of such colimits as a localisation of the total space of the cartesian fibration classified by the diagram of $\infty$-categories, see \cite{L} 3.3.4; but this is not what we're after.  Rather we want simple criteria for showing that a given co-cone diagram is a colimit diagram.

\medskip

In the cases we care about all the $\infty$-categories in our colimit diagram will actually be 1-categories, but still it being a colimit diagram in $\operatorname{Cat}_\infty$ is stronger than it being a colimit diagram in $\operatorname{Cat}_1$, and we need this stronger fact to get our desired consequences.

\subsection{Some consequences of having a colimit in \texorpdfstring{$\operatorname{Cat}_\infty$}{Cat-infinity}}

We start by explaining why we care about colimits in $\operatorname{Cat}_\infty$.  First, as shown in \cite{HY}, they let you decompose both colimits and limits.  We give a slightly different argument based on the Yoneda embedding.

\begin{proposition}\label{decomposelimits}
Let $K$ be an $\infty$-category and $d\colon K\rightarrow\operatorname{Cat}_\infty$ a $K$-diagram in $\operatorname{Cat}_\infty$, with colimit $\mathcal{D}:=\varinjlim_K d$.  Suppose given an $\infty$-category $\mathcal{E}$ and a functor
$$F\colon\mathcal{D} \rightarrow\mathcal{E}.$$
\begin{enumerate}
\item We have
$$\varprojlim F\overset{\sim}{\rightarrow} \varprojlim_{k\in K^{op}} \varprojlim F\vert_{d(k)}$$
in the sense that if the limits on the right exist then so does the limit on the left, and the map is an equivalence.
\item We have
$$\varinjlim F\overset{\sim}{\leftarrow} \varinjlim_{k\in K} \varinjlim F\vert_{d(k)}$$
in the sense that if the colimits on the right exist then so does the colimit on the left, and the map is an equivalence.
\end{enumerate}
The natural comparison maps in play above will be constructed in the course of the proof.
\end{proposition}
\begin{proof}
Since $\mathcal{C}\mapsto \mathcal{C}^{op}$ is an equivalence of $\operatorname{Cat}_\infty$ with itself, it preserves colimits. Thus 2 follows from 1 by replacing every $\infty$-category with its opposite.  We can always Yoneda-embed $\mathcal{E}\hookrightarrow \operatorname{Fun}(\mathcal{E}^{op},\mathcal{S})$ and therefore reduce to $\mathcal{E}$ being a presheaf $\infty$-category; thus to construct the comparison maps in general it suffices to construct them functorially in the case $\mathcal{E}=\mathcal{S}$, and similarly to prove they are equivalences it suffices to treat that case.

\medskip

We note that
$$\operatorname{Fun}(\mathcal{D},\mathcal{S})\overset{\sim}{\rightarrow} \varprojlim_{k\in K^{op}} \operatorname{Fun}(d(k),\mathcal{S})$$
by taking maps out of our colimit diagram to $\operatorname{Fun}(\Delta^n,\mathcal{S})$ and using adjunction.  Now given an $F\in \operatorname{Fun}(\mathcal{D},\mathcal{S})$ we can simply evaluate maps from the terminal functor $\ast$ to $F$ via the above equivalence to deduce the required equivalence.
\end{proof}

Let $\vert\cdot \vert\colon\operatorname{Cat}_\infty\rightarrow\mathcal{S}$ denote the left adjoint to the inclusion of anima into $\infty$-categories.  There are many ways of describing this functor; see Section \ref{realise}.  But in any case, it commutes with colimits, and so from a colimit diagram in $\operatorname{Cat}_\infty$ we obtain a colimit diagram in $\mathcal{S}$, that is a homotopy colimit diagram in the classical language:

\begin{proposition}
Let $K$ be an $\infty$-category and $d\colon K\rightarrow\operatorname{Cat}_\infty$ a $K$-diagram in $\operatorname{Cat}_\infty$, with colimit $\mathcal{D}:=\varinjlim_K d$.  Then
$$\varinjlim_{k\in K} \vert d(k)\vert\overset{\sim}{\rightarrow} \vert\mathcal{D}\vert.$$
\end{proposition}

In particular, this means we have a spectral sequence for computing homology of local systems on $\vert \mathcal{D}\vert $ in terms of the homology of their pullback to the $\vert d(k)\vert$, and in the case where $K$ is the poset $(1>0<1')$ this means a Mayer--Vietoris sequence.  (These consequences could also be obtained from the previous proposition by taking $\mathcal{E}=D(\operatorname{Ab})$).

\subsection{Testing by applying \texorpdfstring{$\operatorname{Fun}(-,\mathcal{S})$}{Fun(-,S)}}

Here we prove the following basic result.

\begin{theorem}\label{detectcolim}
Let $K$ be a small $\infty$-category, and $d\colon K^{\triangleright}\rightarrow\operatorname{Cat}_\infty$ a co-cone diagram of small $\infty$-categories indexed by $K$.  Then $d$ is a colimit diagram if and only if the following two conditions are satisfied:
\begin{enumerate}
\item As $k$ runs over the objects of $K$, the functors $d(k)\rightarrow d(\infty)$ are jointly essentially surjective.
\item The cone diagram of $\infty$-categories obtained by applying $\operatorname{Fun}(-,\mathcal{S})$ to $d$ is a limit diagram.
\end{enumerate}
\end{theorem}

To prove this we will need some preliminaries on presentable $\infty$-categories.   First, recall that for every $\mathcal{C}\in\operatorname{Cat}_\infty$, there is a presentable $\infty$-category $\mathcal{P}(\mathcal{C})$ with a fully faithful functor $\mathcal{C}\rightarrow \mathcal{P}(\mathcal{C})$ uniquely characterised by the universal property that colimit-preserving functors $\mathcal{P}(\mathcal{C})\rightarrow \mathcal{D}$ are equivalent, via restriction, to arbitrary functors $\mathcal{C}\rightarrow \mathcal{D}$.  In fact, $\mathcal{P}(\mathcal{C})$ can be taken to be the $\infty$-category $\operatorname{Fun}(\mathcal{C}^{op},\mathcal{S})$ of presheaves on $\mathcal{C}$ and $\mathcal{C}\rightarrow\operatorname{Fun}(\mathcal{C}^{op},\mathcal{S})$ to be the Yoneda embedding $h$, see \cite{L} 5.1, though we would rather not emphasise this description.

\medskip

Let us characterise the presentable $\infty$-categories of the form $\mathcal{P}(\mathcal{C})$.

\begin{definition}
An object $X$ of a presentable $\infty$-category $\mathcal{D}$ is called \emph{atomic} if the functor $\operatorname{Map}(X,-)\colon\mathcal{D}\rightarrow\mathcal{S}$ commutes with all colimits.   Write $\mathcal{D}^{atom}\subset \mathcal{D}$ for the full subcategory of atomic objects.\defend
\end{definition}

We refer to \cite{L} 4.4.5 for the notion of idempotent-complete $\infty$-categories and the operation of idempotent completion.

\begin{lemma}
For $\mathcal{D}\in\operatorname{Pr}^L$, the $\infty$-category $\mathcal{D}^{atom}$ is essentially small and idempotent-complete.
\end{lemma}
\begin{proof}
Since $\mathcal{D}$ is presentable, every object $X\in\mathcal{D}$ is a colimit of objects each of which lies in some fixed small idempotent-complete full subcategory of $\mathcal{D}$ (namely, the full subcategory of $\kappa$-small objects, if $\mathcal{D}$ is $\kappa$-accessible).  If $X$ is atomic, then this means the identity map on $X$ factors through an object of that full subcategory, hence $X$ lies in that full subcategory.  Thus $\mathcal{D}^{atom}$ is essentially small.  It is also idempotent-complete since $\mathcal{D}$ is (being co-complete), and a retract of an atomic object is clearly atomic.
\end{proof}

\begin{lemma}\label{atomic} \ 
\begin{enumerate}
\item For $\mathcal{C}\in\operatorname{Cat}_\infty$, an object of $\mathcal{P}(\mathcal{C})$ is atomic if and only if it is a retract of an object in the image of $\mathcal{C}\rightarrow\mathcal{P}(\mathcal{C})$.  In particular, $\mathcal{P}(\mathcal{C})$ is generated under colimits by its atomic objects (as it is generated under colimits by objects in the Yoneda image, \cite{L} 5.1).
\item Conversely, if $\mathcal{D}\in\operatorname{Pr}^L$ is generated under colimits by its atomic objects, then the induced colimit-preserving functor $\mathcal{P}(\mathcal{D}^{atom})\rightarrow\mathcal{D}$ is an equivalence.
\item For a colimit-preserving functor $f\colon\mathcal{C}\rightarrow \mathcal{D}$ between presentable $\infty$-categories generated under colimits by their atomic objects, we have $f(\mathcal{C}^{atom})\subset \mathcal{D}^{atom}$ if and only if the right adjoint of $f$ commutes with colimits.
\end{enumerate}

\end{lemma}
\begin{proof}
First of all we note that each $X\in\mathcal{P}(\mathcal{C})$ in the Yoneda image is atomic, since $\operatorname{Map}(h_c,-) = (-)(c)$ and colimits are computed objectwise in presheaf categories, \cite{L} 5.1.  Since the collection of atomic objects is closed under retracts, this shows one direction of 1.  For the other direction, suppose $X$ is atomic.  Then we can write $X$ as a colimit of objects in the Yoneda image.  By definition of atomic the identity map $X\rightarrow X$ factors through some stage of this colimit, so $X$ is a retract of an object in the Yoneda image.

\medskip

Now we show 2.  The functor $\mathcal{P}(\mathcal{D}^{atom})\rightarrow\mathcal{D}$ is fully faithful for general $\mathcal{D}\in\operatorname{Pr}^L$, as we see by writing each object in $\mathcal{P}(\mathcal{D}^{atom})$ as a colimit of objects in $\mathcal{D}^{atom}$.  Then the assumption exactly guarantees that it's also essentially surjective.

\medskip

Finally, 3 is immediate by adjunction.
\end{proof}

We note that the universal property of $\mathcal{P}(-)$ gives a covariant functoriality, more specifically $\mathcal{P}\colon\operatorname{Cat}_\infty\rightarrow\operatorname{Pr}^L$.\footnote{In terms of the indentification of $\mathcal{P}(\mathcal{C})$ with $\operatorname{Fun}(\mathcal{C}^{op},\mathcal{S})$, there is another description of this same functoriality, as being obtained from the obvious contravariant pullback functoriality by coherent passage to left adjoints.  Thankfully, these two descriptions have now been proved to be equivalent in \cite{HLN}; this is a fundamental though quite subtle claim.  Using this equivalence would simplify the arguments which follow, but we will avoid doing so and work only with the functoriality coming from the universal property.}

\begin{proposition}\label{smallvbig}
The functor $\mathcal{C}\mapsto \mathcal{P}(\mathcal{C})$ gives an equivalence from the $\infty$-category of idempotent-complete small $\infty$-categories to the subcategory of $\operatorname{Pr}^L$ whose objects are the presentable $\infty$-categories generated by atomic objects and whose morphisms are the colimit-preserving functors whose right adjoint also preserves colimits.
\end{proposition}
\begin{proof}
We claim that an inverse functor is given by $\mathcal{D}\mapsto \mathcal{D}^{atom}$.  This is well-defined on the subcategory by part 3 of the lemma above.  From part 1 of the lemma above, we know that if $\mathcal{C}$ is idempotent-complete, then $\mathcal{C}\overset{\sim}{\rightarrow}\mathcal{P}(\mathcal{C})^{atom}$.  On the other hand, from part 2 we know that if $\mathcal{D}$ lies in the subcategory then $\mathcal{P}(\mathcal{D}^{atom})\overset{\sim}{\rightarrow}\mathcal{D}$.
\end{proof}

Now, our desired \Cref{detectcolim} follows by combining parts 1 and 2 of the following.

\begin{proposition}
Let $K$ be a small $\infty$-category, and $d\colon K^{\triangleright}\rightarrow\operatorname{Cat}_\infty$ a co-cone diagram of small $\infty$-categories indexed by $K$.  Then:
\begin{enumerate}
\item the map $\varinjlim d\vert_K\rightarrow d(\infty)$ is an equivalence after applying idempotent completion if and only if applying $\operatorname{Fun}(-,\mathcal{S})$ to $d$ gives a limit diagram of $\infty$-categories;
\item $\varinjlim d\vert_K\rightarrow d(\infty)$ is an equivalence if and only if it is an equivalence after applying idempotent completion and the $d(k)\rightarrow d(\infty)$ are jointly essentially surjective, $k\in K$.
\end{enumerate}
\end{proposition}
\begin{proof}
Let's prove 1.  The direction $\Rightarrow$ follows by mapping out to $\operatorname{Fun}(\Delta^n,\mathcal{S})$ for all $n$.  For $\Leftarrow$, suppose we have a limit diagram on functors out to $\mathcal{S}$.  We want to show that the map $F:\varinjlim d|_K\rightarrow d(\infty)$ is an equivalence on idempotent completion.  By the above, for this it suffices to show that $F$ induces an isomorphism in $\operatorname{Pr}^L$ on applying $\mathcal{P}(-)$.  For that, by Yoneda, it suffices to show that $\operatorname{Fun}^L(-,\mathcal{P}(\mathcal{D}))$ applied to $\mathcal{P}(F)$ gives an isomorphism for all $\mathcal{D}\in\operatorname{Cat}_\infty$.  But $\operatorname{Fun}^L(-,\mathcal{P}(\mathcal{D}))=\operatorname{Fun}(\mathcal{D},\operatorname{Fun}^L(-,\mathcal{S}))$, so this follows from the assumption and the universal property of $\mathcal{P}(-)$.

\medskip

Now for 2, first suppose $\varinjlim d\vert_K\overset{\sim}{\rightarrow} d(\infty)$.  Then certainly we also have an equivalence on idempotent completion.  Let $\mathcal{C}\subset d(\infty)$ denote the union of the essential images of the $d(k)\rightarrow d(\infty)$.  Then by the universal property of colimits we deduce that this inclusion $\mathcal{C}\subset d(\infty)$ has a section, whence it's an equality, as desired.  Now suppose we have an equivalence on idempotent completion.  Since every $\infty$-category embeds fully faithfully in its idempotent completion, it follows that $\varinjlim d\vert_K\rightarrow d(\infty)$ is fully faithful.  But the other condition gives essential surjectivity, whence the conclusion.
\end{proof}

\subsection{Inverting all arrows}\label{realise}

In the following sections we will need to use several different ``formulas'' for the functor left adjoint to the inclusion $\mathcal{S}\rightarrow\operatorname{Cat}_\infty$.  The purpose of this section is to collect them.

\begin{theorem}
For a functor $F\colon\operatorname{Cat}_\infty\rightarrow\mathcal{S}$, the following properties are equivalent:
\begin{enumerate}
\item $F$ is left adjoint to the inclusion $\mathcal{S}\subset\operatorname{Cat}_\infty$.
\item $F$ preserves all colimits, $F(\ast)= \ast$, and $F(\Delta^1)=\ast$.
\item $F$ preserves all colimits, and $F(\Delta^n)=\ast$ for all $n$.
\end{enumerate}
Moreover, the $\infty$-category of all such functors is equivalent to the terminal $\infty$-category $\ast$. (In particular, the implicit data of the adjunction in 1 is unique.)
\end{theorem}
\begin{proof}
This is a simple consequence of the complete Segal space presentation of $\operatorname{Cat}_\infty$ (\cite{rezk}, and see \cite{L2} for a natively $\infty$-categorical account).  First, from that presentation (more specifically, from the fact that it realises $\operatorname{Cat}_\infty$ as a localisation of $\mathcal{P}(\Delta)$), one sees that any colimit-preserving functor out of $\operatorname{Cat}_\infty$ is the left Kan extension of its restriction to $\Delta$.  Since the $\infty$-category of terminal functors from any $\infty$-category to $\mathcal{S}$ is always $\ast$, this shows the last claim holds if we take equivalent condition 3.  On the other hand the complete Segal space presentation (more specifically, the Segal condition) also shows that $\Delta^n$ is the colimit of $n$ copies of $\Delta^1$ placed end-to-end, which implies that 2 $\Leftrightarrow$ 3.  Note that the functor in 1 is uniquely characterised up to equivalence, and so is the functor in 3, again by the complete Segal space presentation.  Thus, to see that 1 is equivalent to 2 and 3, we just need to see that the left adjoint $F$ to the inclusion indeed satisfies $F(\ast)=\ast$ and $F(\Delta^1)=\ast$.  The first claim is tautological as $\ast\in\mathcal{S}$.  For the second claim, it exactly corresponds to the completeness criterion in complete Segal spaces. \end{proof}

\medskip

From now on we will write $\mathcal{C}\mapsto \vert \mathcal{C}\vert$ for the functor $F$ characterised by the previous theorem.  For a given $\infty$-category $\mathcal{C}$, to verify a proposed description of $\vert \mathcal{C}\vert$, one has, generally speaking, two options: either make that description functorial in $\mathcal{C}$ and verify condition 2 of the theorem, or else produce a comparison map $\mathcal{C}\rightarrow\vert \mathcal{C}\vert$ and argue that it satisfies the universal property implicit in condition 1, namely that maps to $\infty$-groupoids from $\vert \mathcal{C}\vert$ are the same as from $\mathcal{C}$.

\begin{corollary}\label{describeinvert}
For each $\infty$-category $\mathcal{C}$, the following are descriptions of the $\infty$-groupoid $\vert \mathcal{C}\vert$:
\begin{enumerate}
\item $\vert \mathcal{C}\vert =\mathcal{C}[\operatorname{Ar}\mathcal{C}^{-1}]$, the $\infty$-category obtained by inverting all arrows.
\item $\vert \mathcal{C}\vert = \varinjlim_{[n]\in \Delta, [n]\rightarrow \mathcal{C}} \ast$.
\item $\vert \mathcal{C}\vert$ is the colimit of the simplicial anima $\Delta^n \mapsto \operatorname{Map}_{\operatorname{Cat}_\infty}(\Delta^n,\mathcal{C})$ (the complete Segal anima associated to $\mathcal{C}$).
\item $\vert \mathcal{C}\vert = \varinjlim_\mathcal{C}\ast$.
\item $\vert\mathcal{C}\vert = \mathcal{P}^{lc}(\mathcal{C})^{atom}$, where $\mathcal{P}^{lc}(\mathcal{C})\subset\mathcal{P}(\mathcal{C})$ is the full subcategory on those $\mathcal{C}^{op}\rightarrow\mathcal{S}$ which are constant on every simplex, or equivalently send every morphism to an isomorphism.
\item If a simplicial set $X\in \operatorname{sSet}$ localises to $\mathcal{C}$ in the Joyal presentation $\operatorname{Cat}_\infty\simeq \operatorname{sSet}[ce^{-1}]$, then $X$ localises to $\vert \mathcal{C}\vert$ in the Kan presentation $\mathcal{S}\simeq \operatorname{sSet}[we^{-1}]$.
\end{enumerate}
\end{corollary}
\begin{proof}
1 is tautologically a description of the left adjoint to the inclusion $\mathcal{S}\subset\operatorname{Cat}_\infty$.  For 2, by the objectwise formula for left Kan extensions this is equivalent to saying that $\vert \cdot\vert$ is the left Kan extension of the terminal functor $\Delta\rightarrow \mathcal{S}$.  But it follows again from the complete Segal space picture that every colimit-preserving functor out of $\operatorname{Cat}_\infty$ is left Kan extended from $\Delta$, so this is a rephrasing of condition 3 in the above theorem.  For 3, the colimit of the simplicial space is by definition the left adjoint to the inclusion of constant simplicial spaces into all simplicial spaces, and this restricts to the desired adjunction on complete Segal spaces.  For 4, note that the colimit in question is by definition determined by
$$\operatorname{Map}(\varinjlim_\mathcal{C}\ast,X) = \operatorname{Map}_{\operatorname{Fun}(\mathcal{C},\mathcal{S})}(\ast, X),$$
the mapping space between the constant functor on $\ast$ and the constant functor on $X$.  As $\mathcal{C}\rightarrow\vert\mathcal{C}\vert$ is a localisation, the pullback map on functors to $\mathcal{S}$ is fully faithful, so we deduce $\varinjlim_\mathcal{C}\ast\overset{\sim}{\rightarrow}\varinjlim_{\vert\mathcal{C}\vert}\ast$.  So it suffices to assume $\mathcal{C}$ is an $\infty$-groupoid.  But then there is an equivalence $\operatorname{Fun}(\mathcal{C},\mathcal{S})\simeq \mathcal{S}_{/\mathcal{C}}$ given by pulling back along the forgetful functor $\mathcal{S}_\ast\rightarrow\mathcal{S}$, see \cite{L} 3.3.2.7, and in these terms we see that $\operatorname{Map}_{\operatorname{Fun}(\mathcal{C},\mathcal{S})}(\ast, X)$ identifies with the space of sections of the projection $\mathcal{C}\times X\rightarrow\mathcal{C}$.  But this is just $\operatorname{Map}(\mathcal{C},X)$, as desired.  For 5, we can calculate $\mathcal{P}(\vert \mathcal{C}\vert)$ using the universal property of $\vert \cdot \vert$ to see that 5 holds.  Finally, 6 follows from the fact that the identity functor exhibits $\operatorname{sSet}$ with the Kan model structure as a Bousfield localisation of $\operatorname{sSet}$ with the Joyal model structure, which verifies criterion 1 of the theorem.
\end{proof}

We also can generate more descriptions by applying part 1 of the following:

\begin{corollary}\label{propertiesinvert} \ 
\begin{enumerate}
\item There is a unique functorial equivalence $\vert \mathcal{C}\vert \simeq \vert \mathcal{C}^{op}\vert$.
\item $\vert \cdot \vert$ preserves finite products.
\item If two functors $\mathcal{C}\rightarrow\mathcal{D}$ are related by a natural transformation, they induce homotopic maps $\vert\mathcal{C}\vert\rightarrow\vert\mathcal{D}\vert$.
\item If a functor admits an adjoint in either direction, it induces an equivalence on $\vert \cdot \vert$.
\end{enumerate}
\end{corollary}
\begin{proof}
Claim 1 follows from the theorem because the terminal functor $\Delta\rightarrow\mathcal{S}$ obviously has the required invariance property.  Claim 2 follows from description 3 of the above corollary, since $\Delta^{op}$ is sifted, \cite{L} 5.5.8.4.  Claim 3 follows from claim 2 and $\vert \Delta^1\vert=\ast$.  Claim 4 follows from claim 3.
\end{proof}

\subsection{Topological analogue: proper maps, proper base change, and proper descent}

In the next section we will discuss the notion of proper functors between $\infty$-categories, and the associated proper base-change and proper descent theorems.   But for motivation, and because we will later use it, we start by recalling the more familiar topological analogue.  A map of locally compact Hausdorff topological spaces $f\colon X\rightarrow Y$ is called \emph{proper} if the preimage of every compact subset is compact.  A crucial fact about proper maps is the \emph{tube lemma}: if $y\in Y$, then the $f^{-1}(U)$ for $U$ an open neighbourhood of $y$ form a cofinal system of open neighbourhoods of the fibre $X_y$.

\medskip

We start by recalling the version of the proper base-change theorem proved by Lurie in \cite{L} 7.3.

\begin{theorem}
Let
$$\xymatrix{
X'\ar[r]^{g'}\ar[d]^{f'} & X\ar[d]^f \\
Y'\ar[r]^{g} & Y \\
}$$
be a pullback diagram of locally compact Hausdorff spaces with $f$ proper.  Then the induced commutative diagram of $\infty$-categories gotten by applying $\operatorname{Sh}(-)$
$$\xymatrix{
\operatorname{Sh}(X')& \operatorname{Sh}(X)\ar[l]^{g'^\ast} \\
\operatorname{Sh}(Y')\ar[u]^{f'^\ast} & \operatorname{Sh}(Y)\ar[l]^{g^\ast}\ar[u]^{f^\ast} \\
}$$

\noindent is \emph{right adjointable} (or \emph{right Beck--Chevalley}): the vertical maps $f^\ast$ and $f'^\ast$ have right adjoints $f_\ast$ and $f'_\ast$ respectively, and the natural comparison map is an equivalence
$$g^\ast f_\ast \overset{\sim}{\rightarrow} f'_\ast g'^\ast.$$
\end{theorem}

As observed by Deligne in \cite{D}, this kind of proper base-change can be used to give ``proper descent'' results.  We start with ``cdh descent'':

\begin{corollary}\label{cdhdescenttop}
Suppose given a pullback square of locally compact Hausdorff spaces
$$\xymatrix{
X'\ar[r]^{g'}\ar[d]^{f'} & X\ar[d]^f \\
Y'\ar[r]^{g} & Y \\
}$$
such that:
\begin{enumerate}
\item $f$ is proper;
\item $g$ is the inclusion of a closed subspace;
\item the pullback of $f$ to the open complement $Y\smallsetminus Y'$ is an isomorphism.
\end{enumerate}

Then applying $\operatorname{Sh}(-)$ with pullback functoriality gives a pullback diagram, so
$$\operatorname{Sh}(Y)\overset{\sim}{\rightarrow}\operatorname{Sh}(X)\times_{\operatorname{Sh}(X')}\operatorname{Sh}(Y').$$
\end{corollary}
\begin{proof}
Recall that if we have a closed subset $T'\subset T$ of a topological space $T$, then equivalences of sheaves on $T$ can be detected by pullback to $T'$ and $T\smallsetminus T'$, \cite{L} 7.3.2.  Furthermore, pullback functors on $\infty$-categories of sheaves associated to maps of topological spaces preserve finite limits, because they correspond to geometric morphisms of $\infty$-topoi.  It then follows from \cite{LHA} 5.2.2.37  and the proper base-change theorem that we can test the conclusion of this corollary after pulling back to $Y'$ and $Y\smallsetminus Y'$ (compare with the proof of \Cref{properdescentcat}).  But on pullback to $Y'$ the horizontal maps become equivalences and on pullback to $Y\smallsetminus Y'$ the vertical maps become equivalences, and in either case the conclusion is tautological.
\end{proof}

\begin{corollary}\label{closeddescenttop}
Let $T$ be a locally compact Hausdorff space, and let $P$ be a finite set of closed subsets of $T$.  Suppose that for all $P'\subset P$ the intersection $\cap_{S\in P'}S$ admits a cover by elements of $P$.  (In particular, taking $P'=\emptyset$, we deduce $\cup_{S\in P}S=T$.) Then
$$\operatorname{Sh}(T)\overset{\sim}{\rightarrow} \varprojlim_{S\in P^{op}}\operatorname{Sh}(S)$$
via pullback, viewing $P$ as a poset under inclusion.
\end{corollary}
\begin{proof}
When $P$ has $\leq 3$ elements this reduces to the special case of \Cref{cdhdescenttop} in which the proper map $f$ is also a closed inclusion.  The general case follows by induction.
\end{proof}

\begin{remark}
The locally compact Hausdorff hypothesis is unnecessary here.  Indeed, the proper base-change theorem holds for general topological spaces when the proper map $f$ is a closed inclusion, see \cite{L} 7.3.2.\exend
\end{remark}

If one uses open covers instead of closed covers, the finiteness requirements can be removed. 

\begin{theorem}\label{Bigopendescent}
Let $X$ be a topological space and $\{X_i\rightarrow X\}_{i\in I}$ a set of maps to $X$ such that for all $x\in X$, there is an open $U\subset X$ containing $x$ and an $i\in I$ such that the pullback of $X_i\rightarrow X$ to $U$ has a section.  Let $\mathcal{U}\subset\operatorname{Top}_{/X}$ denote the sieve generated by the $X_i$, so $Y\rightarrow X$ lies in $\mathcal{U}$ if and only if it factors through some $X_i$.  Then
$$\operatorname{Sh}(X)\overset{\sim}{\rightarrow} \varprojlim_{(Y\rightarrow X) \in \mathcal{U}} \operatorname{Sh}(Y)$$
via the pullback functors.
\end{theorem}
\begin{proof}
Let us define a covering family in $\operatorname{Top}$ to be a family of open inclusions $(U_i\rightarrow X)_{i\in I}$ whose images cover $X$.  Then the axioms of a pretopology are clearly satisfied, so we get a Grothendieck topology on $\operatorname{Top}$ for which the covering sieves over $X$ are those sieves which contain some open cover of $X$.  Our sieve $\mathcal{U}$ is clearly such a sieve, so it suffices to show that $X\mapsto \operatorname{Sh}(X)$ satisfies descent for this Grothendieck topology.
However, because the open subsets of $X$ are closed under finite intersection we see that the sieve generated by an open cover in $\operatorname{Open}(X)$ is cofinal in the sieve generated by that open cover in $\operatorname{Top}_{/X}$, meaning our desired descent is equivalent to saying that $U\mapsto \operatorname{Sh}(U)$ is a sheaf of $\infty$-categories on $X$.  But this is a general property of sheaf categories, \cite{L} 6.1.3.\end{proof}

\begin{corollary}\label{opendescent}
Let $X$ be a topological space and $P$ a collection of open subsets of $X$ with $U,V\in P \Rightarrow U\cap V\in P$ and $\cup_{U\in P}U=X$.  Then
$$\operatorname{Sh}(X)\overset{\sim}{\rightarrow} \varprojlim_{U\in P^{op}}\operatorname{Sh}(U)$$
via pullback, where we view $P$ as a poset under inclusion.
\end{corollary}
\begin{proof}
Let $\mathcal{U}$ denote the sieve of those open subsets of $X$ contained in some $U\in P$.  This is a covering sieve by the second condition, so it suffices to show that the inclusion $P\subset \mathcal{U}$ is cofinal, or a $\varinjlim$-equivalence in the language we will use in this paper (see the following section, \Cref{lim-equivalences}).  The right fibre over an element $V\in\mathcal{U}$ is the poset of those $U\in P$ containing $V$.  This is nonempty by definition and is closed under intersection as $P$ is by construction, therefore it is filtered and hence contractible.
\end{proof}

Recall that if $\Gamma$ is a discrete group, then a $\Gamma$-action on an object of a category (or $\infty$-category) $\mathcal{C}$ is a functor $B\Gamma\rightarrow \mathcal{C}$; the underlying object $X\in\mathcal{C}$ is the image of the unique object of $B\Gamma$.  The $\Gamma$-fixed point (or homotopy fixed point) object, if it exists, is the limit over this $B\Gamma$-diagram in $\mathcal{C}$ and is abusively denoted $X^\Gamma$.

\begin{corollary}\label{quotienttop}
Let $X$ be a topological space with a free proper left action of a discrete group $\Gamma$, meaning for all $x\in X$ there is an open neighbourhood $U$ of $x$ such that all the $\gamma\cdot U$ are disjoint, $\gamma\in \Gamma$.  Then
$$\operatorname{Sh}(\Gamma\backslash X)\overset{\sim}{\rightarrow} \operatorname{Sh}(X)^\Gamma$$
via pullback.\end{corollary}
\begin{proof}
By the general descent result \Cref{Bigopendescent}, it suffices to show that if $h_X,h_Y$ denote the sheaves on $\operatorname{Top}$ represented by $X$ and $Y$ respectively, then $h_X/\Gamma\overset{\sim}{\rightarrow} h_Y$, as sheaves on $\operatorname{Top}$ (with the open cover topology).  But the assumption implies that $X\rightarrow Y$ has local sections, hence $h_X\rightarrow h_Y$ is a cover, and moreover $h_X\times_{h_Y}h_X = h_{X\times_YX}$ identifies with $\Gamma\times h_X$, with the two projection maps to $h_X$ identifying with the action and projection maps.  The claim follows.
\end{proof}

\subsection{Proper functors, proper base change, and proper descent}

The basic concepts in this section were picked up from reading Grothendieck's \emph{Pursuing stacks}.  A different kind of treatment of several of the same results and concepts, making explicit use of the quasi-category model for $\infty$-categories, is given in \cite{Cis}. 

\medskip

 We recall the following theorem/definition, Joyal's $\infty$-categorical generalisation of Quillen's theorem A; see \cite{L} 4.1.3.  Although Joyal (and later Lurie) prove this theorem in the quasi-category model using combinatorial arguments, if we take the $\infty$-categorical Yoneda lemma and related results for granted, we can give a quick non-combinatorial proof.

\begin{theorem}\label{lim-equivalences}
Let $f\colon\mathcal{C}\rightarrow\mathcal{D}$ be a functor of small $\infty$-categories.  The following properties are equivalent:
\begin{enumerate}
\item For any functor $X\colon\mathcal{D}\rightarrow \mathcal{E}$ to an $\infty$-category $\mathcal{E}$, the comparison map of limits
 $$\varprojlim X \rightarrow \varprojlim (f^\ast X)$$
 is an equivalence (in the sense that if one limit exists so does the other and the map is an equivalence).
\item Same condition, but just with $\mathcal{E}=\mathcal{S}$.
 \item For any $d\in \mathcal{D}$, the left fibre $\mathcal{C}_{/d}$ is contractible in the sense that $\vert\mathcal{C}_{/d}\vert\simeq \ast$.
\end{enumerate}
If these properties are satisfied, we say that $f$ is a \emph{$\varprojlim$-equivalence}.  If the dual properties are satisfied, meaning if the above conditions are satisfied for $f^{op}\colon\mathcal{C}^{op}\rightarrow\mathcal{D}^{op}$, we say $f$ is a \emph{$\varinjlim$-equivalence}.  (The usual terminology for $\varinjlim$-equivalence is ``cofinal functor'', see \cite{L} 4.1.)
\end{theorem}

Here the left fibre $\mathcal{C}_{/d}$ stands for the $\infty$-category given as the left pullback $\mathcal{C}\overset{\rightarrow}{\times_\mathcal{D}}\{d\}$ as defined in \cite{Tamme}.  Informally, an object of $\mathcal{C}_{/d}$ is an object $c\in\mathcal{C}$ together with a map $f(c)\rightarrow d$.

\begin{proof}
In condition 1 we may as well assume $\mathcal{E}=\mathcal{S}$, because limits in $\mathcal{E}$ can be tested on applying $\operatorname{Map}(e,-)$ for all $e\in\mathcal{E}$.  So 1 and 2 are equivalent.

\medskip

Note that the pullback functor $f^\ast\colon\operatorname{Fun}(\mathcal{D},\mathcal{S})\rightarrow\operatorname{Fun}(\mathcal{C},\mathcal{S})$ has a left adjoint $f_!$ given by Kan extension, \cite{L} 4.3.  By adjunction, 2 holds if and only if the (unique) map
$$f_!(\ast)\rightarrow \ast$$
in $\operatorname{Fun}(\mathcal{D},\mathcal{S})$ is an equivalence.  By the objectwise formula for left Kan extensions, this amounts to the assertion that for all $d\in\mathcal{D}$ the map
$$\varinjlim_{(\mathcal{C}_{/d})^{op}}\ast \rightarrow \ast$$
is an equivalence.  By Corollaries \ref{describeinvert} and \ref{propertiesinvert}, this is equivalent to condition 3.
\end{proof}

\begin{remark}
If $f\colon\mathcal{C}\rightarrow\mathcal{D}$ is a $\varprojlim$-equivalence, then it induces an equivalence on $\vert\cdot\vert$.  Indeed, we need to see that if $K\in\mathcal{S}$, then $\operatorname{Map}(\vert\mathcal{D}\vert,K)\overset{\sim}{\rightarrow}\operatorname{Map}(\vert\mathcal{C}\vert,K)$; but this is the special case of 2 where $X$ is the constant functor with value $K$.  This is why Joyal's theorem A is a generalisation of Quillen's.  (Recall Quillen's says that condition 3 implies that $f$ induces an equivalence on geometric realisation.)\exend
\end{remark}

\begin{example}\label{finalexample} \ 
\begin{enumerate}
\item Any left adjoint functor is a $\varprojlim$-equivalence.  Indeed, being a left adjoint is equivalent to each left fibre admitting a terminal object.
\item If $f$ is a localisation, i.e.\ if it is of the form $\mathcal{C}\rightarrow\mathcal{C}[S^{-1}]$ for some collection of arrows $S$ in $\mathcal{C}$, then $f$ is a $\varprojlim$-equivalence.  Indeed, in this case the map on functors out to $\mathcal{S}$ is fully faithful, so the comparison map in 2 is an equivalence.\exend
\end{enumerate}
\end{example}

\begin{definition}
Let $f\colon\mathcal{C}\rightarrow\mathcal{D}$ be a functor of small $\infty$-categories.  We say that $f$ is \emph{proper} if for every $d\in \mathcal{D}$, the inclusion $\mathcal{C}_d\rightarrow \mathcal{C}_{d/}$ of the fibre into the right fibre is a $\varprojlim$-equivalence.\defend
\end{definition}

Here the fibre $\mathcal{C}_d$ means the pullback of $\mathcal{C}\overset{f}{\rightarrow}\mathcal{D}\leftarrow\{d\}$ in the $\infty$-category of $\infty$-categories; it is the $\infty$-category of objects $c\in C$ together with an equivalence $d\simeq f(c)$.  For the right fibre, we have an arbitrary map $d\rightarrow f(c)$ instead of an equivalence.  This definition is some sort of analogue of the tube lemma for proper maps in the topological context.

\begin{remark}
There is a rather picturesque consequence of this definition which won't play an explicit role for us, but can be useful to keep in mind as intuition for proper functors.  Suppose given a map $d\rightarrow d'$ in $\mathcal{D}$.  Then there is the obvious funtor $\mathcal{C}_{d'}\rightarrow \mathcal{C}_{d/}$, giving the diagram
$$\mathcal{C}_{d'}\rightarrow \mathcal{C}_{d/}\leftarrow\mathcal{C}_d.$$
When $f$ is proper, the right hand map is a $\varprojlim$-equivalence, in particular an isomorphism on $\vert\cdot \vert$.  Composing with the inverse of this isomorphism, we get a natural map of anima
$$\vert \mathcal{C}_{d'}\vert \rightarrow\vert \mathcal{C}_d\vert.$$
In fact, a proper functor $f\colon\mathcal{C}\rightarrow\mathcal{D}$ gives rise to a canonical functor $\mathcal{D}^{op}\rightarrow \mathcal{S}$ whose value on $d$ is $\vert \mathcal{C}_d\vert$.

\medskip

In other words, for a proper functor, the homotopy types of the fibres are contravariantly functorial in the point the fibre is taken over.  There is also a similar phenomenon in the topological context, at least under certain regularity hypotheses: for a proper map, the homotopy types of the fibres form a constructible co-sheaf on the base.\exend
\end{remark}

We will soon show that the class of proper maps is closed under composition and base change.  Here are also some general examples.

\begin{example}\label{properexamples} \ 
\begin{enumerate}
\item Let $\mathcal{C}\overset{f}{\rightarrow} \mathcal{D}\overset{g}{\leftarrow} \mathcal{B}$ be arbitrary functors of $\infty$-categories as indicated.  Then the projection
$$\mathcal{C}\overset{\rightarrow}{\times}_{\mathcal{D}}\mathcal{B}\rightarrow \mathcal{C}$$
from the left pullback ($\infty$-category of tuples $(c\in\mathcal{C},b\in\mathcal{B},f(c)\rightarrow g(b))$, see \cite{Tamme}) is proper.  Indeed, for fixed $c\in\mathcal{C}$, the fibre is the $\infty$-category of $(b\in\mathcal{B},f(c)\rightarrow g(b))$ whereas the right fibre is the $\infty$-category of $(x\in\mathcal{C},b\in\mathcal{B}, c\rightarrow x, f(x)\rightarrow g(b))$.  The inclusion of the former into the latter has right adjoint given by sending this latter data to the object in the fibre given by $b\in\mathcal{B}$ and the composite $f(c)\rightarrow f(x)\rightarrow g(b)$.
\item For any $c\in\mathcal{C}$, the projection $\mathcal{C}_{/c}\rightarrow \mathcal{C}$ is proper.  This is a special case of 1.
\item Suppose $f\colon\mathcal{C}\subseteq \mathcal{D}$ is the inclusion of a full subcategory closed under isomorphisms.  Then $f$ is proper if and only if $\mathcal{C}$ is \emph{left closed}: $x\rightarrow y$ and $y\in \mathcal{C}$ implies $x\in\mathcal{C}$. (In site-theoretic terminology, this means $\mathcal{C}$ is a \emph{sieve} in $\mathcal{D}$.)  Indeed, if $y\in\mathcal{D}$ lies in $\mathcal{C}$ then the condition that $\mathcal{C}_y\rightarrow \mathcal{C}_{y/}$ be a $\varprojlim$-equivalence is automatic as it identifies with the inclusion of an initial object, whereas when $y\not\in\mathcal{C}$ we exactly need that $\mathcal{C}_{y/}$ be empty.
\item If every morphism in $\mathcal{D}$ is invertible then any functor $\mathcal{C}\rightarrow\mathcal{D}$ is proper, as the fibre identifies with the right fibre.
\item Encompassing all the above examples, any \emph{locally cartesian fibration} is proper.  Indeed, by \cite{AF} Lemma 2.20 the locally cartesian fibrations are characterised up to equivalence by $\mathcal{C}_d\rightarrow \mathcal{C}_{d/}$ being a left adjoint for all $d\in\mathcal{D}$.\exend
\end{enumerate}
\end{example}

\begin{proposition}\label{characteriseproper}
Let $f\colon\mathcal{C}\rightarrow\mathcal{D}$ be a functor.  The following are equivalent:
\begin{enumerate}
\item  $f$ is proper.
\item  For any functor $\Delta^1\rightarrow\mathcal{D}$, the pullback $\mathcal{C}'\rightarrow \Delta^1$ of $f$ satisfies the condition that the inclusion $\mathcal{C}'_0\rightarrow \mathcal{C}'$ of the fibre above $0$ is a $\varprojlim$-equivalence.
\end{enumerate}
\end{proposition}
\begin{proof}
First, we remark that, in the situation of a functor $\mathcal{C}'\rightarrow\Delta^1$, to test whether the inclusion $\mathcal{C}'_0\rightarrow \mathcal{C}'$ is a $\varprojlim$-equivalence, it suffices to prove that the left fibre above any object of $\mathcal{C}'_1$ is contractible.  Indeed, the left fibre over an object of $\mathcal{C}'_0$ has a terminal object, hence will automatically be contractible.

\medskip

By definition, 1 holds if and only if for any $d\in\mathcal{D}$ and any $x=(c,d\rightarrow f(c))\in \mathcal{C}_{d/}$, the left fibre of $\mathcal{C}_d\rightarrow\mathcal{C}_{d/}$ above $x$ is contractible.  On the other hand, consider an arbitrary functor $\Delta^1\rightarrow \mathcal{D}$ classifying a map $d_0\rightarrow d_1$.  Then 2 holds if and only if the left fibre of $\mathcal{C}_{d_0}$ including into $\Delta^1\times_\mathcal{D}\mathcal{C}$, taken at some $c$ lying above $d_1$, is contractible.  However, the data of $d_0\rightarrow d_1$ and $c$ is the same as the data of $x$, and the corresponding left fibres are equivalent.  Thus the conditions in the proposition are equivalent.
\end{proof}

\begin{corollary}\label{properclosedunderpullback}
The class of proper functors between small $\infty$-categories is closed under pullbacks.
\end{corollary}
\begin{proof}
The other equivalent condition from the proposition manifestly satisfies this closure property.
\end{proof}

The following is the proper base change theorem in this context.

\begin{theorem}
Let
$$\xymatrix{
\mathcal{C}'\ar[r]^{g'}\ar[d]^{f'} & \mathcal{C}\ar[d]^f \\
\mathcal{D}'\ar[r]^{g} & \mathcal{D} \\
}$$
be a pullback diagram of small $\infty$-categories, and let $\mathcal{E}$ be an $\infty$-category with all limits.

\medskip

Suppose $f$ is proper.  Then the induced commutative diagram gotten from applying $\operatorname{Fun}(-,\mathcal{E})$

$$\xymatrix{
\operatorname{Fun}(\mathcal{C}',\mathcal{E}) & \operatorname{Fun}(\mathcal{C},\mathcal{E})\ar[l]^{g'^\ast} \\
\operatorname{Fun}(\mathcal{D}',\mathcal{E})\ar[u]^{f'^\ast} & \operatorname{Fun}(\mathcal{D},\mathcal{E})\ar[l]^{g^\ast}\ar[u]^{f^\ast} \\
}$$

\noindent is \emph{right adjointable} (or \emph{right Beck--Chevalley}): the vertical maps $f^\ast$ and $f'^\ast$ have right adjoints $f_\ast$ and $f'_\ast$ respectively, and the natural comparison map
$$g^\ast f_\ast \rightarrow f'_\ast g'^\ast$$
is an equivalence.

\medskip

Conversely, suppose that the functor $f$, as well as all its pullbacks, satisfies the condition that the commutative diagram gotten by applying $\operatorname{Fun}(-,\mathcal{E})$ is right adjointable, even just in the special case $\mathcal{E}=\mathcal{S}$.  Then $f$ is proper.
\end{theorem}
\begin{proof}
The fact that $f^\ast$ and $f'^\ast$ admit right adjoints is purely formal and does not require the properness.  Indeed, the right adjoints are given by right Kan extension.  Now assume $f$ proper and choose $F\in\operatorname{Fun}(\mathcal{C},\mathcal{E})$; we want to see that
$$g^\ast f_\ast F\rightarrow f'_\ast g'^\ast F$$
is an equivalence in $\operatorname{Fun}(\mathcal{D'},\mathcal{E})$.

\medskip

First assume that $\mathcal{D}'$ is the terminal category $\ast$, so that the functor $g$ classifies an object $d\in\mathcal{D}$. The objectwise description of the right Kan extension shows that the value $g^\ast f_\ast F \in \mathcal{E}$ identifies with the limit of $F$ over the right fibre $\mathcal{C}_{d/}$.  Meanwhile the value $f'_\ast g'^\ast F$ identifies with the limit of $F$ over the fibre $\mathcal{C}_d$.  By definition of properness this comparison map is an equivalence.  This handles the case $\mathcal{D}'=\ast$.

\medskip

To deduce the general case, note that a map in $\operatorname{Fun}(\mathcal{D'},\mathcal{E})$ is an equivalence if and only if it is so after evaluating on any object, or in other words after pulling back along any functor from $\ast$.  As the pullback of a proper map is proper and Beck--Chevalley comparison maps compose, this reduces us to the case of a point.

\medskip

For the converse, suppose all base-changes of $f$ satisfy the proper base change theorem for $\mathcal{E}=\mathcal{S}$.  Consider the special case of the base-change by a map $\Delta^1\rightarrow\mathcal{D}$, and then apply the proper base change theorem to the pullback of that base-changed map $\mathcal{C}'\rightarrow\Delta^1$ along $0\rightarrow \Delta^1$.  For a functor $\Delta^1\rightarrow \mathcal{S}$ its limit is the same as its evaluation at the initial object $0$, so we see exactly the equivalent condition for properness enunciated in \Cref{characteriseproper}. 
\end{proof}

\begin{corollary}
The composition of proper functors is proper. The class of proper functors is closed under colimits in $\operatorname{Fun}(\Delta^1,\operatorname{Cat}_\infty)$.
\end{corollary}
\begin{proof}
This follows from the above converse to the proper base change theorem, as right adjointability composes and is preserved by limits \cite{LHA} 4.7.4.18.
\end{proof}

Now we turn to proper descent, or how to identify colimits of $\infty$-categories along proper maps.  This is our main purpose for discussing proper functors.

\begin{theorem}\label{properdescentcat}
Let $K$ be a small $\infty$-category and $d$ a functor $K^\triangleright\rightarrow\operatorname{Cat}_\infty$, viewed as a $K$-shaped diagram of small $\infty$-categories together with a co-cone for this diagram.  Suppose that:
\begin{enumerate}
\item The functor $d(f)$ is proper for all maps $f$ in $K^\triangleright$.
\item For all functors $f\colon\ast\rightarrow d(\infty)$ from the terminal category to the co-cone point of $f$, the pullback $f^{-1}d$ is a colimit diagram.  (Here $f^{-1}d$ is the functor $K^\triangleright\rightarrow\operatorname{Cat}_\infty$ defined by $(f^{-1}d)(k) = \ast\times_{d(\infty)}d(k)$.)
\end{enumerate}
Then $d$ is a colimit diagram.
\end{theorem}
\begin{proof}
First, we note that the collection of functors $d(k)\rightarrow d(\infty)$, ranging over all $k\in K$, is jointly essentially surjective.  Indeed, if an object were not in the joint essential image, the pullback of $d$ along the functor $\ast\rightarrow d(\infty)$ classifying that object would have empty restriction to $K$, whence empty colimit, contradicting the assumption.  Thus by \Cref{detectcolim} it suffices to see that applying $\operatorname{Fun}(-,\mathcal{S})$ to our diagram $d$ yields a limit diagram of $\infty$-categories, assuming the same for every pullback $f^{-1}d$ along a functor $f\colon\ast\rightarrow d(\infty)$.

\medskip

Consider $f\colon\sqcup_I \ast\rightarrow d(\infty)$, a disjoint union of terminal categories indexed by the isomorphism classes of objects in $d(\infty)$, mapping to $d(\infty)$ by selecting an object in each isomorphism class.  Consider the induced natural transformation
$$\operatorname{Fun}(d(-),\mathcal{S})\rightarrow \operatorname{Fun}((f^{-1}d)(-),\mathcal{S})$$
of diagrams $(K^{op})^\triangleleft\rightarrow\operatorname{CAT}_\infty$.  We want to see that the source is a limit diagram.  Using the criterion of \cite{LHA} 5.2.2.37, it suffices to check the following four conditions:
\begin{enumerate}
\item The target is a limit diagram.  This holds because it is a product of limit diagrams by assumption.
\item For each $k\in (K^{op})^\triangleleft$, the induced functor $\operatorname{Fun}(d(k),\mathcal{S})\rightarrow\operatorname{Fun}((f^{-1}d)(k),\mathcal{S})$ is conservative.  This holds because the functor $(f^{-1}d)(k)\rightarrow d(k)$ is a pullback of the essentially surjective functor $f$ hence is itself essentially surjective, and equivalences in presheaf categories are detected objectwise.
\item The $\infty$-category $\operatorname{Fun}(d(\infty),\mathcal{S})$ admits $K$-indexed limits, and these are preserved by
$$\operatorname{Fun}(d(\infty),\mathcal{S})\rightarrow \operatorname{Fun}((f^{-1}d)(\infty),\mathcal{S}).$$
In fact functor categories to $\mathcal{S}$ admit all limits and these are preserved by all pullbacks, since limits are calculated objectwise in functor categories.
\item For every morphism $\alpha$ in $(K^{op})^\triangleleft$, the commutative square of $\infty$-categories gotten by applying our natural transformation to $\alpha$ is right adjointable.  This holds by the proper base change theorem and our assumption that $d(\alpha)$ is proper.
\end{enumerate}
Thus the conditions of \cite{LHA} 5.2.2.37 apply and finish the proof.
\end{proof}

\begin{corollary}
Every colimit in $\operatorname{Cat}_\infty$ produced by the above theorem is \emph{universal}: stable under pullback (via an arbitrary map to the co-cone object).
\end{corollary}
\begin{proof}
Clear, since the two conditions are stable under pullback.
\end{proof}

Here are some special cases.  First we have Cech descent along a covering map.

\begin{corollary}
Let $f\colon\mathcal{C}\rightarrow\mathcal{D}$ be an essentially surjective proper functor in $\operatorname{Cat}_\infty$.  Then $\mathcal{D}$ identifies with the colimit of the Cech nerve of $f$.
\end{corollary}
\begin{proof}
Recall that the nondegenerate simplex category is cofinal in $\Delta$, \cite{L} 6.5.3.7, so in calculating the colimit of the Cech nerve we can restrict to the functors induced by non-degenerate maps in $\Delta$.  But all such functors are pullbacks of $f$, hence are also proper by \Cref{properclosedunderpullback}.  Hence by \Cref{properdescentcat} we can reduce to the case $\mathcal{D}=\ast$.  But then $f$ admits a section and hence gives a colimit diagram, \cite{L} 6.1.3.16.
\end{proof}

Here is the analogue of ``cdh descent'' in algebraic geometry.

\begin{corollary}\label{cdhdescentcat}
Suppose given a pullback square $\sigma$ in $\operatorname{Cat}_\infty$
$$\xymatrix{
\mathcal{C}'\ar[r]^{g'}\ar[d]^{f'} & \mathcal{C}\ar[d]^f \\
\mathcal{D}'\ar[r]^{g} & \mathcal{D} \\
}$$
such that:
\begin{enumerate}
\item $f$ is proper;
\item $g$ is the inclusion of a left closed full subcategory;
\item the pullback of $f$ to the full subcategory given by the complement $\mathcal{D}\smallsetminus \mathcal{D}'$ is an equivalence.
\end{enumerate}

Then $\sigma$ is also a pushout square.
\end{corollary}
\begin{proof}
Note that the conditions are closed under base-change along any functor $\mathcal{X}\rightarrow\mathcal{D}$.  Furthermore, all of $f,g,f',g'$ are proper, as they are pullbacks of proper maps.  Therefore, by the proper descent theorem, it suffices to prove this when $\mathcal{D}=\ast$.  But then either $\mathcal{D}'=\emptyset$, in which case $f$ and $f'$ are equivalences and hence the square is a pushout, or $\mathcal{D}'=\ast$, in which case $g$ and $g'$ are equivalences and hence the square is a pushout.
\end{proof}

The following is ``descent for left-closed covers''.

\begin{corollary}\label{closeddescentcat}
Let $\mathcal{C}$ be a small $\infty$-category.  Suppose given a collection $P$ of left-closed full subcategories of $\mathcal{C}$, viewed as a poset under inclusion, such that for all $x\in\mathcal{C}$ the subposet of those elements of $P$ containing $x$ is contractible.  Then
$$\mathcal{C} = \varinjlim_{\mathcal{D}\in P}\mathcal{D}.$$
\end{corollary}
\begin{proof}
As every inclusion of left closed full subcategories is itself a left closed inclusion, it is proper.  Therefore, by the proper descent theorem, it suffices to show that the pullback of our diagram along any functor $\ast\rightarrow \mathcal{C}$, classifying an object $c\in\mathcal{C}$, has colimit $\ast$.  But as the elements of $P$ are full subcategories, every term in this pullback is either $\ast$ (when $c$ lies in the corresponding full subcategory) or $\emptyset$ (otherwise).  Thus we see exactly the condition that the poset of those $\mathcal{D}\in P$ containing $c$ should be contractible, in the form that the colimit of the terminal diagram is terminal.  (Note that colimits in $\mathcal{S}$ are automatically also colimits in $\operatorname{Cat}_\infty$, as the inclusion has a right adjoint given by neglecting the non-invertible morphisms.)
\end{proof}

Here we calculate homotopy orbits for a group action.  We stick to the special case that's relevant for us.

\begin{corollary}\label{quotientcat}
Let $P$ be a poset and $G$ a group acting on $P$.  We can encode this action by a functor $\mathcal{P}\colon BG\rightarrow\operatorname{Posets}\subseteq \operatorname{Cat}_\infty$.  Then the colimit in $\operatorname{Cat}_\infty$
$$\varinjlim_{BG} \mathcal{P}$$
naturally identifies with the action category $G\backslash\backslash P$ whose objects are the $p\in P$ and whose morphisms $p\rightarrow p'$ are the $g\in G$ with $gp\leq p'$.
\end{corollary}
\begin{proof}
Let's make the comparison map.  Note that $G\backslash\backslash \ast = BG$, and the functor $\ast\rightarrow BG$ is tautologically $G$-invariant for the trivial $G$-action on $\ast$.  Now consider the functor $G\backslash\backslash P \rightarrow BG$ induced by $P\rightarrow \ast$.  The pullback of $\ast\rightarrow BG$ along this functor recovers $P$ together with its $G$-action, which gives the desired comparison map.

\medskip

To show the comparison map is an isomorphism, because proper descent is universal it suffices to use proper descent to establish $\varinjlim_{BG}\ast \overset{\sim}{\rightarrow} BG$.  But after we pull back along $\ast \rightarrow BG$ we find that what we need is $\varinjlim_{BG} G=\ast$ where $G$ is promoted to a $G$-object by the translation action.  But $G$ with the translation action is the same as the left Kan extension of the terminal functor along $\ast\rightarrow BG$, so this follows because left Kan extensions preserve colimits.
\end{proof}

Note that the set of isomorphism classes of objects in $G\backslash\backslash P$ is the quotient set $G\backslash P$.  In general, this quotient set does not have a poset structure making the quotient $P\rightarrow G\backslash P$ a map of posets, but under a suitable regularity hypothesis this holds.

\begin{lemma}\label{posetquotient}
Let $P$ be a poset and $G$ a group acting on $P$.  Suppose that for $x\in P$ and $g\in G$ we have the implication $x\leq gx \Rightarrow x=gx$.  Then:
\begin{enumerate}
\item Every endomorphism in $G\backslash\backslash P$ is an isomorphism.
\item There is a poset structure on the quotient set $G\backslash P$ defined by $X\leq Y$ if and only if there exists an $x\in X$ and $y\in Y$ with $x\leq y$.
\item This poset structure on $G\backslash P$ serves as the quotient of $G$ acting on $P$ in the category of posets.
\end{enumerate}
\end{lemma}
\begin{proof}
If $g\in G$ gives a map $x\rightarrow x$ in the action category, then $gx\leq x$, so by hypothesis (applied to $x\leq g^{-1}x$) we deduce that $gx=x$, and it follows that $g^{-1}$ gives an inverse map.  Thus part 1 holds.  Part 2 is a consequence: in general, if $\mathcal{C}$ is a category where every endomorphism is an isomorphism, then the set of isomorphism classes of objects in $\mathcal{C}$ forms a poset with $[x]\leq [y]$ iff there exists a map $x\rightarrow y$.  Finally, part 3 is clear once we know that the quotient set is indeed a poset.
\end{proof}

The last corollary is an almost tautological though fairly fundamental colimit diagram.

\begin{corollary}\label{tautologicalbutuseful}
Let $\mathcal{C}$ be a small $\infty$-category.  Then 
$$\varinjlim_{x\in\mathcal{C}} \mathcal{C}_{/x}\overset{\sim}{\rightarrow}\mathcal{C},$$
and this colimit diagram is universal (still gives a colimit after arbitrary pullback).
\end{corollary}
\begin{proof}
All the functors in the co-cone diagram are of the form $\mathcal{D}_{/y}\rightarrow\mathcal{D}$, hence are proper by \Cref{properexamples}.  Thus, by the proper descent theorem, it suffices to see that we have a colimit diagram after pullback along any $\ast\rightarrow\mathcal{C}$ classifying an object $c\in\mathcal{C}$.  Then the claim becomes that $\varinjlim_{x\in\mathcal{C}} \operatorname{Map}_\mathcal{C}(c,x)=\ast.$  Note that the functor $x\mapsto \operatorname{Map}_{\mathcal{C}}(c,x)$ under consideration here is the left Kan extension of the terminal functor along the projection $\mathcal{C}_{c/}\rightarrow\mathcal{C}$, thus the value of the colimit is equivalently $\varinjlim_{\mathcal{C}_{c/}}\ast=\vert \mathcal{C}_{c/}\vert$.  So we need that $\mathcal{C}_{c/}$ is contractible; but indeed it has an initial object.
\end{proof}

\begin{remark}
One can also give many other proofs of this result.  For example, one can directly check that applying $\operatorname{Map}(\Delta^n,-)$ gives a colimit diagram for all $n$, so that we have an a priori stronger statement: in the complete Segal anima world, we even have a colimit of simplicial anima.  Or else one can use $\infty$-topos theory: in $\operatorname{PSh}(\mathcal{C})$ we have $\ast = \varinjlim_{x\in\mathcal{C}} h_x$ because maps out of either side calculates the limit over a $\mathcal{C}^{op}$-diagram; then the conclusion follows by descent.\exend
\end{remark}

\section{Miscellaneous background on constructible sheaves}\label{posetsection}

Let $\pi\colon X\rightarrow P$ be a stratified topological space in the sense of Lurie, \cite{LHA} Appendix A: a continuous map from a topological space $X$ to a poset $P$ equipped with the Alexandroff topology.  The $X_p:=\pi^{-1}(\{p\})$ are the \emph{strata} of the stratified space.   Recall that in the Alexandroff topology, every point $p\in P$ has a minimal open neighbourhood, namely the set of $q$ with $q\geq p$.  The stratum $X_p$ is a closed subspace of the open subspace $U_p:= \pi^{-1}(\{q\geq p\})$ of $X$, the \emph{open star} around the $p$-stratum.

\medskip

A constructible sheaf on a stratified space is a sheaf which is locally constant along each stratum. It will be handy to be able to test equivalences of constructible sheaves by restricting to strata.  Some hypothesis on $X\rightarrow P$ is necessary for this to be possible.  We prefer to impose the hypothesis only on $P$, and we will take the condition singled out by Lurie in his Theorem A.9.3: $P$ satisfies the ascending chain condition, meaning there is no infinite chain $p_0<p_1<p_2<\ldots$ of strict inequalities in $P$.

\begin{lemma}\label{testonstrata}
Suppose $X\rightarrow P$ is a stratified space with $P$ a poset satisfying the ascending chain condition.  Then a map of constructible sheaves on $X$ is an equivalence if and only if its pullback to the stratum $X_p$ is an equivalence for all $p\in P$.
\end{lemma}
\begin{proof}
Suppose we have a map $f$ of sheaves which is an equivalence on each stratum.  Since the $U_p$ cover $X$, it suffices to show that $f$ induces an isomorphism on restriction to each $U_p$.  To prove this by noetherian induction on $p$, it suffices to show that if it holds for all $q>p$, then it holds for $p$.  But $U_p\smallsetminus X_p$ is covered by the $U_q$ for $q>p$, so we deduce $f$ gives an isomorphism there.  Since $f$ gives an isomorphism on the closed complement $X_p\subset U_p$ by assumption, it follows from the gluing formalism, \cite{LHA} A.8, that $f$ gives an isomorphism on $U_p$, as desired.
\end{proof}

Haine has proved homotopy invariance for hypercomplete constructible sheaves in \cite{Haine}.  Here we prove a non-hypercomplete variant.

\begin{proposition}
Let $X\rightarrow P$ be a stratified topological space such that $P$ satisfies the ascending chain condition.  Consider the projection $f\colon X\times[0,1]\rightarrow X$.  Then the pullback functor
$$f^\ast\colon\operatorname{Sh}^{constr}(X)\rightarrow \operatorname{Sh}^{constr}(X\times [0,1])$$
is an equivalence.  Here $X\times [0,1]$ is stratified by the composition $X\times [0,1]\rightarrow X\rightarrow P$.
\end{proposition}
\begin{proof}
Recall from \cite{LHA} A.2.10 that for arbitrary topological spaces $T$, the pullback
$$f^\ast\colon\operatorname{Sh}(T)\rightarrow\operatorname{Sh}(T\times [0,1])$$
is fully faithful and admits a left adjoint $f_{\natural}$ which commutes with pullbacks in the $T$ variable; and for future reference we recall this also holds for open and half-open intervals replacing $[0,1]$.  We deduce that the $f^\ast$ in our statement is fully faithful, and that an $\mathcal{F}\in\operatorname{Sh}^{constr}(X\times [0,1])$ lies in the essential image if and only if $f_\natural\mathcal{F}$ is constructible and $\mathcal{F}\overset{\sim}{\rightarrow} f^\ast f_\natural\mathcal{F}$.  By the lemma and compatibility of $f_\natural$ with pullbacks, we therefore reduce to the case where $\mathcal{F}$ is locally constant, provided we also show $f_\natural\mathcal{F}$ is locally constant.

\medskip

Thus suppose $\mathcal{F}\in\operatorname{Sh}(X\times [0,1])$ is locally constant.  By refining an open cover of $X\times [0,1]$, we find there is an open cover $\{U_i\}$ of $X$ and open subintervals $I_{1,i},\ldots, I_{n_i,i}$ covering $[0,1]$ such that $I_{a,i}\cap I_{b,i}=\emptyset$ unless $b=a+1$ in which case there is overlap, and moreover such that $\mathcal{F}\vert_{U_i\times I_{j,i}}$ is constant for all $j$.  In particular $\mathcal{F}\vert_{U_i\times I_{j,i}}$ is pulled back from a constant sheaf on $U_i$, hence by fully faithfulness of pullbacks along intervals, the constant sheaf on $U_i$ from which it's pulled back is uniquely and functorially determined.  Thus, working our way from $1$ to $n_i$ along the intersections, we can identify all these constant sheaves on $U_i$ with one another, hence $\mathcal{F}\vert_{U_i\times [0,1]}$ is pulled back from this same constant sheaf on $U_i$.  But again by compatibility of $f_\natural$ with pullbacks, the desired claims are local on $X$, so this suffices as the $\{U_i\}$ cover.
\end{proof}

\begin{corollary}\label{equivonglobalsections}
Suppose given a map $f\colon X\rightarrow Y$ of topological spaces, compatibly stratified by a map $Y\rightarrow P$ to a $P$ satisfying the ascending chain condition.

\medskip

If $f$ is a stratified homotopy equivalence\footnote{This means there is a stratum-preserving map backwards and stratum-preserving homotopies making both composites homotopic to the identity.  In particular, the restriction to each stratum is a homotopy equivalence, but also more.}, then pullback induces
$$f^\ast\colon\operatorname{Sh}^{constr}(Y)\overset{\sim}{\rightarrow}\operatorname{Sh}^{constr}(X),$$
and in particular for any constructible sheaf $\mathcal{F}$ on $Y$ the natural map is an equivalence
$$\Gamma(Y,\mathcal{F})\overset{\sim}{\rightarrow} \Gamma(X,f^\ast\mathcal{F}).$$
\end{corollary}
\begin{proof}
The lemma implies that any two stratified-homotopic maps induce the same pullback functor on constructible sheaves. This gives the first claim, and the second claim follows by taking mapping spaces from the constant sheaf on $\ast$ to the sheaf $\mathcal{F}$.
\end{proof}

Let us return to the general situation of a stratified space $\pi:X\rightarrow P$.  Viewing $P$ as a topological space, we get an induced geometric morphism of $\infty$-topoi
$$\pi_\ast\colon\operatorname{Sh}(X)\rightarrow \operatorname{Sh}(P),$$
defined by $(\pi_\ast\mathcal{F})(V)=\mathcal{F}(\pi^{-1}V)$.  The left adjoint $\pi^\ast$ lands inside the full subcategory $\operatorname{Sh}^{constr}(X)$ of constructible sheaves, because the restriction to each stratum $X_p$ is in fact constant, being pulled back from a sheaf on the one-point space $\{p\}$.  Thus we have a comparison functor
$$\pi^\ast\colon \operatorname{Sh}(P)\rightarrow \operatorname{Sh}^{constr}(X).$$

\medskip

However, when $P$ satisfies the ascending chain condition, Lemma \ref{testonstrata} applied to $X=P$ shows that equivalences of sheaves on $P$ are tested on stalks.  Thus every sheaf on $P$ is \emph{hypercomplete}:
$$\operatorname{Sh}(P) = \operatorname{Sh}^{hyp}(P).$$
On the other hand, Ko Aoki in \cite{Aoki} has shown for a general poset that 
$$\operatorname{Sh}^{hyp}(P)\overset{\sim}{\rightarrow} \operatorname{Fun}(P;\mathcal{S})$$
via the functor $\mathcal{F}\mapsto (p\mapsto \mathcal{F}(\{q:q\geq p\})$.  It follows that when $\pi:X\rightarrow P$ is a stratified space such that $P$ satisfies the ascending chain condition, we get a geometric morphism of $\infty$-topoi
$$\pi_\ast\colon\operatorname{Sh}(X)\rightarrow \operatorname{Fun}(P;\mathcal{S})$$
given by $(\pi_\ast\mathcal{F})(p)=\mathcal{F}(U_p)$, and an induced comparison functor
$$\pi^\ast\colon\operatorname{Fun}(P;\mathcal{S})\rightarrow \operatorname{Sh}^{constr}(X).$$

In the next theorem, we give sufficient conditions on $X$ for this comparison functor to be an equivalence.

\begin{theorem}\label{itsaposet}
Let $\pi\colon X\rightarrow P$ be a stratified topological space with $\pi$ surjective and $P$ satisfying the ascending chain condition.  Suppose there is a collection $\mathcal{B}$ of open subsets of $X$ such that:
\begin{enumerate}
\item the representable sheaves $h_U$ for $U\in\mathcal{B}$ generate the $\infty$-topos $\operatorname{Sh}(X)$;\footnote{This condition implies that $\mathcal{B}$ is a basis for the topology. If every sheaf is hypercomplete, the converse holds.  In general it's enough for $\mathcal{B}$ to be a basis closed under finite intersections, or even just a collection such that every open subset admits a truncated hypercover by elements of $\mathcal{B}$.}
\item for all $U\in\mathcal{B}$, there is a $p\in P$ such that $U$ includes into $U_p$ by a stratified homotopy equivalence.
\end{enumerate}
Then the pullback map
$$\pi^\ast\colon \operatorname{Fun}(P,\mathcal{S})\rightarrow \operatorname{Sh}(X)$$
preserves all limits and colimits and is fully faithful with essential image $\operatorname{Sh}^{constr}(X)$.  Moreover every constructible sheaf on $X$ is the limit of its Postnikov tower, and hence is hypercomplete, compare \cite{LHA} A.5.9.
\end{theorem}
\begin{proof}
Let $\mathcal{F}\in\operatorname{Sh}(X)$.  We claim the following are equivalent:
\begin{enumerate}
\item $\mathcal{F}$ is constructible;
\item For all $U\in\mathcal{B}$, we have $\mathcal{F}(U_p)\overset{\sim}{\rightarrow}\mathcal{F}(U)$ where $U_p$ is as in hypothesis 2 ($U_p$ is uniquely determined by $U$, as $U_p=\pi^{-1}(\pi(U))$);
\item $\pi^\ast\pi_\ast\mathcal{F}\overset{\sim}{\rightarrow}\mathcal{F}$.
\end{enumerate}
Indeed, 1 $\Rightarrow$ 2 follows from \Cref{equivonglobalsections}  and 3 $\Rightarrow$ 1 holds because $\pi^\ast$ lands inside the constructible sheaves.  For 2 $\Rightarrow 3$, note that 2 says that $\mathcal{F}\vert_{\mathcal{B}}$ is the presheaf pullback of the presheaf $\pi_\ast\mathcal{F}$.  Now take an arbitrary sheaf $\mathcal{G}$ on $X$.  By hypothesis 1 in our theorem, we can calculate $\operatorname{Map}(\mathcal{F},\mathcal{G})$ as maps of presheaves on $\mathcal{B}$.  Thus we deduce $\operatorname{Map}(\mathcal{F},\mathcal{G})=\operatorname{Map}_{\operatorname{PSh}(P^{op})}(\pi_\ast\mathcal{F},\pi_\ast\mathcal{G})$, which says exactly that $\pi^\ast\pi_\ast\mathcal{F}\overset{\sim}{\rightarrow}\mathcal{F}$.

\medskip

From $1\Leftrightarrow 2$, we already see that the full subcategory of constructible sheaves is closed under all limits.  Furthermore, from 1 $\Leftrightarrow$ 3 we see that if $\mathcal{F}$ is constructible then $\pi^\ast\pi_\ast\mathcal{F}\overset{\sim}{\rightarrow}\mathcal{F}$, so to see the equivalence $\pi^\ast\colon \operatorname{Fun}(P,\mathcal{S})\overset{\sim}{\rightarrow} \operatorname{Sh}^{constr}(X)$, which also gives preservation under colimits, we only need the other direction $\varphi\overset{\sim}{\rightarrow} \pi_\ast\pi^\ast \varphi$ for $\varphi\in\operatorname{Fun}(P,\mathcal{S})$.  However, by adjunction identities and the previous direction it suffices to show that if a map $\varphi\rightarrow\varphi'$ is an equivalence on $\pi^\ast$, then it is an equivalence.  But $\varphi(p)$ is recovered as the pullback of $\pi^\ast\varphi$ to any point in the stratum $X_p$, so this follows from the surjectivity of $\pi$.

\medskip

The final claim about Postnikov towers follows, because the analogous claim in $\operatorname{Fun}(P,\mathcal{S})$ is clear as Postnikov truncations and limits are computed objectwise.
\end{proof}

We can interpret the above theorem in light of the following definition.

\begin{definition}
Let $P$ be a poset satisfying the ascending chain condition.
\begin{enumerate}
\item We say that a stratified space $X\rightarrow P$ \emph{admits an exit path $\infty$-category} if the following conditions hold:
\begin{enumerate}
\item The full subcategory $\operatorname{Sh}^{constr}(X)\subset\operatorname{Sh}(X)$ is closed under all limits and colimits;
\item The $\infty$-category $\operatorname{Sh}^{constr}(X)$ is generated under colimits by a set of atomic objects (see \Cref{atomic}).
\item $\pi^\ast\colon\operatorname{Fun}(P,\mathcal{S})\rightarrow\operatorname{Sh}^{constr}(X)$ preserves all limits (and colimits, but that is automatic);
\end{enumerate}
\item If $X\rightarrow P$ admits an exit path $\infty$-category, we define its exit path $\infty$-category $\Pi(X\rightarrow P)$ to be the opposite category of the full subcategory $\left[\operatorname{Sh}^{constr}(X)\right ]^{atom}$ of atomic constructible sheaves.
\item If $f\colon(X\rightarrow P)\rightarrow (Y\rightarrow Q)$ is a map of stratified spaces, we say that \emph{$f$ respects exit path $\infty$-categories} if $f^\ast\colon\operatorname{Sh}^{constr}(Y)\rightarrow\operatorname{Sh}^{constr}(X)$ preserves limits (and colimits, but that is automatic).\defend
\end{enumerate}
\end{definition}

If $X\rightarrow P$ admits an exit path $\infty$-category, it follows from \Cref{smallvbig} that there is an induced ``exodromy'' equivalence (cf. \cite{BGH} for the terminology)
$$\operatorname{Fun}(\Pi(X\rightarrow P),\mathcal{S})\overset{\sim}{\rightarrow} \operatorname{Sh}^{constr}(X),$$
and it follows from \Cref{smallvbig} that, in terms of the exodromy equivalence, the pullback map $\pi^\ast\colon\operatorname{Fun}(P,\mathcal{S})\rightarrow\operatorname{Sh}^{constr}(X)$ is recovered as composition along a uniquely determined functor
$$\Pi(X\rightarrow P)\rightarrow P.$$

Similarly, the condition that $f\colon(X\rightarrow P)\rightarrow (Y\rightarrow Q)$ respect exit path $\infty$-categories is equivalent to the condition that the induced pullback functor on constructible sheaves is given, via exodromy, by composition with a functor
$$\Pi(X\rightarrow P)\rightarrow \Pi(Y\rightarrow Q).$$
Moreover, this functor is then uniquely determined as the restriction to atomic objects of the left adjoint to $f^\ast\colon\operatorname{Sh}^{constr}(Y)\rightarrow\operatorname{Sh}^{constr}(X)$, see \Cref{smallvbig}.

\medskip

\Cref{itsaposet} already gives examples of stratified spaces admitting an exit path $\infty$-category; indeed, in those cases the exit path $\infty$-category is the stratifying poset $P$ itself.  Note that if $(X\rightarrow P)$ and $(Y\rightarrow Q)$ are stratified spaces whose exit path $\infty$-category identifies with the stratifying poset, then every map $(X\rightarrow P)\rightarrow (Y\rightarrow Q)$ respects exit path $\infty$-categories simply because we are given the required map $P\rightarrow Q$ as part of the data.  This simple observation, together with the following permanence properties, will be enough for us to identify the exit path $\infty$-categories we need in the next section.

\begin{proposition}\label{calculatebydescent}
Let $P$ be a poset satisfying the ascending chain condition. 
\begin{enumerate}
\item Let $f\colon X\rightarrow P$ be a stratified space admitting an exit path $\infty$-category.  Then for every locally closed subset $Q\subset P$, the stratified space $f^{-1}(Q)\rightarrow Q$ admits an exit path $\infty$-category, the inclusion $(f^{-1}Q\rightarrow Q)\rightarrow (X\rightarrow P)$ respects exit path $\infty$-categories, and
$$\Pi(f^{-1}(Q)\rightarrow Q)\overset{\sim}{\rightarrow} \Pi(X\rightarrow P)\times_P Q.$$

\item Let $K$ be a small $\infty$-category and $\{X_k\rightarrow P_k\}_{k\in K}$ a $K$-shaped diagram of stratified spaces, equipped with a co-cone $X_\infty\rightarrow P_\infty$.  Suppose:
\begin{enumerate}
\item Each $P_k$ for $k\in K$ and $P_\infty$ satisfy the ascending chain condition.
\item Each object $X_k\rightarrow P_k$ admits an exit path $\infty$-category for $k\in K$.
\item Each map $(X_k\rightarrow P_k)\rightarrow (X_{k'}\rightarrow P_{k'})$ respects exit path $\infty$-categories for $k\rightarrow k'$.
\item We have $\operatorname{Sh}(X_\infty)\overset{\sim}{\rightarrow}\varprojlim_{k\in K^{op}}\operatorname{Sh}(X_k)$ and $\operatorname{Sh}^{constr}(X_\infty)\overset{\sim}{\rightarrow}\varprojlim_{k\in K^{op}}\operatorname{Sh}^{constr}(X_k)$ via pullback.
\end{enumerate}
Then:
\begin{enumerate}
\item $X_\infty\rightarrow P_\infty$ admits an exit path $\infty$-category.
\item Each map $(X_k\rightarrow P_k)\rightarrow (X_\infty\rightarrow P_\infty)$ respects exit path $\infty$-categories, $k\in K$.
\item $\Pi(X_\infty\rightarrow P_\infty)\overset{\sim}{\leftarrow} \varinjlim_{k\in K}\Pi(X_k\rightarrow P_k).$
\end{enumerate}
\end{enumerate}
\end{proposition}
\begin{proof}
For part 1, by factoring a locally closed inclusion as a closed inclusion followed by an open inclusion, it suffices to treat those cases separately.  For an open subset $U\subset P$, we have that $\operatorname{Sh}(U)\overset{\sim}{\rightarrow}\operatorname{Sh}(X)_{/h_U}$ via the left adjoint to the pullback, see \cite{L} 7.3.2.  Moreover, this description is compatible with base change, hence it passes to constructible sheaves: $\operatorname{Sh}^{constr}(U)\overset{\sim}{\rightarrow}(\operatorname{Sh}^{constr}(X))_{/h_U}$.  This gives the conclusion in that case.  For a closed subset $Z\subset P$, we argue similarly but using $\operatorname{Sh}(Z)\overset{\sim}{\rightarrow}\operatorname{ker}(\operatorname{Sh}(X)\rightarrow\operatorname{Sh}(U))$, \cite{L} 7.3.2, meaning those sheaves on $X$ which restrict to $\emptyset$ on $U=X\smallsetminus Z$, the equivalence being induced by pushforward.  Part 2 is straightforward from \Cref{smallvbig}.
\end{proof}

We also note that if we take $P=\ast$ then constructible sheaf means locally constant sheaf, and by comparing with \cite{LHA} A.1.5 we find that $X\rightarrow \ast$ admits an exit path $\infty$-category if and only if $\operatorname{Sh}(X)$ is locally of constant shape in the sense of \cite{LHA}, and the exit path $\infty$-category is the $\infty$-groupoid given by the shape.

\medskip

To finish, let us also recall from \cite{OrsnesJansen} that exodromy for constructible sheaves with values in $\mathcal{S}$ automatically extends to sheaves with values in an arbitrary compactly generated $\infty$-category.

\begin{proposition}
Let $P$ be a poset satisfying the ascending chain condition, let $X\rightarrow P$ be a stratified space which admits an exit path $\infty$-category, and let $\mathcal{E}$ be a compactly generated $\infty$-category.  Then there is a natural equivalence
$$\operatorname{Fun}(\Pi(X\rightarrow P),\mathcal{E})\overset{\sim}{\rightarrow}\operatorname{Sh}^{constr}(X;\mathcal{E})$$
where the constructible full subcategory of $\operatorname{Sh}(X;\mathcal{E})$ is again defined as the full subcategory of those sheaves whose pullback to each stratum is locally constant.

\medskip

This equivalence is essentially determined from the version where $\mathcal{E}=\mathcal{S}$ as follows: a functor $\varphi\in \operatorname{Fun}(\Pi(X\rightarrow P),\mathcal{E})$ and a constructible sheaf $\mathcal{F}\in \operatorname{Sh}^{constr}(X;\mathcal{E})$ correspond under the above equivalence if and only if for all $x\in\mathcal{E}$, the functor $\operatorname{Map}(x,\varphi(-))$ and the sheaf $\operatorname{Map}(x, \mathcal{F}(-))$ correspond under the equivalence for $\mathcal{E}=\mathcal{S}$.
\end{proposition}

\section{Borel--Serre and reductive Borel--Serre compactifications}

We start with a recap of some material from Borel--Serre's article (\cite{BorelSerre}), taken from a slightly different perspective.

\subsection{The canonical homogeneous space over \texorpdfstring{$\mathbb{R}$}{R}}

Let $G=G_\mathbb{R}$ be a connected reductive group over $\mathbb{R}$.  The canonical homogeneous space $X$ is a transitive $G(\mathbb{R})$-space whose isotropy groups are exactly the maximal compact subgroups of $G(\mathbb{R})$.  It exists and is unique up to isomorphism as $G(\mathbb{R})$ admits a unique conjugacy class of maximal compact subgroups (\cite{Mostow}).  This does not quite justify calling it ``canonical'' because it is not in general unique up to unique isomorphism.  But we can always fix a choice for $G$, say $X=G(\mathbb{R})/K$ for some choice of maximal compact subgroup $K$, and as we discuss later this determines a choice for each Levi factor of $G$ as well, and that will be enough canonicity for us.

\subsection{The Borel--Serre corners}

Now, switching notation, let us take a connected reductive group $G$ over $\mathbb{Q}$ giving rise to a $G_{\mathbb{R}}$ as in the previous section by extension of scalars.  We will be interested in the restriction of the $G(\mathbb{R})$-action on $X$ to arithmetic subgroups $\Gamma\subset G(\mathbb{Q})$.  The basic problem to ``fix'' is that the $\Gamma$-action on $X$, while properly discontinuous, is not cocompact.  It turns out the explanation for this non-cocompactness lies in the parabolic subgroups of $G$, and we start with a brief recap on those and their relation to relative root systems.

\medskip

For a parabolic subgroup $P$ of $G$, let $S_P$ denote the maximal split torus in the centre of the Levi factor $P/U_P$, where $U_P$ is the unipotent radical of $P$.  If $P\subset Q$ then there is an induced natural injection $S_Q\hookrightarrow S_P$. Moreover, if $P$ is conjugate to $P'$ then any choice of conjugating element induces the \emph{same} isomorphism $S_P\overset{\sim}{\rightarrow} S_{P'}$.  In this sense $S_P=S_{[P]}$ only depends on the conjugacy class $[P]$ of $P$, and all the $S_{[P]}$ can be compatibly viewed as subtori of the ``abstract maximal split torus'', which is $S:=S_{[P_0]}$ for a minimal parabolic $P_0$.

\medskip

We recall also that there is a canonical finite subset $\Delta\subset X^\ast(S)=\operatorname{Hom}(S,\mathbb{G}_m)$ such that if we choose a maximal split torus $S_0$ inside a minimal parabolic $P_0$, determining an isomorphism $S_0\simeq S$ via projection to $P_0/U_{P_0}$, then $\Delta$ corresponds, via this isomorphism, to the basis of the relative root system $\Phi(S_0,G)$ occuring as weights in $\operatorname{Lie}(U_{P_0})$, compare \cite[\S 5]{BorelTits}, \cite[III.1.14]{BorelJi}.  Then there is an inclusion-reversing bijective correspondence between conjugacy classes of parabolic subgroups $P\subset G$ and subsets of $\Delta$, determined as follows: if $\Delta_P\subset \Delta$ is the subset corresponding to $[P]$, then $S_{[P]}=\left(\cap_{\chi\in \Delta\smallsetminus \Delta_P}\operatorname{ker}(\chi)\right)^\circ$, see \cite[4.1]{BS}.  Another way of describing the situation is that the restriction of $\Delta_P$ to $S_{[P]}$ gives a basis of $\operatorname{ker}(X^\ast(S_{[P]})\otimes\mathbb{Q}\rightarrow X^\ast(S_{[G]})\otimes\mathbb{Q})$.

\medskip

Let us note the following consequence of this discussion of conjugacy classification of parabolic subgroups.  It will be used over and over again.
\begin{proposition}\label{parabolicnormalise}
Let $P'\subset P$ be an inclusion of parabolic subgroups of $G$.  If $\gamma \in G(\mathbb{Q})$ satisfies $\gamma P'\gamma^{-1} \subset  P$, then $\gamma\in P(\mathbb{Q})$.
\end{proposition}
\begin{proof}
The classification of conjugacy classes of parabolic subgroups recalled above, applied to both $G$ and the Levi factor $P/U_P$, implies in particular that two parabolic subgroups of $P/U_P$ are $P/U_P$-conjugate if and only if their preimages are $G$-conjugate.  We deduce that there is a $\rho\in P(\mathbb{Q})$ with $\rho P'\rho^{-1}/U_P = \gamma P'\gamma^{-1}/U_P$, which implies $\rho P'\rho^{-1} = \gamma P'\gamma^{-1}$.  (Note that $U_{P}\subset U_{P'}\subset P'$.)  Thus $\rho^{-1}\gamma$ normalises $P'$.  But every parabolic is its own normaliser, \cite[4.3]{BorelTits}, so $\gamma \in P(\mathbb{Q})$ as desired.
\end{proof}

Now, again for a parabolic subgroup $P\subset G$, let $A_P=S_P(\mathbb{R})^\circ$.  This group plays a crucial role in the story.  Namely, on the one hand there is a natural proper and free right action $\bullet_P\colon X\times A_P\rightarrow X$ of $A_P$ on $X$, the \emph{geodesic action} of \cite[\S 3]{BS}.  But on the other hand the discussion above provides natural \emph{root coordinates}

$$A_P/A_G\overset{\sim}{\rightarrow} (\mathbb{R}_{>0})^{\Delta_P},$$

\noindent see \cite[4.2]{BS}.  For $P\subset Q$, the geodesic actions of $A_P$ and $A_Q$ are compatible under the natural injection $A_Q\hookrightarrow A_P$, and the root coordinates are too, in that they make this injection correspond to the inclusion $(\mathbb{R}_{>0})^{\Delta_Q}\subset(\mathbb{R}_{>0})^{\Delta_P}$ of the coordinate hypersurface corresponding to $\Delta_Q\subset\Delta_P$.

\medskip

Take care that while $A_P$, $\Delta_P$, and the root coordinates only depend on the conjugacy class $[P]$, the geodesic action depends on $P$ itself. In fact, if $\gamma\in G(\mathbb{Q})$ and $a\in A_{[P]}$, then
$$x\bullet_{\gamma P\gamma^{-1}}a = \gamma x \bullet_{P}a,$$
see \cite[5.6]{BS}.  Note that this in particular says that the $P(\mathbb{Q})$-action on $X$ commutes with the geodesic action by $A_P$; but in fact the whole $P(\mathbb{R})$-action does.  (This is explained by the fact that one can also define an analogous geodesic action for an arbitrary parabolic subgroup of $G_{\mathbb{R}}$, and then the previous formula holds for all $\gamma\in G(\mathbb{R})$.)

\medskip

Loosely speaking, the point of all this is that the geodesic actions by parabolic subgroups give enough directions via which a point in $X$ can ``wander off to $\infty$'' to fully account for the non-cocompactness of the $\Gamma$-action on $X$.  Actually, with the conventions of  \cite{BS}, it is the limit as $t\rightarrow 0$ in $\mathbb{R}_{>0}$ that corresponds, under the root coordinates and the geodesic action, to wandering off to $\infty$ in $X$.  This leads to the following definition.
\begin{definition}
Let $G$ be a reductive group over $\mathbb{Q}$, and recall the canonical homogeneous space $X$ associated to $G_{\mathbb{R}}$ as above, whose stabilisers are the maximal compact subgroups of $G(\mathbb{R})$.  Let $P\subset G$ be a parabolic subgroup.
\begin{enumerate}
\item The \emph{Borel--Serre corner} is the topological space defined by
$$\widehat{X}_{\geq P} := X\times^{A_P}(\mathbb{R}_{\geq 0})^{\Delta_P},$$
the quotient of $X\times (\mathbb{R}_{\geq 0})^{\Delta_P}$ which equalises the right geodesic action on $X$ and the left action by componentwise multiplication on $(\mathbb{R}_{\geq 0})^{\Delta_P}$ via the root coordinates $A_P/A_G\simeq (\mathbb{R}_{>0})^{\Delta_P}$.

\item The \emph{combinatorial Borel--Serre corner} is the partially ordered set
$$\{0<1\}^{\Delta_P},$$
which we will identify with the poset $\mathcal{P}_{P/}$ of parabolic subgroups containing $P$ under containment, by matching $Q\supseteq P$ with the indicator function of the subset $\Delta_P\smallsetminus \Delta_Q\subseteq \Delta_P$.
\item The \emph{stratified Borel--Serre corner} is the continuous projection map
$$\widehat{X}_{\geq P}\rightarrow \{0<1\}^{\Delta_P}$$
induced by the map $\mathbb{R}_{\geq 0}\rightarrow \{0<1\}$ sending $0$ to $0$ and $t\neq 0$ to $1$.\defend
\end{enumerate}
\end{definition}

We have a stratified homeomorphism $\widehat{X}_{\geq P} \simeq \mathbb{R}^d\times (\mathbb{R}_{\geq 0})^{\Delta_P}\rightarrow \{0<1\}^{\Delta_P}$ for some $d\geq 0$.  Indeed, if we fix a point $x$ on $X$ determining a maximal compact subgroup of $G(\mathbb{R})$, then the Langlands decomposition of $P(\mathbb{R})$ and the fact that $P(\mathbb{R})$ acts transitively on $X$ (by the Iwasawa decomposition) give an isomorphism $X\simeq \mathbb{R}^d\times A_P$ via which the geodesic action is the right action on the second coordinate, see \cite[5.4]{BS}.

\begin{lemma}\label{exitcorner}
The Borel--Serre corner $\widehat{X}_{\geq P}$ admits an exit path $\infty$-category, which identifies with its stratifying poset $\mathcal{P}_{P/}$, see Section \ref{posetsection}.  In particular, the pullback functor
$$\operatorname{Fun}(\mathcal{P}_{P/},\mathcal{S})\rightarrow \operatorname{Sh}^{constr}(\widehat{X}_{\geq P})$$
is an equivalence of $\infty$-categories.
\end{lemma}
\begin{proof}
We need to produce the neighbourhood bases as in condition 1 and 2 of \Cref{itsaposet}, for $X=\mathbb{R}^d\times(\mathbb{R}_{\geq 0})^n\rightarrow \{0<1\}^n$.  It suffices to take the open boxes $(a_1,b_1)\times\ldots\times (a_{d+n},b_{d+n})$ in $\mathbb{R}^{d+n}$ and intersect with $X$.
\end{proof}

\subsection{The Borel--Serre compactification}

In order to define the Borel--Serre compactification, we need to discuss the functoriality of the Borel--Serre corners.  There are two types of functoriality:
\begin{enumerate}
\item First of all, if $P\subset Q$ is an inclusion of parabolics, then the compatibility of the geodesic action and root coordinates with the inclusion $A_Q\hookrightarrow A_P$ gives a natural open inclusion
$$\widehat{X}_{\geq Q}\hookrightarrow \widehat{X}_{\geq P}$$
lying above the combinatorial analogue
$$\mathcal{P}_{Q/}\hookrightarrow \mathcal{P}_{P/}$$
coming from including the poset of parabolics containing $Q$ into that of those containing $P$.
\item Second, for $\gamma\in G(\mathbb{Q})$ the action of $\gamma$ on $X$ induces a natural homeomorphism
$$\widehat{X}_{\geq P}\overset{\sim}{\rightarrow} \widehat{X}_{\geq \gamma P\gamma^{-1}},$$
lying above the combinatorial analogue
$$\mathcal{P}_{P/}\overset{\sim}{\rightarrow}\mathcal{P}_{\gamma P\gamma^{-1}/}$$
coming from conjugating a parabolic containing $P$ by $\gamma$.
\end{enumerate}

The first functoriality is more formally a functor $P\mapsto \widehat{X}_{\geq P}$ from the poset $\mathcal{P}^{op}$ of parabolic subgroups under reverse inclusion to topological spaces, lying over an analogous functor $P\mapsto \mathcal{P}_{P/}$ from $\mathcal{P}^{op}$ to posets.

\begin{definition}
For $G$ a reductive group over $\mathbb{Q}$, the \emph{Borel--Serre partial compactification} (of $X/A_G$) is the topological space defined as the colimit
$$\widehat{X}:=\varinjlim_{P\in \mathcal{P}^{op}} \widehat{X}_{\geq P},$$
viewed as a stratified space over
$$\mathcal{P}=\varinjlim_{P\in\mathcal{P}^{op}} \mathcal{P}_{P/},$$
the poset of parabolic subgroups of $G$ under inclusion.\defend
\end{definition}

Note that the entire structure of the colimit defining $\widehat{X}$ is recovered from the output stratified space $\pi\colon\widehat{X}\rightarrow \mathcal{P}$, because $\widehat{X}_{\geq P}$ identifies with the open star $\pi^{-1}(\mathcal{P}_{P/})$ around the $P$-stratum of $\widehat{X}$.

\begin{proposition}
Let $G$ be a reductive group over $\mathbb{Q}$.  Then via the natural pullback functors we have:
\begin{enumerate}
\item $\operatorname{Sh}(\widehat{X})\overset{\sim}{\rightarrow} \varprojlim_{P\in\mathcal{P}^{op}}\operatorname{Sh}(\widehat{X}_{\geq P})$.
\item $\operatorname{Sh}^{constr}(\widehat{X})\overset{\sim}{\rightarrow} \varprojlim_{P\in\mathcal{P}^{op}}\operatorname{Sh}^{constr}(\widehat{X}_{\geq P})$.
\item $\operatorname{Fun}(\mathcal{P},\mathcal{S})\overset{\sim}{\rightarrow}\varprojlim_{P\in\mathcal{P}^{op}}\operatorname{Fun}(\mathcal{P}_{P/},\mathcal{S})$.
\end{enumerate}
\end{proposition}
\begin{proof}
Part 1 follows from \Cref{opendescent} and part 3 follows from \Cref{tautologicalbutuseful}.  To deduce 2 from 1, we need to know that a sheaf on $\widehat{X}$ is constructible if its pullback to each $\widehat{X}_{\geq P}$ is.  For that it suffices to note that the $P$-stratum is fully contained in $\widehat{X}_{\geq P}$.
\end{proof}

\begin{corollary}\label{partialexit}
The stratified space $\widehat{X}\rightarrow\mathcal{P}$ admits an exit path $\infty$-category which identifies with its stratifying poset $\mathcal{P}$; in particular, the comparison functor gives an equivalence
$$c^\ast\colon\operatorname{Fun}(\mathcal{P},\mathcal{S})\overset{\sim}{\rightarrow}\operatorname{Sh}^{constr}(\widehat{X}).$$
\end{corollary}
\begin{proof}
The comparison functor is natural in the stratified space by construction, so \Cref{calculatebydescent} parts 2 and 3 reduce us to the analogous claim for the $\widehat{X}_{\geq P}$, which is \Cref{exitcorner}.
\end{proof}

Now it is time to consider the second functoriality on the Borel--Serre corners.  In terms of the glued space $\widehat{X}$, this simply manifests itself in a continuous action of the discrete group $G(\mathbb{Q})$, extending the natural action on the interior $X/A_G$ and covering the conjugation action of $G(\mathbb{Q})$ on $\mathcal{P}$.  The main result of Borel--Serre is that if $\Gamma\subset G(\mathbb{Q})$ is a torsionfree arithmetic subgroup, then $\Gamma$ acts properly, freely, and cocompactly on $\widehat{X}$.  Thus the quotient space $\Gamma\backslash \widehat{X}$ is a compact Hausdorff space compactifying its interior $\Gamma \backslash X/A_G$, which has the homotopy type of $B\Gamma$.

\begin{definition}
Let $G$ be a reductive group over $\mathbb{Q}$ and let $\Gamma\subset G(\mathbb{Q})$ be a torsionfree arithmetic subgroup.  Then the \emph{Borel--Serre compactification} (of $\Gamma\backslash X/A_G$) is the quotient space
$$\Gamma\backslash \widehat{X}$$
of the Borel--Serre partial compactification by the natural $\Gamma$-action, viewed as a stratified space over the quotient poset $\Gamma\backslash \mathcal{P}$ of $\Gamma$-conjugacy classes of parabolic subgroups under the relation induced by inclusion.\defend
\end{definition}

By \Cref{parabolicnormalise} and \Cref{posetquotient}, the quotient $\Gamma\backslash \mathcal{P}$ in the category of sets gets an induced poset structure from $\mathcal{P}$.  However, this is not the same as the (homotopy) quotient in the $\infty$-category of $\infty$-categories; rather that is the action category $\Gamma\backslash \backslash \mathcal{P}$, see \Cref{quotientcat}, whose objects are the $P\in\mathcal{P}$ and whose morphisms $P\rightarrow Q$ are the $\gamma\in\Gamma$ with $\gamma P\gamma^{-1}\subseteq Q$.

\begin{proposition}
Let $G$ be a reductive group over $\mathbb{Q}$ and $\Gamma\subset G(\mathbb{Q})$ a torsionfree arithmetic subgroup.  Then:
\begin{enumerate}
\item $\operatorname{Sh}(\Gamma\backslash\widehat{X})\overset{\sim}{\rightarrow} \operatorname{Sh}(\widehat{X})^{\Gamma}$.
\item $\operatorname{Sh}^{constr}(\Gamma\backslash\widehat{X})\overset{\sim}{\rightarrow} \operatorname{Sh}^{constr}(\widehat{X})^{\Gamma}$
\item $\operatorname{Fun}(\Gamma\backslash\backslash \mathcal{P},\mathcal{S})\overset{\sim}{\rightarrow} \operatorname{Fun}(\mathcal{P},\mathcal{S})^\Gamma.$
\end{enumerate}
\end{proposition}
\begin{proof}
Part 1 follows from \Cref{quotienttop}, and part 3 follows from \Cref{quotientcat}.  To deduce part 2 from part 1, we need to know that a sheaf on $\Gamma\backslash \widehat{X}$ is constructible if its pullback to $\widehat{X}$ is.  This follows because for any parabolic $P$, the projection from the $P$-stratum in $\widehat{X}$ to the $[P]$-stratum in $\Gamma\backslash \widehat{X}$ has local sections; in fact it is the quotient by the proper free action of $\Gamma_P=\Gamma \cap P(\mathbb{Q})$.
\end{proof}

\begin{corollary}
Let $G$ be a reductive group over $\mathbb{Q}$ and $\Gamma\subset G(\mathbb{Q})$ a torsionfree arithmetic subgroup.  Then $\Gamma\backslash\widehat{X}$ admits an exit path $\infty$-category which identifies with $\Gamma\backslash \backslash \mathcal{P}$.  In particular, there is a natural equivalence

$$\operatorname{Fun}(\Gamma\backslash \backslash \mathcal{P},\mathcal{S})\overset{\sim}{\rightarrow} \operatorname{Sh}^{constr}(\Gamma\backslash \widehat{X}).$$

\end{corollary}
\begin{proof}
Follows by combining the previous proposition,  \Cref{calculatebydescent} and \Cref{partialexit}.
\end{proof}

\subsection{The reductive Borel--Serre compactification}

To motivate the reductive Borel--Serre compactification, we start by taking a closer look at the Borel--Serre compactification.  For a parabolic subgroup $P$, the $P$-stratum of $\widehat{X}$ identifies with $X\times^{A_P}\ast = X/A_P$, and the $[P]$-stratum of $\Gamma\backslash \widehat{X}$ identifies with the quotient
$$\Gamma_P\backslash X/A_P.$$
Thus all the strata are of a similar form as the open stratum $\Gamma \backslash X/A_G$, except with the reductive group $G$ replaced by the non-reductive group $P$.

\medskip

(One may be bothered by the fact that this description of the $[P]$-stratum, on the face of it, depends on the chosen representative $P$.  But this is an illusion: if $P$ is conjugate to $P'$ via $\gamma$, then the induced homeomorphism $\Gamma_P\backslash X/A_P\simeq \Gamma_{P'}\backslash X/A_P$ is independent of $\gamma$.  This follows from the fact that parabolic subgroups are their own normalisers.)

\medskip

To get a better inductive structure we would like to replace the parabolic subgroup by its Levi quotient $L=P/U_P$, which is reductive.   Let $\Gamma_L\subset L(\mathbb{Q})$ denote the quotient $\Gamma_P/\Gamma_{U_P}$. If $\Gamma$ is not just torsionfree but \emph{neat}, \cite{Ji}, then $\Gamma_L$ is also a neat, and in particular torsionfree, arithmetic subgroup of $L(\mathbb{Q})$.  Moreover, there is a natural map

$$\Gamma_P\backslash X/A_P\rightarrow \Gamma_L \backslash X_L/A_L$$
from the $[P]$-stratum of the Borel--Serre compactification for $\Gamma \subset G(\Q)$ to the open stratum of the Borel--Serre compactification for $\Gamma_L\subset L(\Q)$, which is in fact a fibre bundle with compact fibre $U_P(\mathbb{R})/\Gamma_{U_P}$.  Indeed, this map comes from the canonical identification
$$X_L = U_P(\mathbb{R})\backslash X$$
of the canonical homogeneous space of $L_{\mathbb{R}}$ with the indicated quotient of the canonical homogeneous space of $G_{\mathbb{R}}$, together with the fact that $A_L$ acting on $X_L$ identifies with $A_P$ acting on $U_P(\mathbb{R})\backslash X$. More generally if $P'\subset P$ is parabolic, then the geodesic action of $A_{P'/U_P}$ on $X_L$ identifies with the geodesic action of $A_{P'}$ on $U_P(\mathbb{R})\backslash X$. We will write $X_G=X$ if we want to stress the distinction between the canonical homogeneous spaces associated to these different reductive algebraic groups, $G$ and the Levi quotients $L$. This leads to the following definition made by Zucker (\cite{Z}).

\begin{definition}
Let $G$ be a reductive group over $\mathbb{Q}$ and $\Gamma\subset G(\mathbb{Q})$ a neat arithmetic subgroup.  The \emph{reductive Borel--Serre compactification} (of $\Gamma\backslash X/A_G$) is the quotient topological space
$$\widehat{Y}_\Gamma=(\Gamma\backslash \widehat{X})/\sim$$
obtained from $\Gamma\backslash \widehat{X}$ by collapsing the $[P]$-stratum to $\Gamma_L\backslash X_L/A_L$ via the above quotient map, for all parabolic subgroups $P$ (or just one representative from each $\Gamma$-conjugacy class).

\medskip

We view $\widehat{Y}_\Gamma$ as stratified over the poset $\Gamma\backslash \mathcal{P}$ of $\Gamma$-conjugacy classes of parabolic subgroups by the unique factoring
$$\widehat{Y}_\Gamma\rightarrow \Gamma\backslash\mathcal{P}$$
of the stratifying map $\Gamma\backslash \widehat{X}\rightarrow \Gamma\backslash \mathcal{P}$ of the Borel--Serre compactification.\defend
\end{definition}

Zucker checked that that $\widehat{Y}_\Gamma$ is Hausdorff, hence it is a compact Hausdorff space.   An important aspect of the reductive Borel--Serre compactification is its inductive nature, based on the following:

\begin{proposition}\label{inductivetop}
Let $G$ be a reductive group over $\mathbb{Q}$ and $\Gamma\subset G(\mathbb{Q})$ a neat arithmetic subgroup.  Then:
\begin{enumerate}
\item The projection map $\Gamma\backslash\widehat{X}\rightarrow \widehat{Y}_\Gamma$ is proper and restricts to an isomorphism over the open stratum $\Gamma\backslash X/A_G$ of $\widehat{Y}_\Gamma$;
\item For $P\subset G$ parabolic with Levi factor $L$, there is a canonical closed inclusion
$$\widehat{Y}_{\Gamma_L}\hookrightarrow \widehat{Y}_{\Gamma}$$
identifying $\widehat{Y}_{\Gamma_L}$ with $\widehat{Y}_\Gamma\times_{\Gamma\backslash \mathcal{P}}(\Gamma\backslash \mathcal{P})_{/[P]}$, the closure of the $[P]$-stratum in $\widehat{Y}_{\Gamma}$.
\end{enumerate}
\end{proposition}
\begin{proof}
The first claim is obvious from the definition (and the fact that $\widehat{Y}_\Gamma$ is Hausdorff).  For the second claim, the existence of the map follows from the identification of symmetric spaces $X_L = U_P(\mathbb{R})\backslash X_G$ and the compatibility of the geodesic actions, as discussed above.  The map clearly restricts to a homeomorphism on each stratum, so to see it is an inclusion it suffices to recall the fact that if $P',P''\subset P$ are parabolic subgroups and $\gamma\in G(\mathbb{Q})$ conjugates $P'$ to $P''$, then $\gamma$ actually lies in $P(\mathbb{Q})$, see \Cref{parabolicnormalise}.
\end{proof}

Let $\partial \widehat{Y}_\Gamma= \widehat{Y}_\Gamma\smallsetminus (\Gamma\backslash X/A_G)$ denote the \emph{boundary} of the reductive Borel--Serre compactification: the complement of the open stratum. Similarly set $\partial \Gamma\backslash\widehat{X}=(\Gamma\backslash \widehat{X})\smallsetminus (\Gamma\backslash X/A_G)$. Then it follows from the above proposition that, in the category of topological spaces, we have:
\begin{enumerate}
\item $\widehat{Y}_\Gamma = (\Gamma\backslash \widehat{X}) \coprod_{\partial \Gamma\backslash \widehat{X}}\partial \widehat{Y}_\Gamma$, and
\item $\partial \widehat{Y}_\Gamma = \varinjlim_{[P]\in (\Gamma\backslash\mathcal{P})^{op}, [P]\neq [G]} \widehat{Y}_{\Gamma_{P/U_P}}$,
\end{enumerate}

\noindent giving a sense in which reductive Borel--Serre compactifications for a given group $G$ are built up from Borel--Serre compactifications together with reductive Borel--Serre compactifications for proper Levi factors of $G$.  In fact, this inductive nature of the reductive Borel--Serre compactification is very robust: these colimit diagrams turn into limit diagrams on categories of sheaves of all sorts.  This follows not from the bare statement about colimits in topological spaces, but from the more primitive \Cref{inductivetop}:

\begin{corollary}\label{inductivesheaf}
We have:
\begin{enumerate}
\item $\operatorname{Sh}(\widehat{Y}_\Gamma)\overset{\sim}{\rightarrow} \operatorname{Sh}(\Gamma\backslash \widehat{X})\times_{\operatorname{Sh}(\partial \Gamma\backslash \widehat{X})}\operatorname{Sh}(\partial\widehat{Y}_\Gamma)$, and similarly for constructible sheaves;
\item $\operatorname{Sh}(\partial \widehat{Y}_\Gamma)\overset{\sim}{\rightarrow}\varprojlim_{[P]\in \Gamma\backslash\mathcal{P}, [P]\neq [G]}\operatorname{Sh}(\widehat{Y}_{\Gamma_{P/U_P}})$ and similarly for constructible sheaves.
\end{enumerate}
\end{corollary}
\begin{proof}
Part 1 for sheaves follows from the topological cdh descent, \Cref{cdhdescenttop}.  To deduce the claim for constructible sheaves, we need to know that a sheaf on $\widehat{Y}_\Gamma$ is constructible if its pullback to the other three terms is.  This is clear because the only stratum not in the boundary is the open stratum, and the projection from $\Gamma\backslash \widehat{X}$ is an isomorphism over the open stratum.  Part 2 follows for sheaves from the descent for closed covers, \Cref{closeddescenttop}.  To deduce the claim for constructible sheaves it suffices to note that every stratum is contained in its closure which is some $\widehat{Y}_{\Gamma_{P/U_P}}$.
\end{proof}

In principle, this corollary inductively yields an identification of the exit path $\infty$-category of $\widehat{Y}_\Gamma$, based on the case of $\Gamma\backslash \widehat{X}$ treated in the previous section.  But for technical reasons we will need to make a comparison functor before we can use the inductive description to prove it's an equivalence.  To accomplish that we will describe the reductive Borel--Serre compactification in terms of stratified spaces whose exit path $\infty$-categories are equivalent to posets; on such stratified spaces the required comparison functor comes for free, see Section \ref{posetsection}, and then we deduce the correct comparison functor for $\widehat{Y}_\Gamma$ by passing to colimits.

\medskip

Since the inductive description is based on closed subsets and not open subsets, the suitable building blocks will be the closures of strata in Borel--Serre compactifications, for all Levi factors of $G$.  We therefore start with a discussion of these.

\medskip

Let $G$ be a reductive group over $\mathbb{Q}$, and let $P$ be a parabolic subgroup of $G$. Denote by
$$\widehat{X}_{\leq P} := \widehat{X}\times_{\mathcal{P}}\mathcal{P}_{/P}$$
the closure of the $P$-stratum in $\widehat{X}$, which we view as stratified over the poset $\mathcal{P}_{/P}$ of parabolic subgroups contained in $P$.  

\begin{proposition}
Let $G$ be a reductive group over $\mathbb{Q}$, let $P$ be a parabolic subgroup of $G$, and set $L=P/U_P$.  There is a natural stratum-preserving free and proper $U_P(\mathbb{R})$-action on $\widehat{X}_{\leq P}$ extending the $U_P(\mathbb{Q})$-action, and the quotient
$$U_P(\mathbb{R})\backslash \widehat{X}_{\leq P}$$
identifies with the Borel--Serre partial compactification associated to the reductive group $L$.

\medskip

More generally, if $Q$ is a parabolic subgroup containing $P$, then $U_Q(\mathbb{R})\backslash \widehat{X}_{\leq P}$ identifies with the closure of the $P/U_Q$-stratum in the Borel--Serre partial compactification associated to $Q/U_Q$.
\end{proposition}
\begin{proof}  Note that $\widehat{X}_{\leq P}$ is glued from its open subsets $\widehat{X}_{[P',P]}$ indexed by parabolic subgroups $P'\subset P$, which are in turn closed subsets of the Borel--Serre corners $\widehat{X}_{\geq P'}$, namely
$$\widehat{X}_{[P',P]}= X\times^{A_{P'}}(\mathbb{R}_{\geq 0})^{\Delta_{P'}\smallsetminus\Delta_{P}}\subset X\times^{A_{P'}}(\mathbb{R}_{\geq 0})^{\Delta_{P'}} = \widehat{X}_{\geq P'},$$
coming from the inclusion of the coordinate hypersurface $(\mathbb{R}_{\geq 0})^{\Delta_{P'}\smallsetminus\Delta_{P}}\subset (\mathbb{R}_{\geq 0})^{\Delta_{P'}}$.  We recall that the $P'(\mathbb{R})$-action on $X$ commutes with the geodesic action by $A_{P'}$.  Since 
$$U_P(\mathbb{R})\subset U_{P'}(\mathbb{R})\subset P'(\mathbb{R})$$
when $P'\subset P$, this induces compatible $U_P(\mathbb{R})$-actions on all the $\widehat{X}_{[P',P]}$, whence the required action on $\widehat{X}_{\leq P}$.  For the identification of the quotient with the Borel--Serre partial compactification associated to $\Gamma_L$, it follows from comparing the definitions using the identification $X_L=U_P(\mathbb{R})\backslash X_G$ discussed above.  The more general claim is completely analogous.
\end{proof}

\begin{corollary}\label{exitRBScorner}
The exit path $\infty$-category of $U_Q(\mathbb{R})\backslash \widehat{X}_{\leq P}$ identifies with its stratifying poset $\mathcal{P}_{\leq P}$.
\end{corollary}
\begin{proof}
By \Cref{calculatebydescent} this follows from the identification of the exit path $\infty$-category of the Borel--Serre partial compactification, \Cref{partialexit}.
\end{proof}

The following category will end up being the exit path $\infty$-category of $\widehat{Y}_\Gamma$, see \cite{OrsnesJansen}.

\begin{definition}

Let $G$ be a reductive group over $\mathbb{Q}$ and $\Gamma\subset G(\mathbb{Q})$ a neat arithmetic subgroup.  Let $\operatorname{RBS}_\Gamma$ denote the category whose objects are the parabolic subgroups $P\subset G$, and where the set of maps $P\rightarrow Q$ is the quotient
$$\{\gamma\in\Gamma : \gamma P\gamma^{-1}\subset Q\}/\Gamma_{U_P},$$
composition being induced by multiplication in $\Gamma$, which is well-defined as $\Gamma_{U_Q}\subset \Gamma_{U_{\gamma P\gamma^{-1}}}$.\defend

\end{definition}

Recall the \emph{twisted arrow category} $\operatorname{Tw}(\mathcal{C})$ of a category $\mathcal{C}$: its objects are the maps $X\rightarrow Y$ in $\mathcal{C}$, and a map $(f\colon X\rightarrow Y)\rightarrow (f'\colon X'\rightarrow Y')$ is a factorisation of the latter through the former, namely maps $a\colon X'\rightarrow X$ and $b\colon Y\rightarrow Y'$ such that $f'=bfa$.  We will identify the exit path $\infty$-category of $\widehat{Y}_\Gamma$ with $\operatorname{RBS}_\Gamma$ by expressing $\widehat{Y}_\Gamma$ as a colimit indexed by $\operatorname{Tw}(\operatorname{RBS}_\Gamma)^{op}$ of the spaces $U_Q(\mathbb{R})\backslash \widehat{X}_{\leq P}$ discussed above.  To make this work we need some lemmas.

\begin{lemma}\label{twistedcofinal}
Let $\mathcal{C}$ be a category.  The projection functor $\operatorname{Tw}(\mathcal{C})^{op}\rightarrow \mathcal{C}$ sending $x\rightarrow y$ to $x$ is a $\varinjlim$-equivalence.
\end{lemma}
\begin{proof}
The right fibre over an object $c\in \mathcal{C}$ identifies with the category whose objects are the composable maps
$$c\rightarrow x\rightarrow y$$
emanating from $c$, and whose morphisms $(c\rightarrow x\rightarrow y)\rightarrow (c\rightarrow x'\rightarrow y')$ are pairs of maps
$$(x\rightarrow x',y'\rightarrow y)$$
making the evident diagram commute.  The full subcategory on those objects for which $c\overset{=}{\rightarrow} x$ is equivalent to $(\mathcal{C}_{c/})^{op}$ and is therefore contractible; but on the other hand there is a retraction to the inclusion of this full subcategory given by composition, $c\rightarrow x\rightarrow y\mapsto (c=c\rightarrow y)$, and an obvious natural transformation from this retraction to the identity.  Thus the right fibre is homotopy equivalent to $(\mathcal{C}_{c/})^{op}$ and is therefore also contractible, hence by \Cref{lim-equivalences} our functor is a $\varinjlim$-equivalence as desired.
\end{proof}

\begin{lemma}\label{twistedRBS}
The category $\operatorname{Tw}(\operatorname{RBS}_\Gamma)^{op}$ is equivalent to the category whose objects are the inclusions $P\subset Q$ of parabolic subgroups, and whose set of maps $(P\subset Q)\rightarrow (P'\subset Q')$ is given by
$$\Gamma_{U_{Q'}}\backslash \{\gamma\in \Gamma: \gamma P\gamma^{-1}\subset P', \gamma Q\gamma^{-1}\supset Q'\}$$
with composition induced by multiplication in $\Gamma$.

\medskip

More precisely, the equivalence is given by the functor which on objects sends the inclusion $P\subset Q$ to the map $P\overset{[id]}{\rightarrow} Q$ in $\operatorname{RBS}_{\Gamma}$, and on maps sends $[\gamma]\colon(P\subset Q)\rightarrow (P'\subset Q')$ to the pair of maps $P\overset{[\gamma]}{\rightarrow} P'$ and $Q'\overset{[\gamma^{-1}]}{\rightarrow} Q$ in $\operatorname{RBS}$.
\end{lemma}
\begin{proof}
The functor is well-defined because $\Gamma_{U_{Q'}}\subset \Gamma_{U_{P'}}\subset \Gamma_{U_{\gamma P\gamma^{-1}}}$ so that $\Gamma_{U_{Q'}}\gamma \subset \gamma \Gamma_{U_P}$.  To give the identification, we should show that every map $P\rightarrow Q$ in $\operatorname{RBS}_\Gamma$ is equivalent to an inclusion, meaning a map given by $\gamma=id$, and that maps in $\operatorname{Tw}(\operatorname{RBS}_\Gamma)^{op}$ between inclusions are given by the set posited in the statement.

\medskip

For the first claim, if $P\rightarrow Q$ is induced by $\gamma\in \Gamma$ then we can factor it as $P\overset{\sim}{\rightarrow} \gamma P\gamma^{-1}\overset{\subset}\rightarrow Q$ where the first map is induced by $\gamma$ and the second map is an inclusion.  For the second claim, maps $(P\subset Q)\rightarrow (P'\subset Q')$ in $\operatorname{Tw}(\operatorname{RBS}_\Gamma)^{op}$ are by definition given by
$$\{\gamma_a,\gamma_b\in \Gamma: \gamma_aP\gamma_a^{-1}\subset P', \gamma_bQ'\gamma_b^{-1}\subset Q, \gamma_b\gamma_a\in \Gamma_{U_P}\}/\sim$$
where $(\gamma_a,\gamma_b)\sim (\rho_a,\rho_b)$ iff $\gamma_a\Gamma_{U_P}=\rho_a\Gamma_{U_P}$ and $\gamma_b\Gamma_{U_{Q'}}=\rho_b\Gamma_{U_{Q'}}$.  It follows that we can uniquely specify $\gamma_a$ in terms of $\gamma_b$ by setting $\gamma_a=\gamma_b^{-1}$, which leads to the claim.
\end{proof}

\begin{lemma}\label{RBSlocallyaposet}
For $P$ a parabolic subgroup, the natural functor $\mathcal{P}_{/P}\rightarrow (\operatorname{RBS}_\Gamma)_{/P}$ induced by $(P'\subset P) \mapsto (P'\overset{[id]}{\rightarrow}P)$ is an equivalence.
\end{lemma}
\begin{proof}
This follows similarly: by direct calculation, the functor is essentially surjective and fully faithful.
\end{proof}

In terms of the equivalence of \Cref{twistedRBS}, we find that the spaces $U_Q(\mathbb{R})\backslash \widehat{X}_{\leq P}$ organise into a functor
 $$\widehat{Y}\colon \operatorname{Tw}(\operatorname{RBS}_\Gamma)^{op}\rightarrow \operatorname{Top}$$
which sends $(P\subset Q) \mapsto U_Q(\mathbb{R})\backslash \widehat{X}_{\leq P}$ on objects, and on morphisms is induced by the $\Gamma$-action on $\widehat{X}$.  This lies above the combinatorial analogue $\operatorname{Tw}(\operatorname{RBS}_\Gamma)^{op}\rightarrow\operatorname{Posets}$ defined by
$$(P\subset Q)\mapsto \mathcal{P}_{\leq P}$$
on objects, and induced by $\Gamma$-conjugation on maps.  Thus we promote $\widehat{Y}$ to a functor from $\operatorname{Tw}(\operatorname{RBS}_\Gamma)^{op}$ to stratified spaces.

\medskip

A rephrasing of Zucker's definition of the reductive Borel--Serre compactification is that
$$\widehat{Y}_\Gamma = \varinjlim \widehat{Y}.$$
Indeed, in both cases we are gluing together $\Gamma$-orbits and collapsing unipotent fibres, we just do it in a different order.  But this new description is more robust, in that it promotes to a statement about categories of sheaves:

\begin{theorem}\label{RBSdescent}
Let $G$ be a reductive group over $\mathbb{Q}$ and $\Gamma\subset G(\mathbb{Q})$ a neat arithmetic subgroup.  Then 
$$\operatorname{Sh}(\widehat{Y}_\Gamma)\overset{\sim}{\rightarrow} \varprojlim_{(P\subset Q)\in \operatorname{Tw}(\operatorname{RBS}_\Gamma)}\operatorname{Sh}(U_Q(\mathbb{R})\backslash \widehat{X}_{\leq P}),$$
and similarly for constructible sheaves.
\end{theorem}

We will prove this theorem shortly, but for now let us deduce the following consequence.

\begin{corollary}\label{RBSexit}
The stratified space $\widehat{Y}_\Gamma$ admits an exit path $\infty$-category, and this exit path $\infty$-category identifies with the category $\operatorname{RBS}_\Gamma$.
\end{corollary}
\begin{proof}
By \Cref{exitRBScorner}, the exit path $\infty$-category of $U_Q(\mathbb{R})\backslash \widehat{X}_{\leq P}$ identifies with its stratifying poset $\mathcal{P}_{/P}$.  Thus, by \Cref{calculatebydescent}, it suffices to calculate that in $\operatorname{Cat}_\infty$ we have
$$\varinjlim_{(P\subset Q)\in \operatorname{Tw}(\operatorname{RBS}_\Gamma)^{op}} \mathcal{P}_{/P}\simeq \operatorname{RBS}_\Gamma.$$
However, for a parabolic subgroup $P$ the natural functor $\mathcal{P}_{/P}\rightarrow (\operatorname{RBS}_\Gamma)_{/P}$ sending $P'\subset P$ to $[id]\colon P'\rightarrow P$ is an equivalence by \Cref{RBSlocallyaposet}.  Thus this will follow from the more general claim that for any category $\mathcal{C}$ we have
$$\varinjlim_{(c\rightarrow d)\in \operatorname{Tw}(\mathcal{C})^{op}}\mathcal{C}_{/c}\overset{\sim}{\rightarrow}\mathcal{C}.$$
But we have $\varinjlim_{c\in\mathcal{C}} \mathcal{C}_{/c}\overset{\sim}{\rightarrow}\mathcal{C}$ by \Cref{tautologicalbutuseful}, so this follows from \Cref{twistedcofinal}. \end{proof}

To prove \Cref{RBSdescent}, we need the following diagrammatic analogue of the inductive structure of the reductive Borel--Serre compactification.

\begin{proposition}
Let $G$ be a reductive group over $\mathbb{Q}$ and $\Gamma\subset G(\mathbb{Q})$ a subgroup.  Consider the following full subcategories of $\operatorname{Tw}(\operatorname{RBS}_\Gamma)^{op}$:
\begin{enumerate}
\item For a $\Gamma$-conjugacy class of parabolic subgroups $[P]$ of $G$, we write 
$$\operatorname{Tw}(\operatorname{RBS}_\Gamma)^{op}_{\leq [P]}$$
for the full subcategory of $\operatorname{Tw}(\operatorname{RBS}_\Gamma)^{op}$ on those inclusions $P'\subset Q'$ with $P'$ contained in a representative of $[P]$.
\item For a parabolic subgroup $P$, we view
$$\operatorname{Tw}(\operatorname{RBS}_{\Gamma_{P/U_P}})^{op}$$
as the further full subcategory of those $P'\subset Q'$ with $Q'\subset P$.
\item We view $\Gamma\backslash\backslash \mathcal{P}$ as a full subcategory of $\operatorname{Tw}(\operatorname{RBS}_\Gamma)^{op}$ by the embedding $P\mapsto (P\subset G)$.
\item For $P\subset Q$, view $B\Gamma_{P/U_Q}$ as the full subcategory spanned by $(P\subset Q)$.
\end{enumerate}
Then:
\begin{enumerate}
\item The full subcategories $\operatorname{Tw}(\operatorname{RBS}_\Gamma)^{op}_{\leq [P]}$  and $\Gamma\backslash\backslash \mathcal{P}$ are left-closed.
\item The union of the subcategories $\operatorname{Tw}(\operatorname{RBS}_\Gamma)^{op}_{\leq [P]}$ for $[P]\neq [G]$ is equal to the complement $\operatorname{Tw}(\operatorname{RBS}_\Gamma)^{op}\smallsetminus B\Gamma$, and for all $(P'\subset Q')\in \operatorname{Tw}(\operatorname{RBS}_\Gamma)^{op}\smallsetminus B\Gamma$ the collection of those $\operatorname{Tw}(\operatorname{RBS}_\Gamma)^{op}_{\leq [P]}$ containing $P'\subset Q'$ has a minimal element, namely $\operatorname{Tw}(\operatorname{RBS}_\Gamma)^{op}_{\leq [P']}$.
\item For a parabolic subgroup $P\subset G$, the inclusion
$$B\Gamma_P\subset (\Gamma\backslash\backslash\mathcal{P})_{\leq [P]}=(\Gamma\backslash\backslash\mathcal{P})\cap\operatorname{Tw}(\operatorname{RBS}_\Gamma)^{op}_{\leq [P]}$$
has a left adjoint.
\item For a parabolic subgroup $P\subset G$, the inclusion
$$\operatorname{Tw}(\operatorname{RBS}_{\Gamma_{P/U_P}})^{op}\subset \operatorname{Tw}(\operatorname{RBS}_\Gamma)^{op}_{\leq [P]}$$
has a left adjoint.
\end{enumerate}
\end{proposition}
\begin{proof}
Claims 1 and 2 are obvious.  For claim 3, note that $(\Gamma\backslash\backslash\mathcal{P})_{\leq [P]}$ is equivalent to its full subcategory on those parabolic $P'$ such that $P'\subset P$.  Recall that any $\gamma\in \Gamma$ which conjugates such a $P'\subset P$ back inside $P$ must necessarily lie in $\Gamma_P$, see \Cref{parabolicnormalise}.  Thus we can identify $(\Gamma\backslash\backslash\mathcal{P})_{\leq [P]}\simeq \Gamma_P\backslash\backslash \mathcal{P}_{/P}$, and the projection to $B\Gamma_P$ provides a left adjoint proving the claim.  Claim 4 follows similarly: we can replace $\operatorname{Tw}(\operatorname{RBS}_\Gamma)^{op}_{\leq [P]}$ with the equivalent full subcategory of those $(P'\subset Q')$ such that $P'\subset P$, and then the functor backwards given by $(P'\subset Q')\mapsto (P'\subset P\cap Q')$ on objects and $[\gamma]\mapsto [\gamma]$ on morphisms provides a left adjoint to the inclusion.
\end{proof}

\begin{corollary}
Let $F\colon\operatorname{Tw}(\operatorname{RBS}_\Gamma)\rightarrow \mathcal{C}$ be a functor to an arbitrary $\infty$-category $\mathcal{C}$ with all limits.  Then:
\begin{enumerate}
\item $\varprojlim F\overset{\sim}{\rightarrow} \varprojlim F\vert_{\operatorname{Tw}(\operatorname{RBS}_\Gamma)\smallsetminus B\Gamma}\times_{\varprojlim F\vert_{(\Gamma\backslash\backslash\mathcal{P})^{op}\smallsetminus B\Gamma}} \varprojlim F\vert_{(\Gamma\backslash\backslash\mathcal{P})^{op}}$.
\item $\varprojlim F\vert_{\operatorname{Tw}(\operatorname{RBS}_\Gamma)\smallsetminus B\Gamma}\overset{\sim}{\rightarrow} \varprojlim_{[P]\in(\Gamma\backslash \mathcal{P})^{op}\smallsetminus [G]} \varprojlim F\vert_{\operatorname{Tw}(\operatorname{RBS}_\Gamma)_{\leq [P]}},$
and
$$\varprojlim F\vert_{\operatorname{Tw}(\operatorname{RBS}_\Gamma)_{\leq [P]}} \overset{\sim}{\rightarrow} \varprojlim F\vert_{\operatorname{Tw}(\operatorname{RBS}_{\Gamma_{P/U_P}})}.$$
\item $\varprojlim F\vert_{(\Gamma\backslash\backslash\mathcal{P})^{op}\smallsetminus B\Gamma}\overset{\sim}{\rightarrow} \varprojlim_{[P]\in(\Gamma\backslash \mathcal{P})^{op}\smallsetminus [G]} \varprojlim F\vert_{((\Gamma\backslash\backslash\mathcal{P})_{\leq [P]})^{op}}$, and for all parabolic $P\subset G$ we have
$$\varprojlim F\vert_{((\Gamma\backslash\backslash\mathcal{P})_{\leq [P]})^{op}}\overset{\sim}{\rightarrow} \varprojlim F\vert_{B\Gamma_P}.$$
\end{enumerate}
\end{corollary}
\begin{proof}
This follows by descent for closed covers, \Cref{closeddescentcat} which lets one decompose limits by \Cref{decomposelimits}, and the fact that a left adjoint functor is a $\varprojlim$-equivalence, \Cref{finalexample}.
\end{proof}

Now we can prove \Cref{RBSdescent}.

\begin{proof}
Consider the functor $F\colon\operatorname{Tw}(\operatorname{RBS}_\Gamma)\rightarrow \operatorname{Cat}_\infty$ defined by
$$F = \operatorname{Sh}\circ \widehat{Y},$$
so $F(P\subset Q)=\operatorname{Sh}(U_Q(\mathbb{R})\backslash \widehat{X}_{\leq P})$ with pullback functoriality. For any full subcategory $\mathcal{D}\subset \operatorname{Tw}(\operatorname{RBS}_\Gamma)$, we have the associated comparison map
$$\operatorname{Sh}(\varinjlim \widehat{Y}\vert_{\mathcal{D}^{op}})\rightarrow \varprojlim F\vert_{\mathcal{D}}.$$
We want to prove that this is an equivalence for $\mathcal{D}=\operatorname{Tw}(\operatorname{RBS}_\Gamma)$.  Proceeding by induction on the $\mathbb{Q}$-rank of $G$, we can assume it is an equivalence for $\mathcal{D}=\operatorname{Tw}(\operatorname{RBS}_{\Gamma_{P/U_P}})$ for any proper parabolic $P\subset G$.  Then comparing part 2 of the corollary above with part 2 from \Cref{inductivesheaf} we get that it is an equivalence for $\mathcal{D}=\operatorname{Tw}(\operatorname{RBS}_{\Gamma})\smallsetminus B\Gamma$; on the other hand part 3 of the corollary above plus descent for the closed cover of $\partial \Gamma\backslash \widehat{X}$ by the $(\Gamma\backslash \widehat{X})_{\leq [P]}$ for $[P]\neq [G]$ shows that it is an equivalence for $\mathcal{D}=(\Gamma\backslash\backslash\mathcal{P})^{op}\smallsetminus B\Gamma$.  Then comparing part 1 of the corollary with part 1 of \Cref{inductivesheaf} gives the desired claim, finishing the proof of \Cref{RBSdescent} for sheaves, without the constructibility condition.  But since the maps from the strata of the Borel--Serre compactification to those of the reductive Borel--Serre compactification are fibre bundles and hence have local sections, a sheaf on $\widehat{Y}_\Gamma$ is constructible if and only if its pullback to $\Gamma\backslash\widehat{X}$ is constructible, which shows that the variant with constructible sheaves follows.\end{proof}

\section{\texorpdfstring{$\operatorname{RBS}(M)$}{RBS(M)} as unstable algebraic K-theory}

We now turn to our main goal of describing an unstable algebraic K-theory space.  As an intermediary step, let us transport some of the above discussion into the general context of reductive groups over commutative rings.  We recall the definition from \cite{OrsnesJansen}.

\begin{definition}
Let $G$ be a reductive group over a commutative ring $R$.  Define the category $\operatorname{RBS}_G$  to have objects the parabolic subgroups $P\subset G$ and morphisms $P\rightarrow P'$ the set
$$\{g\in G(R):gPg^{-1}\subset P'\}/U_P(R),$$
composition being induced by multiplication in $G(R)$.
\defend
\end{definition}

Thus, if we take $R=\mathbb{Q}$ and further restrict to the subcategory specified by the choice of an arithmetic subgroup $\Gamma\subset G(\mathbb{Q})$, this recovers the category $\operatorname{RBS}_\Gamma$ of the previous section, which we identified with the exit path $\infty$-category of the reductive Borel--Serre compactification $\widehat{Y}_\Gamma$.  But now we want to consider general $R$ and forget about $\Gamma$.

\medskip

For connections to algebraic K-theory, we restrict to $G=GL_n$, or more generally $G=GL(M)$ for a finitely generated projective $R$-module $M$.  For the classification of parabolic subgroups of reductive group schemes we refer to \cite{SGA} Expos\'{e} XXVI.  We see that if $\operatorname{Spec}(R)$ is connected, then parabolic subgroups of $GL(M)$ correspond to splittable flags $\mathcal{F}$ of submodules of $M$: chains of inclusions
$$\mathcal{F}=(M_1\subsetneq \ldots \subsetneq M_{d-1})$$
such that each quotient $M_i/M_{i-1}$ is nonzero and projective (we set $M_0=0$ and $M_d=M$).  We call $d$ the \emph{length} of the flag; it is the number of these associated graded pieces $M_i/M_{i-1}$.  The corresponding parabolic subgroup $P_\mathcal{F}$ represents the automorphisms $g$ of $M$ preserving the flag, meaning $g(M_i)=M_i$ for all $i$, and its unipotent radical $U_\mathcal{F}\subset P_\mathcal{F}$ represents those automorphisms preserving the flag and inducing the identity on each $M_i/M_{i-1}$.  Note that the Levi factor $L_\mathcal{F}=P_\mathcal{F}/U_\mathcal{F}$ identifies with the product
$$L_{\mathcal{F}} = \prod_{i=1}^d GL(M_i/M_{i-1}).$$
Furthermore, the partial order of inclusion of parabolic subgroups translates into the partial order of refinement of flags, defined by setting $\mathcal{F} \leq \mathcal{G}$ if and only if the set of submodules occurring in $\mathcal{G}$ is a subset of the set of submodules occurring in $\mathcal{F}$.  Note that the inclusion on unipotent radicals goes the opposite direction: if $\mathcal{F}\leq \mathcal{G}$ then while $P_\mathcal{F}\subset P_{\mathcal{G}}$, we have $U_{\mathcal{F}}\supset U_{\mathcal{G}}$.

\medskip

This discussion of flags $\mathcal{F}$ of splittable submodules of $M$ and their associated subgroups $P_\mathcal{F}\subset GL(M)$ quotients $L_{\mathcal{F}}=P_\mathcal{F}/U_\mathcal{F}$ makes no use of the commutativity of $R$.  Thus we arrive at the following.

\begin{definition}\label{RBS(M)definition}
Let $A$ be an associative ring and $M$ a finitely generated projective $A$-module. Define the category $\operatorname{RBS}(M)$ to have objects the splittable flags $\mathcal{F}$ of submodules of $M$, with set of maps $\mathcal{F}\rightarrow \mathcal{F}'$ given by
$$\{g\in GL(M): g \mathcal{F}\leq \mathcal{F}'\}/U_\mathcal{F},$$
with composition induced by multiplication in $GL(M)$ (it is well-defined because $\mathcal{F}\leq \mathcal{G}$ implies that $U_\mathcal{G}\subset U_\mathcal{F}$).

\medskip

Let also $\mathcal{P}$ denote the poset of splittable flags of submodules of $M$ with partial order $\leq$ given by refinement, as above.\defend
\end{definition}

First we show that $\operatorname{RBS}(M)$ ``behaves'' like an exit path $\infty$-category with stratifying poset $GL(M)\backslash \mathcal{P}$ and $K(\pi,1)$ strata.

\begin{lemma}
Let $A$ be an associative ring and $M$ a finitely generated projective $A$-module.  Then:
\begin{enumerate}
\item The quotient set $GL(M)\backslash \mathcal{P}$ inherits the poset structure from $\mathcal{P}$.
\item There is a functor
$$\pi\colon\operatorname{RBS}(M)\rightarrow GL(M)\backslash\mathcal{P}$$
given by $\pi(\mathcal{F})=[\mathcal{F}]$.
\item For a point $x\in GL(M)\backslash\mathcal{P}$ the fibre $\pi^{-1}(x)$ is a connected groupoid.
\item For a splittable flag $\mathcal{F}$, the automorphism group of $\mathcal{F}$ in $\operatorname{RBS}(M)$ identifies as
$$\operatorname{Aut}_{\operatorname{RBS}(M)}(\mathcal{F}) = P_{\mathcal{F}}/U_\mathcal{F} = L_{\mathcal{F}}.$$
\end{enumerate}
\end{lemma}
\begin{proof}
For part 1, by \Cref{posetquotient} we need to check that $g\mathcal{F}\leq \mathcal{F}$ implies $g\mathcal{F}=\mathcal{F}$.  But the $GL(M)$-action preserves length, and a refinement between flags of equal length is necessarily an identity.  Then parts 2, 3 and 4 are immediate.
\end{proof}

Now we describe the basic inductive structure of $\operatorname{RBS}(M)$.  For this we will impose the split noetherian hypothesis on $M$ described in the introduction: that there are no infinite ascending sequences of split submodules of $M$.  This has the following consequence.

\begin{lemma}\label{noetherianlemma}
Let $A$ be a ring and $M$ a split noetherian finitely generated projective $A$-module.  Then:
\begin{enumerate}
\item The posets $\mathcal{P}$ and $GL(M)\backslash \mathcal{P}$ satisfy the descending chain condition.
\item If $N\subset M$ is a split submodule and we have $g\in GL(M)$ with $gN\subset N$, then $gN=N$.
\end{enumerate}
\end{lemma}
\begin{proof}
For part 1, it suffices to show that just $\mathcal{P}$ satisfies the descending chain condition.  If not, we would get an infinite chain of split submodules of $M$, so there would either be an infinite descending sequence of split submodules or an infinite ascending sequence.  But as the submodules are split we can convert one situation to the other so both are ruled out by our split noetherian hypothesis.  For 2, if $gN\subsetneq N$ then we get the infinite chain
$$\ldots \subsetneq g^kN\subsetneq g^{k-1}N \subsetneq \ldots \subsetneq N$$ which contradicts our assumption.
\end{proof}

Furthermore, it is often the case that every $M$ is split noetherian.

\begin{lemma}\label{whennoetherian}
Let $A$ be a ring.  If either:
\begin{enumerate}
\item $A$ is noetherian, or
\item $A$ is commutative and $\operatorname{Spec}(A)$ has only finitely many connected components,
\end{enumerate}
then every finitely generated projective $A$-module $M$ is split noetherian.
\end{lemma}
\begin{proof}
The claim is clear if $A$ is noetherian.  If $A$ is commutative, then the dimension function $x\mapsto \operatorname{dim}_{k(x)}(M\otimes_{A} k(x))$ on $\operatorname{Spec}(A)$ is locally constant as every finitely generated projective module is locally free; moreover if $N$ is a proper split submodule of $M$ then the dimension of $N$ must be strictly less than that of $M$ on at least one connected component, as otherwise $M/N$ would be a finitely generated projective module of dimension $0$ everywhere, whence $M/N=0$ by Nakayama's lemma.  Thus the claim reduces to the fact that a finite product of copies of the poset $(\mathbb{N},\leq)$ satisfies the descending chain condition, which is clear.
\end{proof}

\begin{definition}
Let $A$ be an associative ring and $M$ a finitely generated projective $A$-module. For $[\mathcal{F}]\in GL(M)\backslash \mathcal{P}$ an orbit of splittable flags, denote by $\operatorname{RBS}(M)_{\leq [\mathcal{F}]}$ the full subcategory of those objects which admit a map to $\mathcal{F}$, and similarly for $(GL(M)\backslash\backslash\mathcal{P})_{\leq [\mathcal{F}]}$.\defend
\end{definition}

The following will be the basis of many inductive arguments.

\begin{proposition}\label{inductive structure}
Let $A$ be an associative ring and $M$ a split noetherian finitely generated projective $A$-module.
\begin{enumerate}
\item For a splittable flag $\mathcal{F}$ in $M$ with associated graded $gr(\mathcal{F})=(M_1,\ldots,M_d)$  we have identifications
$$\operatorname{RBS}(M)_{\leq [\mathcal{F}]}=\prod_{i=1}^d \operatorname{RBS}(M_i)$$
and
$$(GL(M)\backslash\backslash\mathcal{P})_{\leq [\mathcal{F}]} = P_\mathcal{F}\backslash\backslash\mathcal{P}_{\leq \mathcal{F}},$$
where $P_\mathcal{F}\subset GL(M)$ is the stabiliser group of $\mathcal{F}$. Moreover, the natural functor $BP_\mathcal{F}\rightarrow P_\mathcal{F}\backslash\backslash \mathcal{P}_{\leq \mathcal{F}}$ is a right adjoint and in particular induces an isomorphism on anima.
\item The natural functor $p\colon GL(M)\backslash\backslash\mathcal{P}\rightarrow\operatorname{RBS}(M)$ is proper and an isomorphism over $BGL(M)$, the full subcategory spanned by the empty flag.
\end{enumerate}
\end{proposition}
\begin{proof}
For claim 1, note that every splittable flag $\mathcal{G}$ with a map to $\mathcal{F}$ is $GL(M)$-equivalent to a splittable flag with $\mathcal{G}\leq \mathcal{F}$, so one can replace the left-hand categories by their full subcategories on such flags.  This lets one match up the objects, and then one has to calculate maps, where one needs the claim that if $g\in GL(M)$ satisfies $g\mathcal{G}\leq \mathcal{F}$, then necessarily $g\in P_\mathcal{F}$.\footnote{This is the analogue of the key lemma about parabolic subgroups, \Cref{parabolicnormalise}.}  But this follows from \Cref{noetherianlemma} part 2. The last claim, about $BP_\mathcal{F}\rightarrow P_\mathcal{F}\backslash\backslash \mathcal{P}_{\leq \mathcal{F}}$ being a right adjoint, is immediate to verify by taking the left adjoint to be the projection map backwards.

\medskip
For claim 2, note that $GL(M)\backslash\backslash\mathcal{P}$ identifies with the left pullback
$$\operatorname{RBS}(M)\overset{\rightarrow}{\times}_{\operatorname{RBS}(M)}BGL(M).$$
Indeed, $\mathcal{P}$ identifies with $\operatorname{RBS}(M)_{/\emptyset}$ by sending $\mathcal{F}$ to the map $\mathcal{F}\rightarrow \emptyset$ given by the identity $e\in GL(M)$; then when we factor in the automorphisms of $\emptyset$ we get the claim.
\end{proof}

In terms of colimits in $\operatorname{Cat}_\infty$, or colimits in $\mathcal{S}$ after applying geometric realisation, we have the following.

\begin{corollary}\label{inductiveRBScat}
Let $A$ be a ring and $M$ a split noetherian finitely generated projective $A$-module.  There are the following colimits in $\operatorname{Cat}_\infty$:
\begin{enumerate}
\item $$GL(M)\backslash\backslash\mathcal{P}\sqcup_{GL(M)\backslash\backslash\mathcal{P}\smallsetminus BGL(M)} (\operatorname{RBS}(M) \smallsetminus BGL(M))\overset{\sim}{\rightarrow} \operatorname{RBS}(M).$$
\item $$\operatorname{RBS}(M) \smallsetminus BGL(M) = \varinjlim_{[\mathcal{F}]\in GL(M)\backslash\mathcal{P}, [\mathcal{F}]\neq [\emptyset]} \operatorname{RBS}(M)_{\leq [\mathcal{F}]}$$
and
$$(GL(M)\backslash\backslash\mathcal{P})\smallsetminus BGL(M) = \varinjlim_{[\mathcal{F}]\in GL(M)\backslash\mathcal{P}, [\mathcal{F}]\neq [\emptyset]} (GL(M)\backslash\backslash\mathcal{P})_{\leq [\mathcal{F}]}.$$
\end{enumerate}
\end{corollary}
\begin{proof}
This follows from the cdh descent, \Cref{cdhdescentcat}, and descent for covers, \Cref{closeddescentcat}.
\end{proof}

There is a natural comparison map
$$BGL(M)\rightarrow \vert\operatorname{RBS}(M)\vert$$
coming from the empty flag, and we want to understand the extent to which this is an equivalence.  First of all, it is clear that $\operatorname{RBS}(M)$ is connected, as every flag maps to the empty flag. Thus the first thing to look at is $\pi_1$.  This turns out to not be so difficult to analyse.

\begin{theorem}\label{pi1}
Let $A$ be an associative ring and $M$ a split noetherian finitely generated projective $A$-module.  Denote by $E(M)\subset GL(M)$ the subgroup generated by the $U_\mathcal{F}$ as $\mathcal{F}$ runs through all splittable flags in $M$.  Then the map
$$GL(M)=\pi_1 BGL(M)\rightarrow \pi_1\vert \operatorname{RBS}(M)\vert$$
is surjective with kernel $E(M)$.
\end{theorem}
\begin{proof}
First let us note that $E(M)$ is in the kernel.  Indeed, if $\mathcal{F}$ is a splittable flag and $g\in U_\mathcal{F}$, then the refinement $\mathcal{F}\leq \emptyset$ is invariant under the $g$ action on $\emptyset$, which produces a nullhomotopy of the image of $g$ in $\pi_1\vert\operatorname{RBS}(M)\vert$.

\medskip

Next let us produce a map $\pi_1\vert \operatorname{RBS}(M)\vert \rightarrow GL(M)/E(M)$ such that the composition with $GL(M)\rightarrow \pi_1\vert\operatorname{RBS}(M)\vert$ is the natural quotient.  For this define a functor
$$\operatorname{RBS}(M)\rightarrow B(GL(M)/E(M))$$
by sending each flag to the basepoint and the map $\mathcal{F}\rightarrow\mathcal{F}'$ induced by an element $g\in GL(M)$ with $g\mathcal{F}\leq \mathcal{F}'$ to the image of $g$ in $GL(M)/E(M)$.  This is clearly well-defined and functorial.

\medskip

To finish the proof, it suffices to show that the map $BGL(M)\rightarrow \vert\operatorname{RBS}(M)\vert$ is surjective on $\pi_1$.  For this, recall that a map of anima $X\rightarrow Y$ is an isomorphism on $\pi_0$ and surjective on $\pi_1$ if and only if it is left orthogonal to the class of $0$-truncated maps, meaning those maps each of whose homotopy fibres is $0$-truncated.  It follows that the collection of maps $X\rightarrow Y$ which are isomorphism on $\pi_0$ and surjective on $\pi_1$ is closed under colimits in $\operatorname{Fun}(\Delta^1,\mathcal{S})$.  It is also clearly closed under products and composition.

\medskip

Using these permanence properties, let us now prove the claim by noetherian induction on $M$.  Thus, we can assume the claim holds for all proper splittable submodules of $M$, hence it holds for all associated graded pieces of nonempty flags in $M$.  But then part 1 of \Cref{inductive structure} shows that
$$BL_\mathcal{F} \rightarrow \vert\operatorname{RBS}(M)_{\leq [\mathcal{F}]}\vert$$
is an isomorphism on $\pi_0$ and surjective on $\pi_1$.  It follows that the same is true for the composition $BP_\mathcal{F}\rightarrow BL_\mathcal{F} \rightarrow \vert \operatorname{RBS}(M)_{\leq [\mathcal{F}]}\vert$, which is equivalent to saying that the same is true for
$$\vert (GL(M)\backslash\backslash\mathcal{P})_{\leq [\mathcal{F}]}\vert \rightarrow \vert \operatorname{RBS}(M)_{\leq [\mathcal{F}]}\vert,$$
since $(GL(M)\backslash\backslash\mathcal{P})_{\leq [\mathcal{F}]}=P_\mathcal{F}\backslash\backslash\mathcal{P}_{\leq \mathcal{F}}$ by part 2 of \Cref{inductive structure}.

\medskip

But parts 1 and 2 of \Cref{inductiveRBScat} show that our map $BGL(M)\rightarrow\vert\operatorname{RBS}(M)\vert$ is an iterated colimit of such maps, so we deduce the desired claim.
\end{proof}

In particular, if the subgroup $E(M)$ happens to be perfect, we can perform the plus construction and obtain a comparison map
$$BGL(M)^+\rightarrow \vert \operatorname{RBS}(M)\vert.$$
which is an isomorphism on $\pi_0$ and $\pi_1$.  In the next section we will see that if $A$ \emph{has many (central) units} in the technical sense introduced by Nesterenko--Suslin, and every finitely generated projective $A$-module is free, then this map is an equivalence.  For the proof we use the inductive structure explained in this section to reduce to proving a certain homology isomorphism for matrix groups.  This is a close analogue to the homology isomorphism proved by Nesterenko--Suslin in \cite{NS}, and our proof is based on theirs.  We do have to take care to ensure that we get the desired result with local coefficient systems as well, though.

\subsection{Comparison with the plus-construction}

\begin{lemma}\label{fromflagstoRBS(M)}
Let $A$ be an associative ring and $M$ a split noetherian finitely generated projective $A$-module.  Let $\mathcal{L}$ be a local system of abelian groups on $\vert \operatorname{RBS}(M)\vert $, viewed also as a local system on $BL_\mathcal{F}$ for any splittable flag $\mathcal{F}$ of submodules of $M$,  by pullback to the full subcategory on $\mathcal{F}$.  Suppose that for all $\mathcal{F}\leq \mathcal{G}$ the quotient map $BP_\mathcal{F}\rightarrow B(P_\mathcal{F}/U_\mathcal{G})$ induces an isomorphism on homology with $\mathcal{L}$ coefficients.  Then the map
$$BGL(M)\rightarrow \vert \operatorname{RBS}(M)\vert$$
also induces an isomorphism on homology with $\mathcal{L}$-coefficients.
\end{lemma}
\begin{proof}
Using the inductive nature of the $\operatorname{RBS}$ categories, \Cref{inductive structure} and \Cref{inductiveRBScat}, we will prove by noetherian induction on a splittable flag $\mathcal{G}$ that the map
$$BL_\mathcal{G}\rightarrow \vert\operatorname{RBS}(M)_{\leq [\mathcal{G}]}\vert$$
is an isomorphism on homology with $\mathcal{L}$-coefficients.  Thus, assume the claim holds for all finer flags.  Let $(M_1,\dots,M_d)$ denote the associated graded of $\mathcal{G}$, and consider the proper functor
$$\prod_{i=1}^d GL(M_i)\backslash\backslash \mathcal{P}_i\rightarrow \prod_{i=1}^d \operatorname{RBS}(M_i) = \operatorname{RBS}(M)_{\leq [\mathcal{G}]},$$
where $\mathcal{P}_i$ denotes the poset of splittable flags in $M_i$.  We want to show it's an isomorphism on $\mathcal{L}$-homology.  By cdh descent, it suffices to show the same for its pullback to $\operatorname{RBS}(M)_{\leq [\mathcal{F}]}$ for any finer flag $\mathcal{F}\leq \mathcal{G}$.  But if we write $\mathcal{F}_i$ for the image of $\mathcal{F}$ in $M_i$, this pullback gives
$$\prod_{i=1}^d P_{\mathcal{F}_i}\backslash\backslash (\mathcal{P}_i)_{\leq [\mathcal{F}_i]}\rightarrow \operatorname{RBS}(M)_{\leq [\mathcal{F}]},$$
so it suffices to see that $B(P_\mathcal{F}/U_\mathcal{G})=\prod_{i=1}^d BP_{\mathcal{F}_i}\rightarrow \operatorname{RBS}(M)_{\leq [\mathcal{F}]}$ gives an isomorphism on $\mathcal{L}$-homology.  But now this follows from the inductive hypothesis, our hypothesis, and the 2 out of 3 property for isomorphisms.
\end{proof}

\begin{lemma}\label{vanishingtrick}
Let $k$ be a prime field, let $1\rightarrow U\rightarrow P\rightarrow L\rightarrow 1$ be a short exact sequence of groups and let $\mathcal{L}$ be a local system of $k$-modules on $BL$.  Suppose there exist:
\begin{enumerate}
\item A normal subgroup $D\subset L$;
\item A map $s\colon D\rightarrow P$ giving a splitting of the pullback of $P\rightarrow L$ to $D$;
\end{enumerate}
such that:
\begin{enumerate}
\item The local system $\mathcal{L}$ is constant when restricted to $BD$;
\item For all $i\geq 1$ the $k$-module $H_i(BU;k)$, equipped with $D$-action induced by the conjugation action of $s(D)$ on $U$, has vanishing $D$-homology in all degrees.
\end{enumerate}
Then the map
$$BP\rightarrow BL$$
induces an isomorphism on homology with $\mathcal{L}$-coefficients.
\end{lemma}
\begin{proof}
By the Serre spectral sequence, it suffices to show that $H_p(BL;H_q(BU;\mathcal{L}))=0$ for $p\geq 0$ and $q\geq 1$.  By the Serre spectral sequence for $BD \rightarrow BL\rightarrow B(L/D)$, for this it suffices to show that $H_p(BD;H_q(BU;\mathcal{L}))=0$ for $p\geq 0$ and $q\geq 1$.  But now by hypothesis 1 the local system is constant so it suffices to show $H_p(BD;H_q(BU;k))=0$ for $p\geq 0$ and $q\geq 1$.  But using the splitting $s$, the action of $D$ on $H_q(BU;k)$ is induced by the conjugation action of $s(D)$ on $U$, so this is handled by hypothesis 2.
\end{proof}

\begin{lemma}\label{diagonal}
Let $A$ be an associative ring, let $\lambda \in Z(A)^\times$ be a central unit, and let $N$ and $N'$ be finitely generated projective $A$-modules.  Fix $p,q\in \mathbb{N}$ and let $D_\lambda$ denote the element of $GL(N)\times GL(N')\subset GL(N\oplus N')$ given by multiplication by $\lambda^p$ in the first factor and multiplication by $\lambda^{-q}$ in the second factor.  Then:
\begin{enumerate}
\item $D_\lambda$ lies in the centre of $GL(N)\times GL(N')$.
\item For a homomorphism $f\colon N'\rightarrow N$, let $U_f\in GL(N\oplus N')$ denote the map which fixes $N$ and sends $N'\rightarrow N\oplus N'$ via $(f,id)$.  Then
$$D_\lambda\cdot U_f \cdot (D_\lambda)^{-1}= U_{\lambda^{p+q}f}.$$
\item If $N\simeq A^n$ and $N'\simeq A^{n'}$ and we choose $p=n'$ and $q=n$, then $D_\lambda$ lies in $E(N\oplus N')$.
\end{enumerate}
\end{lemma}
\begin{proof}
Parts 1 and 2 are simple calculations.  For part 3, note that if we consider $D_\lambda$ as an $(n+n')\times (n+n')$ matrix, then it has determinant $1$.  Thus it suffices to show in general that a $d\times d$ diagonal matrix with entries lying in $Z(A)$ and determinant $1$ necessarily lies in $E_d(A)$.  We can clearly assume $A=R$ commutative.  In the world of $2\times 2$ matrices, a standard calculation shows that $(\lambda, 0;0,\lambda^{-1})$ lies in $E_2(A)$.  Hence a diagonal matrix with determinant one and only two nontrivial adjacent entries lies in $E_d(A)$.  But by multiplying by such matrices we can inductively arrange to make a matrix in $E_d(A)$ which agrees with our given diagonal matrix in its first $d-1$ diagonal entries.  Then the last diagonal entries have to also be the same because of the determinant condition.\end{proof}

Let us adopt the following notation.  If $V$ is an $Z(A)$-module and $n\in\mathbb{Z}$, write $V(n)$ for $V$ considered as an additive abelian group, equipped with $Z(A)^\times$-action described by
$$\lambda \cdot m := \lambda^n m.$$

\begin{theorem}\label{NSvanishingthm}
Let $k$ be a prime field, and let $A$ be an associative ring with centre $R=Z(A)$.  Suppose that for all $R$-modules $V$ isomorphic to $A^d$ for some $d\geq 0$ we have
$$H_p(BR^\times;H_q(BV(n);k))=0$$
for all $p\geq 0$, all $q\geq 1$, and all $n\geq 1$.  Here the $V(n)$-action on $k$ is trivial and the $R^\times$-action on $H_q(BV(n);k)$ comes by functoriality from its action on $V(n)$.

\medskip

Then for all split noetherian finitely generated projective $A$-modules $M$, the natural map
$$BGL(M)\rightarrow \vert \operatorname{RBS}(M)\vert$$
is an isomorphism on homology with $k$-coefficients.

\medskip

If furthermore we assume that either:
\begin{enumerate}
\item every split submodule of $M$ is free, or
\item $H_q(BV;k)=0$ for all $q\geq 1$ and all $R$-modules $V$ isomorphic to $A^d$ for some $d$,
\end{enumerate}
then it is an isomorphism on homology with all $k$-module local coefficients.
\end{theorem}
\begin{proof}
By \Cref{fromflagstoRBS(M)}, it suffices to show that for all splittable flags $\mathcal{F}\leq \mathcal{G} $ on $M$, the map
$$BP_\mathcal{F} \rightarrow B(P_\mathcal{F}/U_\mathcal{G})$$
induces an isomorphism on homology with the correct coefficients.  Let's prove this by induction on $d$, the length of $\mathcal{G}=(M_1\subsetneq \ldots\subsetneq M_{d-1})$.  Write $\mathcal{G}_1$ for the flag of length $2$ given by just $M_1$.  Then we can factor the map in two steps:
$$BP_\mathcal{F}\rightarrow B(P_{\mathcal{F}}/U_{\mathcal{G}_1}) \rightarrow B(P_\mathcal{F}/U_\mathcal{G}).$$
The second map is the product with $BGL(M_1)$ of an instance of our comparison map with length $d-1$, thus it gives an isomorphism.  For the first map, set $N=M_1$ and let $N'$ be the image of a splitting of $M\rightarrow M/M_1$.  By \Cref{vanishingtrick}, it suffices to find a central subgroup
\begin{align*}
D\subset P_{\mathcal{F}}/U_{\mathcal{G}_1}= GL(M_1)\times GL(M/M_1) = GL(N)\times GL(N')
\end{align*}
such that:
\begin{enumerate}
\item the local system is constant on $BD$; and
\item $H_*(BD;H_q(BU_{\mathcal{G}_1};k))=0$ for all $q\geq 1$.
\end{enumerate}
For the constant local system, claim 1 is trivial, and to arrange claim 2, we can fix $p,q\geq 1$ arbitrarily and let $D$ be the subgroup formed by all $D_\lambda$ from \Cref{diagonal}.  Note that since $N,N'$ are finitely generated projective, the $R$-module $U_{\mathcal{G}_1}=\operatorname{Hom}_A(N,N')$ is isomorphic to a retract of $A^d$ for some $d$, so the assumption on homology vanishing for $A^d(p+q)$ implies it for $U_{\mathcal{G}_1}$.
Next, for non-constant local systems, if every split submodule of $M$ is free, then we can make the specific choice of $p,q$ as in (3) of \Cref{diagonal}, and then claim (1) holds by (3) of \Cref{diagonal} and claim (2) follows as before.  Finally, again for non-constant local systems, if we have the strong vanishing $H_q(BV;k)=0$ for all $q\geq 1$ as given in assumption (2), then we get the same vanishing $H_q(BU_{\mathcal{G}_1};k)=0$ for $q\geq 1$, and we can take $D=\{1\}$ to conclude.
\end{proof}

The following is Nesterenko--Suslin's key observation, see \cite{NS}.

\begin{lemma}\label{NSlemma}
Let $R$ be a commutative ring \emph{with many units}: for any $n\geq 1$ there exist $r_1,\ldots, r_n \in R^\times$ such that for all nonempty $I\subset \{1,\ldots, n\}$, the sum $\sum_{i\in I} r_i$ is also a unit.

\medskip

Then for $n> 0$, all prime fields $k$, and all $R$-modules $M$, we have
$$H_p(BR^\times;H_q(BM(n);k))=0$$
for all $p\geq 0$ and $q\geq 1$.
\end{lemma}
\begin{proof}
This is explained in \cite{NS} when $n=1$, and the same argument works in general.  Recall Nesterenko--Suslin's result: for all $n\geq 1$, if we let $S_n(R)$ denote the ring $(R^{\otimes n})^{S_n}$ and take the diagonal embedding $R^\times\rightarrow S_n(R)^\times$, then for every $S_n(R)$-module $N$ we have $H_p(BR^\times;N)=0$ for all $p\geq 0$.  Here all tensor products are over $\mathbb{Z}$. Now, if $k=\mathbb{Q}$ we have
$$H_q(BM(n);\mathbb{Q})=\Lambda^q M_{\mathbb{Q}},$$
with $R^\times$-action given by
$$\lambda \cdot (m_1\wedge \ldots \wedge m_q) = (\lambda^nm_1)\wedge \ldots \wedge (\lambda^n m_q).$$
We can put an $S_{nq}(R)$-module structure on $\Lambda^qM$ by viewing $M$ as an $R^{\otimes n}$-module via restriction along the multiplication map $R^{\otimes n}\rightarrow R$, then using this to view $\otimes^q M$ as an $(R^{\otimes n})^{\otimes q}$-module, hence as an $S_{nq}(R)$-module by restriction, and passing to the quotient $\Lambda^q M$.  The correct $R^\times$-action is recovered and so Nesterenko--Suslin's result proves the desired vanishing in this case.

\medskip

If $k=\mathbb{F}_p$, then we have to use the fact that there is functorial filtration on $H_q(M;\mathbb{F}_p)$ with associated graded pieces given by $\Lambda^{q-2j} (M/pM) \otimes \Gamma_j(M[p])$, and then we can similarly argue for vanishing  of $H_\ast(R^\times;-)$ on these pieces, hence on $H_q(M;\mathbb{F}_p)$, by equipping $\Lambda^{q-2j} (M/pM) \otimes \Gamma_j(M[p])$ with appropriate $S_{n\cdot(q-j)}$-action.
\end{proof}

\begin{theorem}
Let $A$ be an associative ring with many units, meaning that its centre $R=Z(A)$ has many units in the sense described in the above lemma, and let $M$ be a split noetherian finitely generated projective $A$-module.  Then
$$BGL(M)\rightarrow\vert \operatorname{RBS}(M)\vert$$
is an isomorphism on $\mathbb{Z}$-homology, and if every split submodule of $M$ is free then it is also an isomorphism on homology with all local coefficient systems, and hence $E(M)\subset GL(M)$ is perfect and for the associated plus construction we have
$$BGL(M)^+ \overset{\sim}{\rightarrow}\vert\operatorname{RBS}(M)\vert.$$
\end{theorem}

\begin{proof}
We already saw that the map
$$c\colon BGL(M)\rightarrow \vert \operatorname{RBS}(M)\vert$$
is an isomorphism on $\pi_0$, and on $\pi_1$ it identifies $\pi_1\vert \operatorname{RBS}(M)\vert$ with $GL(M)/E(M)$.  Thus it suffices to show the homology isomorphism statements.  But these follow from \Cref{NSlemma}  and \Cref{NSvanishingthm}.\end{proof}

\subsection{The case of finite fields}

The simplest example of a ring not having many units is a finite field $\mathbb{F}_q$, and here we will see that for a finite dimensional $\mathbb{F}_q$-vector space the anima $\vert \operatorname{RBS}(V)\vert$ is ``better'' than the plus construction in the sense that it can be computed and identified with the absolute most naive unstable analogue of the K-theory $K(\mathbb{F}_q)$ as computed by Quillen.

\medskip

First we note that with $\mathbb{Q}$-coefficients or $\mathbb{F}_\ell$ coefficients for $\ell\neq p$, there is no difference.  So all the interest lies in $\mathbb{F}_p$-coefficients.

\begin{lemma}\label{primetop}
Let $k$ be a prime field and $A$ be an associative ring with $A\otimes_\mathbb{Z}k=0$.  Then for any split noetherian finitely generated projective $A$-module $M$ the map
$$BGL(M)\rightarrow \vert \operatorname{RBS}(M)\vert$$
is an isomorphism on homology with local coefficients a $k$-module.
\end{lemma}
\begin{proof}
By \Cref{NSvanishingthm}, it suffices to show that $H_p(B A^r;k)=0$ for all $r\geq 0$, and $p\geq 1$; actually we can take $r=1$ by the K\"unneth theorem.  But the description of homology of abelian groups with $k$-coefficients recalled in the proof of \Cref{NSlemma} shows that each $H_p(BA;k)$ admits an $A\otimes_{\mathbb{Z}}k$-module structure, hence vanishes.
\end{proof}

Also, the $\pi_1$ is easy to identify.

\begin{lemma}\label{SLn}
Let $A$ be a local commutative ring and $M$ a finitely generated projective $A$-module. Then
$$\pi_1\vert \operatorname{RBS}(M)\vert=GL(M)/SL(M)=A^\times.$$
\end{lemma}
\begin{proof}
By Nakayama's lemma it follows that $M\simeq A^n$ is free.  Recall that $M$ is automatically split noetherian as $\operatorname{Spec}(A)$ is connected, \Cref{whennoetherian}.  Then by \Cref{pi1} it suffices to see that $E(A^n)=SL_n(A)$.  But this follows by induction because $GL_n(A)=E_n(A)\cdot GL_{n-1}(A)$, see \cite{W} III.1.4.
\end{proof}

Now we handle $\mathbb{F}_p$-coefficients.

\begin{theorem}\label{finite field Fp coeff}
Let $k$ be a finite field with $q=p^r$ elements, $p$ prime, and let $V$ be a finite dimensional $k$-vector space.  Then the Postnikov truncation map $\vert \operatorname{RBS}(V)\vert\rightarrow Bk^\times$ induces an isomorphism on homology with coefficients in any local system $\mathcal{L}$ of $\mathbb{F}_p$-modules.
\end{theorem}
\begin{proof}
By extending coefficients to an algebraic closure $\overline{\mathbb{F}_p}$ of $\mathbb{F}_p$, it suffices to prove the claim for local systems of $\overline{\mathbb{F}_p}$-modules instead.  Then, since $k^\times$ is abelian and of order prime to $p$, we see that every such local system is a direct sum of one-dimensional local systems.  In total, we can therefore reduce to the case where $\mathcal{L}$ is a one-dimensional local system over $\overline{\mathbb{F}_p}$ (corresponding to a character $k^\times\rightarrow\overline{\mathbb{F}_p}^\times$).

\medskip

Let us prove the homology isomorphism statement by induction on the dimension $n$ of $V$.   For $n=1$ we have that $\operatorname{RBS}(V)=Bk^\times$ and the claim is tautological.  Now, since $k^\times$ has order prime to $p$ its $\mathcal{L}$-homology vanishes in positive degrees, and its degree zero part is $\mathcal{L}_{k^\times}$, so we have to show the same for the $\mathcal{L}$-homology of $\vert \operatorname{RBS}(V)\vert$.  We can assume $n>1$ and use induction.

\medskip

Note that the inductive hypothesis and the decomposition in part 1 of \Cref{inductive structure} imply that for any nonempty flag $\mathcal{F}$ in $V$ the pullback of $\mathcal{L}$ to $\operatorname{RBS}(V)_{\leq [\mathcal{F}]}$ has vanishing homology in positive degrees and is $\mathcal{L}_{k^\times}$ in degree $0$.  By part 2 of \Cref{inductiveRBScat} it follows that the $\mathcal{L}$-homology of the boundary $\vert \operatorname{RBS}(V)\smallsetminus BGL(V)\vert$ identifies with the homology of the constant local system with value $\mathcal{L}_{k^\times}$ on the poset $(GL(V)\backslash \mathcal{P})\smallsetminus [\emptyset]$; but this poset has a minimal element given by the full flags and hence is contractible, so this homology is just $\mathcal{L}_{k^\times}$ in degree $0$.  Thus it suffices to show that the $\mathcal{L}$-homology of $\operatorname{RBS}(V)$ relative to the boundary vanishes.

\medskip

However, by \Cref{inductiveRBScat} part 1 we can replace $\operatorname{RBS}(V)$ by $GL(V)\backslash\backslash\mathcal{P}$ for this question.  Now, $\vert GL(V)\backslash\backslash \mathcal{P}\vert $ is the (homotopy) quotient of $\vert \mathcal{P}\vert$ by the $GL(V)$-action, \Cref{quotientcat}, and similarly $\vert \partial\, GL(V)\backslash\backslash\mathcal{P}\vert$ is the analogous quotient of $\vert \mathcal{P}\smallsetminus \{\emptyset\}\vert$ of nonempty flags, also known as the Tits building.  Thus $\vert\operatorname{RBS}(V)\vert/\vert \partial\operatorname{RBS}(V)\vert$ identifies with the homotopy quotient
$$(\vert\mathcal{P}\vert /\vert\mathcal{P}\smallsetminus \{\emptyset\}\vert)_{hGL(V)}.$$
Now, recall the Solomon--Tits theorem: $\vert\mathcal{P}\vert/\vert\mathcal{P}\smallsetminus \{\emptyset\}\vert\simeq \Sigma \vert\mathcal{P}\smallsetminus\{\emptyset\}\vert$ has the pointed homotopy type of a wedge of $q^{\frac{n^2-n}{2}}$ many $(n-1)$-spheres.  Thus what need to show is that for the \emph{Steinberg representation} of $GL(V)$ with $\overline{\mathbb{F}_p}$-coefficients, defined as
$$\operatorname{St}_{\mathbb{F}_p}:= \widetilde{H}_{n-1}(\vert\mathcal{P}\vert/\vert\mathcal{P}\smallsetminus \{\emptyset\}\vert;\overline{\mathbb{F}_p}),$$
we have that $\operatorname{St}_{\mathbb{F}_p}\otimes_{\overline{\mathbb{F}_p}} \mathcal{L}$ has vanishing $GL(V)$-homology.  We refer to the survey \cite{HS} for the Steinberg representation.  Character computations and the Brauer--Nesbitt theorem show it is projective and irreducible, hence so is its tensor with the character $\mathcal{L}$.  The desired homology vanishing in positive degrees follows directly from projectivity, and as for degree $0$, it suffices to see that $\operatorname{St}_{\mathbb{F}_p}\otimes_{\overline{\mathbb{F}_p}} \mathcal{L}$ has no nonzero trivial quotients.  Being projective and irreducible, this is equivalent to saying it's not the trivial one-dimensional character.  But its dimension $q^{\frac{n^2-n}{2}}>1$, whence the conclusion.
\end{proof}

It follows from these lemmas that each homotopy group of $\vert \operatorname{RBS}(V)\vert$ is prime-to-$p$-torsion, and is all accounted for by the homology of $GL(V)$ with mod $\ell$-coefficients, $\ell\neq p$, which Quillen has computed, \cite{Qfinite}.  In fact we can be more precise:

\begin{corollary}
Let $k$ be a finite field with $q=p^r$ elements and $V$ a $k$-vector space of dimension $n\in\mathbb{N}$.  Then
$$\vert \operatorname{RBS}(V)\vert \simeq ((B\vert U(n)\vert)')^{\psi^q \sim \operatorname{id}},$$
the homotopy fixed space for the unstable $q$-Adams operation on the prime-to-$p$ completion of the delooping of the group anima underlying the compact Lie group $U(n)$.\footnote{Thus, $\psi^q$ restricted to $B\vert U(1)^n\vert'$ is induced by the endomorphism $x\mapsto x^q$ of $U(1)$.  See \cite{JMO} for unstable Adams operations.}
\end{corollary}
\begin{proof}
The Friedlander--Quillen argument, see \cite{FQ} Theorem 12.2, uses étale homotopy theory and the Lang isogeny to produce a map
$$BGL(V)\rightarrow ((B\vert U(n)\vert)') ^{\psi^q \sim \operatorname{id}},$$
which is an equivalence on homology with $\mathbb{F}_\ell$-coefficients for $\ell\neq p$, hence identifies the target as the prime-to-$p$ completion of the source.  Since \Cref{primetop} implies that the map $BGL(V)\rightarrow\vert \operatorname{RBS}(V)\vert$ is a $\mathbb{Z}[1/p]$-equivalence, we deduce a comparison map
$$\vert \operatorname{RBS}(V)\vert\rightarrow  ((B\vert U(n)\vert)') ^{\psi^q \sim \operatorname{id}}$$
which is an isomorphism with $\mathbb{Z}[1/p]$-coefficients, and on $\pi_0$ and $\pi_1$ as we see from \Cref{SLn}.  But the homotopy groups of the target are prime-to-$p$, so the map from the target to its $\tau_{\leq 1}$-Postnikov truncation induces an isomorphism on homology with coefficients an $\mathbb{F}_p$-module.  The same was checked for the left hand side in the previous theorem, so we conclude our comparison map is an isomorphism both on $\tau_{\leq 1}$ and on homology with all local coefficient systems, hence is an isomorphism of anima.
\end{proof}

\section{Monoidal categories and actions: a lemma}\label{monoidal categories and actions}

In the next section, \Cref{comparison with algebraic K-theory}, we relate the categories $\operatorname{RBS}(M)$ for finitely generated projective modules $M$ over an associative ring $A$ to the algebraic K-theory space $K(A)$. The present section is in some sense just a long and technical lemma that we will need to make the final comparison. To motivate the work to be done, we will describe in broad terms what we do in the following section and in slightly more detail what will happen in this section.

\medskip

For an associative ring $A$ and the exact category $\mathcal{P}(A)$ of finitely generated projective $A$-modules, we consider a monoidal category $M_{\mathcal{P}(A)}$ whose objects are finite ordered lists $(M_1,\ldots,M_d)$ of objects in $\mathcal{P}(A)$ and a morphism $(M_1,\ldots, M_d)\rightarrow (N_1,\ldots,N_e)$ is the data of a flag on each $N_i$ together with an isomorphism of the total associated graded of this list of flags with the $M_i$, in order. In particular, such a morphism can only exist if $e\leq d$. The relationship between $M_{\mathcal{P}(A)}$ and the categories $\operatorname{RBS}(M)$ should be thought of as an analogue of the relationship between the (symmetric) monoidal category $i\mathcal{P}(A)$ and the $BGL(M)$: if $\mathcal{M}$ is a set of representatives of isomorphism classes of finitely generated projective $A$-modules, then, in the same way that $i\mathcal{P}(A)\simeq \coprod_{M\in \mathcal{M}}BGL(M)$ with monoidal product given by direct sum, we have $M_{\mathcal{P}(A)}\simeq \coprod_{M\in \mathcal{M}} \operatorname{RBS}(M)$ with monoidal product, not symmetric, induced by concatenation.

\medskip

In \Cref{comparison with algebraic K-theory}, we do all this in much greater generality, defining a monoidal category $M_{\mathscr{E}}$ for any exact category $\mathscr{E}$ (and in fact, a little more generally than that), and the main theorem of that section is the following (see \Cref{the K-theory space}).

\begin{theorem*}
For any exact category $\mathscr{E}$, the geometric realisation of Quillen's Q-construction $Q(\mathscr{E})$ is homotopy equivalent to the classifying space $B|M_{\mathscr{E}}|$ of $|M_{\mathscr{E}}|$. In particular,
$$\Omega B|M_{\mathscr{E}}|\simeq K(\mathscr{E}).$$
\end{theorem*}

\begin{remark}
The arguments in this section and the next are of a very different type from the arguments in the previous sections.  Whereas up to now we've relied on soft $\infty$-categorical techniques, here we use direct and explicit simplicial manipulations.  Our basic objects are not $\infty$-categories, but rather $1$-categories, $2$-categories, and double categories and their associated multi-simplicial objects; see \Cref{appendix nerves and geometric realisations} for the relevant definitions.  Thus we adopt a more classical language and notation.  In particular, when we use the notation $|\cdot |$ here, we mean the classical geometric realisation, a topological space.  So the $|M_{\mathscr{E}}|$ above is a topological monoid, though the reader may take it to mean the underlying $E_1$-anima instead if desired.  Indeed, we had previously used the $|\cdot |$ notation for the anima associated to an $\infty$-category as discussed in \Cref{describeinvert}; the consistency between the two choices in notation results from \Cref{propertiesinvert} part 6 together with the Quillen equivalence between simplicial sets and topological spaces.
\exend
\end{remark}

To make the comparison in the above theorem we introduce an intermediary Q-construction, which we on the one hand can compare with Quillen's Q-construction and on the other hand can compare with the classifying space of our monoidal category. This latter comparison is what we prepare for in this section.

\medskip

Recall Segal's classical result in which he uses edgewise subdivision to prove that the classifying space of a monoid $M$ is homeomorphic to the geometric realisation of the category $\mathscr{C}(M)$ with objects the elements of $M$ and morphisms $(a,b)\colon m\rightarrow m'$ where $a,b\in M$ such that $amb=m'$ (\cite[Proposition 2.5]{Segal73}). The intermediary Q-construction that we introduce is a $2$-categorical version of Segal's $\mathscr{C}(M)$ associated to a monoidal category instead of a monoid, and the comparison of the classifying space with the geometric realisation also goes through edgewise subdivision. The extra categorical level means, however, that we have to go through a wealth of simplicial manipulations to make the comparison.

\medskip

To ease notation, we work in greater generality and introduce a 2-categorical Q-construction $Q(M,X)$ encoding the action of a monoidal category $M$ on a category $X$, and we then compare the geometric realisation of this $2$-category with the total realisation of the simplicial category whose category of $n$-simplices is $M^n\times X$ and whose structure maps are given by the action, multiplication and projection maps (let us stress here once and for all that by a simplicial category, we mean a simplicial object in categories and \textit{not} a category enriched in simplicial sets). The case that we will then ultimately be interested in is the following: for a monoidal category $M$, the product $M\times M^{\otimes^{\operatorname{op}}}$ acts on $M$ by left and right multiplication, and in this case our $2$-categorical $Q$-construction generalises Segal's category, and the simplicial category $(M\times M^{\otimes^{\operatorname{op}}})^\bullet\times M$ identifies with the edgewise subdivision of the usual bar construction on $M$.

\medskip

\textbf{Outline and proof strategy.} 
Let $M$ be a strict monoidal category acting strictly on a category $X$ via a functor $M\times X\rightarrow X$ satisfying the necessary coherency axioms. We will compare various constructions which encode this action in different ways, but before we begin we sketch the outline of the section so as to make the ideas easier to follow. The proofs given here are straightforward but the technicalities build up as we enlarge different structures by incorporating ``redundant'' data in order to compare them, so it is easy to loose sight of the bigger picture.

\medskip

We consider the double category $\mathscr{M}\times X=[M\times X\rightrightarrows X]$ which encodes the action of $M$ on $X$ and whose geometric realisation is a model for the homotopy quotient $|X|_{|M|}$ of the topological monoid $|M|$ acting on the geometric realisation $|X|$. We define another double category $\mathscr{M}\ltimes X$, which is in some sense a lax version of $\mathscr{M}\times X$, and a double functor $\mathscr{M}\times X\rightarrow \mathscr{M}\ltimes X$ and we show that this induces a homotopy equivalence of geometric realisations.

\medskip

We then define $Q(M,X)$ to be the vertical $2$-category of $\mathscr{M}\ltimes X$; that is, the sub-double category whose only horizontal morphisms are the identities. More precisely, the objects are those of $\mathscr{M}\ltimes X$, the morphisms are the vertical morphisms and the $2$-cells are the $2$-cells whose source and target horizontal morphisms are identities. It is the proof of the following theorem which will take up the most of this section.

\begin{xreftheorem}{classifying spaces of lax action double cat and action 2-cat are htpy eq}
The inclusion $Q(M,X) \rightarrow \mathscr{M}\ltimes X$ induces a homotopy equivalence of geometric realisations.
\end{xreftheorem}

The horizontal morphisms of $\mathscr{M}\ltimes X$ are special cases of the vertical morphisms, so this result sounds like a double categorical version of Waldhausen's swallowing lemma (\cite[Lemma 1.6.5]{Waldhausen}). The proof that we present is, although much more involved, inspired by Waldhausen's proof.

\medskip

The proof strategy is as follows. We define a bisimplicial category $\mathscr{B}_{\bullet\bullet}$ which horizontally collapses to a simplicial double category $\mathscr{B}_\bullet$ and vertically collapses to a simplicial $2$-category $\mathscr{A}_\bullet$. For every $n\geq 0$, we define a double functor $\mathscr{M}\ltimes X\rightarrow \mathscr{B}_n$ and a pseudofunctor $Q(M,X)\rightarrow \mathscr{A}_n$, and we show that these induce homotopy equivalences of geometric realisations. In both cases, the proofs are analogous to that of Waldhausen's swallowing lemma, in that we have adjunctions given by inclusion at zero and retraction to zero. It follows that we have a zig-zag of homotopy equivalences
\begin{align*}
|\mathscr{M}\ltimes X|\xrightarrow{\ \simeq\ }|\mathscr{B}_{\bullet\bullet}|\xleftarrow{\ \simeq \ }|Q(M,X)|,
\end{align*}
and one can then verify on the diagonals that this is given by the inclusion $Q(M,X)\rightarrow \mathscr{M}\ltimes X$.

\medskip

The idea behind the construction of $\mathscr{B}_{\bullet\bullet}$ is to incorporate the $2$-cells of $\mathscr{M}\ltimes X$ both horizontally and vertically. When we collapse $\mathscr{B}_{\bullet\bullet}$ vertically, we ``swallow'' the vertical $2$-cell structure into the horizontal $2$-cell structure, and vice versa in the other direction.

\medskip

Note that the proof as it is written up runs through this procedure backwards. We define double categories $\mathscr{B}_n$ and $2$-categories $\mathscr{A}_n$ for all $n\geq 0$ and establish the desired homotopy equivalences, and we then incorporate the $\mathscr{A}_n$'s and $\mathscr{B}_n$'s into a bisimplicial category $\mathscr{B}_{\bullet\bullet}$ at the end.

\subsection{Lax action double category}

Let $M$ be a strict monoidal category acting strictly on a category $X$, i.e.~via a functor $M\times X\rightarrow X$ satisfying the necessary coherency axioms. To ease notation, we denote the monoidal product and the action by juxtaposition, and we also write $m\phi$ in place of $\operatorname{id}_m\phi$ for an object $m$ in $M$ and a morphism $\phi$ in $X$.

\medskip

The action of $M$ on $X$ gives rise to a double category
\begin{align*}
\mathscr{M}\times X=[M\times X\rightrightarrows X]
\end{align*}
where the source map is projection to $X$, the target map is given by the action map, the identity section is the section at the identity element $e$ of $M$, and vertical composition is given by the product in $M$. In other words, the objects and horizontal morphisms are those of $X$, the vertical morphisms are of the form $m\colon x\rightarrow mx$ with $m$ in $M$ and $x$ in $X$ and with composition given by the monoidal product, and finally the 2-cells are morphisms $(\alpha,f)\colon (m,x)\rightarrow (m',x')$ in $M\times X$, which can be interpreted as commutative diagram
\begin{center}
\begin{tikzpicture}
\matrix (m) [matrix of math nodes,row sep=1em,column sep=1em, text height=1.5ex, text depth=0.25ex]
  {
	x & x' \\
	mx & m'x' \\
  };
  \path[-stealth] 
	(m-1-1) edge node[above]{$f$} (m-1-2) edge node[left]{$m$} (m-2-1)
	(m-2-1) edge node[below]{$\alpha f$} (m-2-2)
	(m-1-2) edge node[right]{$m'$}(m-2-2)
  ;
\end{tikzpicture}
\end{center}

We now define another double category, which can be interpreted as a ``lax'' version of $\mathscr{M}\times X$.

\begin{construction}\label{lax action double category}
Let $M.X$ be the following category. The objects are of the form
\begin{align*}
(x,y,m,\phi\colon mx\rightarrow y),
\end{align*}
where $x,y$ are objects of $X$, $m$ is an object of $M$ and $\phi$ is a morphism in $X$. The morphisms are tuples
\begin{align*}
(x,y,m,\phi)\xrightarrow{(f,g,\alpha)} (x',y',m',\phi')
\end{align*}
where $f\colon x\rightarrow x'$, $g\colon y\rightarrow y'$ in $X$ and $\alpha\colon m\rightarrow m'$ in $M$ such that the following diagram commutes
\begin{center}
\begin{tikzpicture}
\matrix (m) [matrix of math nodes,row sep=1em,column sep=1em, text height=1.5ex, text depth=0.25ex]
  {
	mx & m'x' \\
	y & y' \\
  };
  \path[-stealth] 
	(m-1-1) edge node[above]{$\alpha f$} (m-1-2) edge node[left]{$\phi$} (m-2-1)
	(m-2-1) edge node[below]{$g$} (m-2-2)
	(m-1-2) edge node[right]{$\phi'$}(m-2-2)
  ;
\end{tikzpicture}
\end{center}
Composition is given by coordinatewise composition.

\medskip

Define a double category
\begin{align*}
\mathscr{M}\ltimes X=[M.X \rightrightarrows X]
\end{align*}
whose source and target maps are the projections
\begin{align*}
s\colon (x,y,m,\phi)\mapsto x,\quad(f,g,\alpha)\mapsto f,\quad \text{and}\quad t\colon (x,y,m,\phi)\mapsto y,\quad(f,g,\alpha)\mapsto g.
\end{align*}
Vertical composition is given by the monoidal product: for morphisms
\begin{align*}
(y,z,n,\psi)\circ_v (x,y,m,\phi) = (x,z,nm,\psi\circ n\phi),
\end{align*}
and for $2$-cells
\begin{align*}
(g,h,\beta)\circ_v (f,g,\alpha)= (f,h,\beta\alpha).
\end{align*}
The identity section is the functor $X\rightarrow M.X$, $x\mapsto (x,x,e,\operatorname{id}_x)$.
\exend
\end{construction}

\begin{definition}\label{action Q-construction}
The \textit{action Q-construction} of the action of $M$ on $X$ is the vertical $2$-category $Q(M,X)$ of $\mathscr{M}\ltimes X$. More precisely, the objects are those of $\mathscr{M}\ltimes X$, i.e. the objects of $X$, the morphisms are the vertical morphisms and the $2$-cells are the $2$-cells whose source and target morphisms are identities. We denote the hom-categories of $Q(M,X)$ by $M(x,y)$.
\exend
\end{definition}

\begin{remark}
Recall that a strict monoidal category $M$ can be viewed as a $2$-category $\mathscr{M}$ with one object: the morphisms are the objects of $M$ with composition given by the monoidal product, and the $2$-cells are the morphisms of $M$ with the usual composition. The double category $\mathscr{M}\ltimes X$ can be viewed as a ``double categorical Grothendieck construction'' for the functor $\mathscr{M}\rightarrow \operatorname{Cat}$ from the 2-category associated to $M$ into the $2$-category of small categories which sends the unique object to the category $X$, a morphism $m$ to the map $m\colon X\rightarrow X$ given by the action of $M$, and a $2$-cell $m\rightarrow m'$ to the corresponding natural transformation.
\exend
\end{remark}

\begin{remark}\label{Q(M,X) vs Quillen's SS}
The proof of the ``$Q=+$'' Theorem uses an intermediary $S^{-1}S$-construction (see \cite{Grayson}) which is a category $\langle S\times S, S\rangle$ defined for a monoidal category $S$ and an action of $S\times S$ acting on $S$. The construction $\langle M,X\rangle$ is defined more generally in \cite{Grayson} for an action of a monoidal $M$ on a category $X$. The construction $Q(M,X)$ above is related to the category $\langle M,X \rangle$ in the following way: $\langle M,X\rangle$ is the $1$-category obtained by taking isomorphism classes of objects in the hom-categories of $Q(M,X)$.
\exend
\end{remark}

Consider the functor $\Phi_1\colon M\times X\rightarrow M.X$ given by 
\begin{align*}
\Phi_1(m,x)=(x,mx,m,\operatorname{id}_{mx})\quad \text{and}\quad \Phi_1(\alpha,f)=(f,\alpha f ,\alpha)
\end{align*}
and consider the double functor
\begin{align*}
\Phi\colon \mathscr{M}\times X\rightarrow \mathscr{M}\ltimes X
\end{align*}
restricting to the identity on the object category $X$ and given by $\Phi_1$ on morphism categories.

\begin{lemma}\label{Phi induces homotopy equivalence}
The double functor $\Phi\colon \mathscr{M}\times X\rightarrow \mathscr{M}\ltimes X$ induces a homotopy equivalence of geometric realisations.
\end{lemma}
\begin{proof}
First of all, $\Phi$ induces a morphism of the vertical nerves
\begin{align*}
\Phi_\bullet \colon N_\bullet^v(\mathscr{M}\times X)\rightarrow N_\bullet^v(\mathscr{M}\ltimes X).
\end{align*}
Recall that the vertical nerve of a double category $\mathscr{C}=[C_1\rightrightarrows C_0]$ is a simplicial category $N_\bullet^v(\mathscr{C})$, where $N_n^v(\mathscr{C})$ has as objects sequences of vertical morphisms
\begin{align*}
c_0\xrightarrow{\phi_1} c_1 \xrightarrow{\phi_2} \cdots \xrightarrow{\phi_n} c_n
\end{align*}
and a morphism from $c_0\xrightarrow{\phi_1} \cdots \xrightarrow{\phi_n} c_n$ to $d_0\xrightarrow{\psi_1} \cdots \xrightarrow{\psi_n} d_n$ is a collection of $2$-cells
\begin{align*}
\alpha_i\colon \phi_i\Rightarrow \psi_i
\end{align*}
satisfying $t(\alpha_i)=s(\alpha_{i+1})$ for all $i$ (see also \Cref{appendix nerves and geometric realisations}).

\medskip

We show that for all $n$, the functor $\Phi_n$ admits a right adjoint, which proves the claim in view of the realisation lemma. Consider the functor $\Psi_n\colon N_n^v(\mathscr{M}\ltimes X)\rightarrow N_n^v(\mathscr{M}\times X)$ which sends an object
\begin{align*}
x_0\xrightarrow{(x_0,x_1,m_1,\phi_1)} x_1 \xrightarrow{(x_1,x_2,m_2,\phi_2)} \cdots \xrightarrow{(x_{n-1},x_n,m_n,\phi_n)} x_n
\end{align*}
to
\begin{align*}
x_0\xrightarrow{m_1} m_1x_0 \xrightarrow{m_2} m_2m_1x_0\xrightarrow{m_3} \cdots \xrightarrow{m_n} m_n\cdots m_1x_0
\end{align*}
and a morphism
\begin{align*}
(f_{i-1},f_i,\alpha_i)\colon (x_{i-1},x_i,m_i,\phi_i)\Rightarrow (x'_{i-1},x'_i,m'_i,\phi'_i),\quad i=1,\ldots,n,
\end{align*}
to
\begin{align*}
(\alpha_i,\alpha_{i-1}\cdots \alpha_1f_0)\colon (m_i,m_{i-1}\cdots m_1 x_0)\Rightarrow (m_i',m_{i-1}'\cdots m_1' x_0'),\quad i=1,\ldots,n.
\end{align*}

We claim that $\Psi_n$ is right adjoint to $\Phi_n$. Indeed, the unit transformation is the identity and for the counit transformation $\Phi_n\circ \Psi_n\Rightarrow \operatorname{id}$, we take the morphisms
\begin{align*}
(*,\phi_i\circ m_i(*))\colon (m_{i-1}\cdots m_1x_0, m_i\cdots m_1x_0,m_i,\operatorname{id})\Rightarrow (x_{i-1},x_i,m_i,\phi_i)
\end{align*}
where $*=\phi_{i-1}\circ m_{i-1}(\phi_{i-2}\circ m_{i-2}(\cdots m_3(\phi_2\circ m_2\phi_1))$.
\end{proof} 

\subsection{Enlarging the lax action double category}

Fix $n\geq 0$. We define a double category $\mathscr{B}_n$ and a double functor $\mathscr{M}\ltimes X\rightarrow \mathscr{B}_n$ which induces a homotopy equivalence of geometric realisations. There will be quite a bit of redundant data in our notation, but we keep it in order to make the final comparison clearer.

\begin{construction}
We define a double category $\mathscr{B}_n=[\operatorname{mor}\mathscr{B}_n\rightrightarrows X]$ whose morphism category $\operatorname{mor} \mathscr{B}_n$ is given as follows. The object set is
\begin{align*}
\coprod_{x,y\in \operatorname{ob} X} N_n M(x,y),
\end{align*}
where $M(x,y)$ is the hom-category in $Q(M,X)$ (\Cref{action Q-construction}). In other words, an object in $\operatorname{mor} \mathscr{B}_n$ is a sequence 
\begin{align*}
(x,y,m_0,\phi_0)\xrightarrow{(\operatorname{id}_x,\operatorname{id}_y,\beta_1)}(x,y,m_1,\phi_1) \xrightarrow{(\operatorname{id}_x,\operatorname{id}_y,\beta_2)}\cdots \xrightarrow{(\operatorname{id}_x,\operatorname{id}_y,\beta_n)}(x,y,m_n,\phi_n)
\end{align*}
in $M. X$. The morphisms in $\operatorname{mor} \mathscr{B}_n$ are given by commutative diagrams in $M. X$ as pictured below.
\begin{center}
\begin{tikzpicture}
\matrix (m) [matrix of math nodes,row sep=2em,column sep=3em, text height=1.5ex, text depth=0.25ex]
  {
\scriptstyle(x,y,m_0,\phi_0) & \scriptstyle(x,y,m_1,\phi_1) & \ & \ & \scriptstyle(x,y,m_n,\phi_n) \\
\scriptstyle(x',y',m'_0,\phi'_0)& \scriptstyle(x',y',m'_1,\phi'_1)& \ & \ & \scriptstyle(x',y',m'_n,\phi'_n) \\
  };
   \path[-stealth]
   (m-1-1) edge node[above]{$\scriptstyle(\operatorname{id}_x,\operatorname{id}_y,\beta_1)$} (m-1-2) 
   (m-1-2) edge node[above]{$\scriptstyle(\operatorname{id}_x,\operatorname{id}_y,\beta_2)$} (m-1-3)
   (m-1-4) edge node[above]{$\scriptstyle(\operatorname{id}_x,\operatorname{id}_y,\beta_n)$} (m-1-5)
   (m-2-1) edge node[below]{$\scriptstyle(\operatorname{id}_{x'},\operatorname{id}_{y'},\beta'_1)$} (m-2-2) 
   (m-2-2) edge node[below]{$\scriptstyle(\operatorname{id}_{x'},\operatorname{id}_{y'},\beta'_2)$} (m-2-3)
   (m-2-4) edge node[below]{$\scriptstyle(\operatorname{id}_{x'},\operatorname{id}_{y'},\beta'_n)$} (m-2-5)
   (m-1-1) edge node[left]{$\scriptstyle(f,g,\alpha_0)$} (m-2-1)
   (m-1-2) edge node[left]{$\scriptstyle(f,g,\alpha_1)$} (m-2-2)
   (m-1-5) edge node[left]{$\scriptstyle(f,g,\alpha_n)$} (m-2-5)
   ;
	\path[dotted]
		(m-1-3) edge (m-1-4)
		(m-2-3) edge (m-2-4)	
	;
\end{tikzpicture}
\end{center}
Composition is given by composition of the $(f,g,\alpha_i)$ in $M. X$.

\medskip

To ease notation, we denote objects by
\begin{align*}
\bigg[(x,y,m_{i-1},\phi_{i-1})\xrightarrow{\ (\operatorname{id}_x,\operatorname{id}_y,\beta_i)\ } (x,y,m_i, \phi_i)\bigg]_{1\leq i\leq n}
\end{align*}
or simply $(\operatorname{id}_x,\operatorname{id}_y, \beta_i)_{1\leq i\leq n}$ if the $m_i$ and $\phi_i$ are implicit. A morphism will be denoted by
\begin{center}
\[\left[
\begin{tikzpicture}[baseline=(current bounding box.center)]
\matrix (m) [matrix of math nodes,row sep=2em,column sep=4em, text height=1ex, text depth=0.25ex]
  {
(x,y,m_{i-1},\phi_{i-1}) & (x,y,m_i,\phi_i) \\
(x',y',m'_{i-1},\phi'_{i-1})& (x',y',m'_i,\phi'_i)\\
  };
   \path[-stealth]
   (m-1-1) edge node[above]{$\scriptstyle(\operatorname{id}_x,\operatorname{id}_y,\beta_i)$} (m-1-2) 
   (m-2-1) edge node[below]{$\scriptstyle(\operatorname{id}_{x'},\operatorname{id}_{y'},\beta'_i)$} (m-2-2) 
   (m-1-1) edge node[left]{$\scriptstyle(f,g,\alpha_{i-1})$} (m-2-1)
   (m-1-2) edge node[right]{$\scriptstyle(f,g,\alpha_i)$} (m-2-2)
;
\end{tikzpicture}
\right]_{1\leq i\leq n}\]
\end{center}
or simply by $(f,g,\alpha_i)_{0\leq i\leq n}$ if the source and target are implicit.

\medskip

The double category $\mathscr{B}_n=[\operatorname{mor}\mathscr{B}_n\rightrightarrows X]$ has the following structure maps: the source and target maps are the obvious projections
\begin{align*}
s(\operatorname{id}_x,\operatorname{id}_y,\beta_i)=x,\quad s(f,g,\alpha_i)=f\quad \text{and}\quad t(\operatorname{id}_x,\operatorname{id}_y,\beta_i)=y, \quad t(f,g,\alpha_i)=g,
\end{align*}
the identity section sends an object $x$ in $X$ to the sequence
\begin{align*}
\bigg[(x,x,e,\operatorname{id}_x)\xlongequal{(\operatorname{id}_x,\operatorname{id}_x,\operatorname{id}_e)}(x,x,e,\operatorname{id}_x)\bigg]_{1\leq i\leq n}
\end{align*}
and it sends a morphism $h\colon x\rightarrow y$ to the morphism $(h,h,\operatorname{id}_e)_{0\leq i\leq n}$.

\medskip

Vertical composition is given by vertical composition in $\mathscr{M}\ltimes X$. More precisely, the vertical composite of
\begin{align*}
\bigg[(x,y,m_{i-1},\phi_{i-1})\xrightarrow{\ (\operatorname{id}_x,\operatorname{id}_y,\beta_i)\ } (x,y,m_i, \phi_i)\bigg]_{1\leq i\leq n}
\end{align*}
and
\begin{align*}
\bigg[(y,z,m_{i-1}',\phi_{i-1}')\xrightarrow{\ (\operatorname{id}_y,\operatorname{id}_z,\beta_i')\ } (y,z,m_i', \phi_i')\bigg]_{1\leq i\leq n}
\end{align*}
is the sequence
\begin{align*}
\bigg[(x,z,m_{i-1}'m_{i-1},&\phi_{i-1}'\circ m_{i-1}'\phi_{i-1})\xrightarrow{(\operatorname{id}_x,\operatorname{id}_z,\beta'_i\beta_i)}(x,z,m_i'm_i,\phi_i'\circ m_i'\phi_i)\bigg]_{1\leq i\leq n}.
\end{align*}

The vertical composite of $2$-cells is given by
\begin{align*}
(f,g,\alpha_i)_{0\leq i\leq n}\circ_v(g,h,\alpha_i')_{0\leq i\leq n}=(f,h,\alpha_i'\alpha_i)_{0\leq i\leq n}
\end{align*}
for composable $2$-cells.
\exend
\end{construction}

We define a functor $\iota_n\colon M.X\rightarrow \operatorname{mor} \mathscr{B}_n$ sending an object $(x,y,m,\phi)$ of $M.X$ to the sequence
\begin{align}\label{spanning objects}
\bigg[(x,y,m,\phi)\xlongequal{(\operatorname{id}_x,\operatorname{id}_y,\operatorname{id}_m)}(x,y,m,\phi)\bigg]_{1\leq i\leq n},
\end{align}
and a morphism $(f,g,\alpha)\colon (x,y,m,\phi)\rightarrow (x',y',m',\phi')$ to $(f,g, \alpha)_{0\leq i\leq n}$. The functor $\iota_n$ can be thought of as inclusion at zero and it identifies $M.X$ with the full subcategory of $\operatorname{mor} \mathscr{B}_n$ on the objects of the form (\ref{spanning objects}). In fact, if $n=0$, this is an equality $M.X=\operatorname{mor} \mathscr{B}_0$.

\medskip

Consider the double functor $I_n\colon \mathscr{M}\ltimes X\hookrightarrow \mathscr{B}_n$ which restricts to the identity on object categories and is given by the embedding $\iota_n\colon M.X\hookrightarrow \operatorname{mor}\mathscr{B}_n$ on morphism categories.

\medskip

The proof of the following lemma resembles the proof of Waldhausen's swallowing lemma (\cite[Lemma 1.6.5]{Waldhausen}). We show that $\mathscr{B}_n$ retracts to $\mathscr{M}\ltimes X$.

\begin{lemma}\label{iota_n induces homotopy equivalence}
The double functor $I_n\colon \mathscr{M}\ltimes X\hookrightarrow \mathscr{B}_n$ induces a homotopy equivalence of geometric realisations.
\end{lemma}
\begin{proof}
Define a functor $\rho_n\colon \operatorname{mor}\mathscr{B}_n\rightarrow M.X$ given by retraction to zero, that is, an object
\begin{align*}
\bigg[(x,y,m_{i-1},\phi_{i-1})\xrightarrow{(\operatorname{id}_x,\operatorname{id}_y,\beta_i)}(x,y,m_i,\phi_i) \bigg]_{1\leq i\leq n}
\end{align*}
is mapped to $(x,y,m_0,\phi_0)$, and a morphism
\vspace*{-1.5em}
\begin{center}
\[\left[
\begin{tikzpicture}[baseline=(current bounding box.center)]
\matrix (m) [matrix of math nodes,row sep=2em,column sep=4em, text height=1ex, text depth=0.25ex]
  {
(x,y,m_{i-1},\phi_{i-1}) & (x,y,m_i,\phi_i) \\
(x',y',m'_{i-1},\phi'_{i-1})& (x',y',m'_i,\phi'_i)\\
  };
   \path[-stealth]
   (m-1-1) edge node[above]{$\scriptstyle(\operatorname{id}_x,\operatorname{id}_y,\beta_i)$} (m-1-2) 
   (m-2-1) edge node[below]{$\scriptstyle(\operatorname{id}_{x'},\operatorname{id}_{y'},\beta'_i)$} (m-2-2) 
   (m-1-1) edge node[left]{$\scriptstyle(f,g,\alpha_{i-1})$} (m-2-1)
   (m-1-2) edge node[right]{$\scriptstyle(f,g,\alpha_i)$} (m-2-2)
;
\end{tikzpicture}
\right]_{1\leq i\leq n}\]
\end{center}
is mapped to $(x,y,m_0,\phi_0) \xrightarrow{(f,g,\alpha_0)} (x',y',m'_0,\phi'_0)$.

\medskip

The composite $\rho_n\circ \iota_n$ is the identity and the morphisms
\vspace*{-1.5em}
\begin{center}
\[\left[
\begin{tikzpicture}[baseline=(current bounding box.center)]
\matrix (m) [matrix of math nodes,row sep=2em,column sep=4em, text height=1ex, text depth=0.25ex]
  {
(x,y,m_0,\phi_0) & (x,y,m_0,\phi_0) \\
(x,y,m_{i-1},\phi_{i-1})& (x,y,m_i,\phi_i)\\
  };
  \path[-]
(m-1-1) edge[double equal sign distance] node[above]{$\scriptstyle(\operatorname{id}_x,\operatorname{id}_y,\operatorname{id}_{m_0})$} (m-1-2)   
  ;
   \path[-stealth]
   (m-2-1) edge node[below]{$\scriptstyle(\operatorname{id}_x,\operatorname{id}_y,\beta_i)$} (m-2-2) 
   (m-1-1) edge node[left]{$\scriptstyle(\operatorname{id}_x,\operatorname{id}_y,\beta_{i-1}\circ \cdots\circ \beta_1)$} (m-2-1)
   (m-1-2) edge node[right]{$\scriptstyle(\operatorname{id}_x,\operatorname{id}_y,\beta_i\circ \cdots \circ \beta_1)$} (m-2-2)
;
\end{tikzpicture}
\right]_{1\leq i\leq n}\]
\end{center}
define a counit transformation $\iota_n\circ \rho_n\Rightarrow \operatorname{id}$. The functor $\rho_n$ and the unit and counit transformations sit above the identity on $X$, and it follows that $\rho_n$ induces a morphism of vertical nerves which at each simplicial level is right adjoint to the morphism induced by $\iota_n$. This proves the claim.
\end{proof}

\subsection{Enlarging the action \texorpdfstring{$2$}{2}-category}

Fix $n\geq 0$. Consider the vertical $2$-category $Q(M,X)$ of the double category $\mathscr{M}\ltimes X=[M.X\rightrightarrows X]$ (\Cref{action Q-construction}). We define a $2$-category $\mathscr{A}_n$ and a pseudofunctor $Q(M,X)\rightarrow \mathscr{A}_n$ which induces a homotopy equivalence of geometric realisations. As with the double category $\mathscr{B}_n$ in the previous section, there will be some redundant data in the notation, but this should make the final comparison somewhat clearer, and the reader might already notice parallels with the $\mathscr{B}_n$ construction.

\begin{construction}\label{lax action 2-category}
The $2$-category $\mathscr{A}_n$ is given as follows. The object set is $N_n(X)$ and the morphisms are elements in $N_n(M.X)$ with source and target maps inherited from $\mathscr{M}\ltimes X$. In other words, the objects are sequences $x_0\rightarrow x_1 \rightarrow \cdots \rightarrow x_n$ in $X$, and a morphism
\begin{align*}
(x_0\xrightarrow{f_1}x_1\xrightarrow{f_2}\cdots \xrightarrow{f_n}x_n)\ \longrightarrow \ (y_0\xrightarrow{g_1}y_1\xrightarrow{g_2}\cdots \xrightarrow{g_n}y_n)
\end{align*}
is a sequence
\begin{align*}
(x_0,y_0,m_0,\phi_0)\xrightarrow{(f_1,g_1,\alpha_1)}(x_1,y_1,m_1,\phi_1)
\xrightarrow {(f_2,g_2,\alpha_2)}\cdots \xrightarrow{(f_n,g_n,\alpha_n)}(x_n,y_n,m_n,\phi_n)
\end{align*}
in $M.X$. Finally, the $2$-cells are given by commutative diagrams in $M.X$ as pictured below.
\begin{center}
\begin{tikzpicture}
\matrix (m) [matrix of math nodes,row sep=2em,column sep=3em, text height=1.5ex, text depth=0.25ex]
  {
\scriptstyle(x_0,y_0,m_0,\phi_0) & \scriptstyle(x_1,y_1,m_1,\phi_1) & \ & \ & \scriptstyle(x_n,y_n,m_n,\phi_n) \\
\scriptstyle(x_0,y_0,m'_0,\phi'_0)& \scriptstyle(x_1,y_1,m'_1,\phi'_1)& \ & \ & \scriptstyle(x_n,y_n,m'_n,\phi'_n) \\
  };
   \path[-stealth]
   (m-1-1) edge node[above]{$\scriptstyle(f_1,g_1,\alpha_1)$} (m-1-2) 
   (m-1-2) edge node[above]{$\scriptstyle(f_2,g_2,\alpha_2)$} (m-1-3)
   (m-1-4) edge node[above]{$\scriptstyle(f_n,g_n,\alpha_n)$} (m-1-5)
   (m-2-1) edge node[below]{$\scriptstyle(f_1,g_1,\alpha'_1)$} (m-2-2) 
   (m-2-2) edge node[below]{$\scriptstyle(f_2,g_2,\alpha'_2)$} (m-2-3)
   (m-2-4) edge node[below]{$\scriptstyle(f_n,g_n,\alpha'_n)$} (m-2-5)
   (m-1-1) edge node[left]{$\scriptstyle(\operatorname{id}_{x_0},\operatorname{id}_{y_0},\beta_0)$} (m-2-1)
   (m-1-2) edge node[left]{$\scriptstyle(\operatorname{id}_{x_1},\operatorname{id}_{y_1},\beta_1)$} (m-2-2)
   (m-1-5) edge node[left]{$\scriptstyle(\operatorname{id}_{x_n},\operatorname{id}_{y_n},\beta_n)$} (m-2-5)
   ;
	\path[dotted]
		(m-1-3) edge (m-1-4)
		(m-2-3) edge (m-2-4)	
	;
\end{tikzpicture}
\end{center}

To ease notation, we will denote a morphism by 
\begin{align*}
\bigg[ (x_{i-1},y_{i-1},m_{i-1},\phi_{i-1})\xrightarrow{(f_i,g_i,\alpha_i)} (x_i,y_i,m_i,\phi_i)\bigg]_{1\leq i\leq n}
\end{align*}
or simply $(f_i,g_i,\alpha_i)_{1\leq i\leq n}$ if the objects are implicit. Similarly, we denote a $2$-cell by
\vspace*{-1.5em}
\begin{center}
\[\left[
\begin{tikzpicture}[baseline=(current bounding box.center)]
\matrix (m) [matrix of math nodes,row sep=2em,column sep=4em, text height=1ex, text depth=0.25ex]
  {
(x_{i-1},y_{i-1},m_{i-1},\phi_{i-1}) & (x_i,y_i,m_i,\phi_i) \\
(x_{i-1},y_{i-1},m_{i-1}',\phi_{i-1}') & (x_i,y_i,m_i',\phi_i')\\
  };
  \path[-stealth]
	(m-1-1) edge node[above]{$\scriptstyle(f_i,g_i,\alpha_i)$} (m-1-2)   
   (m-2-1) edge node[below]{$\scriptstyle(f_i,g_i,\alpha'_i)$} (m-2-2) 
   (m-1-1) edge node[left]{$\scriptstyle(\operatorname{id}_{x_{i-1}},\operatorname{id}_{y_{i-1}},\beta_{i-1})$} (m-2-1)
   (m-1-2) edge node[right]{$\scriptstyle(\operatorname{id}_{x_i},\operatorname{id}_{y_i},\beta_i)$} (m-2-2)
;
\end{tikzpicture}
\right]_{1\leq i\leq n}\]
\end{center}
or simply by $(\operatorname{id}_{x_i}, \operatorname{id}_{y_i},\beta_i)_{0\leq i\leq n}$ if the objects and morphisms are implicit.

\medskip

Composition of morphisms is given by vertical composition in $\mathscr{M}\ltimes X$, that is, the composite of
\begin{align*}
\bigg[ (x_{i-1},y_{i-1},m_{i-1},\phi_{i-1})\xrightarrow{(f_i,g_i,\alpha_i)} (x_i,y_i,m_i,\phi_i)\bigg]_{1\leq i\leq n}
\end{align*}
and
\begin{align*}
\bigg[ (y_{i-1},z_{i-1},m'_{i-1},\phi'_{i-1})\xrightarrow{(g_i,h_i,\alpha'_i)} (y_i,z_i,m'_i,\phi'_i)\bigg]_{1\leq i\leq n}
\end{align*}
is the sequence
\begin{align*}
\bigg[(x_{i-1},z_{i-1},m'_{i-1}m_{i-1},&\phi'_{i-1}\circ m'_{i-1}\phi_{i-1})\xrightarrow{(f_i,h_i,\alpha'_i\alpha_i)}(x_i,z_i,m'_im_i,\phi'_i\circ m'_i\phi_i)\bigg]_{1\leq i\leq n}.
\end{align*}

Composition of the $2$-cells along morphisms (within the hom-categories) is given by horizontal composition in $\mathscr{M}\ltimes X$, that is, by composition of morphisms in $M$:
\begin{align*}
(\operatorname{id}_{x_i},\operatorname{id}_{y_i},\beta_i')_{0\leq i\leq n}\circ_h(\operatorname{id}_{x_i},\operatorname{id}_{y_i},\beta_i)_{0\leq i\leq n}=(\operatorname{id}_{x_i},\operatorname{id}_{y_i},\beta_i'\circ\beta_i)_{0\leq i\leq n}
\end{align*}
for composable morphisms.

\medskip

Composition of the $2$-cells along objects is given by vertical composition in $\mathscr{M}\ltimes X$, that is, by the product in $M$:
\begin{align*}
(\operatorname{id}_{y_i},\operatorname{id}_{z_i},\beta_i')_{0\leq i\leq n} \circ_v (\operatorname{id}_{x_i},\operatorname{id}_{y_i},\beta_i)_{0\leq i\leq n}=(\operatorname{id}_{x_i},\operatorname{id}_{z_i},\beta_i'\beta_i)_{0\leq i\leq n}
\end{align*}
for composable morphisms.
\exend
\end{construction}

Consider the (strict) pseudofunctor $\Upsilon_n\colon Q(M,X)\rightarrow \mathscr{A}_n$ given by inclusion at zero; that is, it sends an object $x$ to the sequence $x=x=\cdots=x$, a morphism $(x,y,m,\phi)$ to the sequence
\begin{align*}
\bigg[(x,y,m,\phi)\xlongequal{(\operatorname{id}_x,\operatorname{id}_y,\operatorname{id}_m)}(x,y,m,\phi)\bigg]_{1\leq i\leq n}
\end{align*}
and a 2-cell $(\operatorname{id}_x,\operatorname{id}_y,\beta)\colon (x,y,m,\phi)\rightarrow (x,y,m',\phi')$ to the $2$-cell $(\operatorname{id}_x,\operatorname{id}_y,\beta)_{0\leq i\leq n}$. If $n=0$, this is an equality $Q(M,X)=\mathscr{A}_0$.

\medskip

As was the case for \Cref{Phi induces homotopy equivalence}, the proof of the following lemma resembles the proof of Waldhausen's swallowing lemma (\cite[Lemma 1.6.5]{Waldhausen}). We show that $\mathscr{A}_n$ retracts onto $Q(M,X)$.

\begin{lemma}\label{I_n induces homotopy equivalence}
The pseudofunctor $\Upsilon_n\colon Q(M,X)\rightarrow \mathscr{A}_n$ admits a right $2$-adjoint. In particular, it induces a homotopy equivalence of geometric realisations.
\end{lemma}
\begin{proof}
Consider the (strict) pseudofunctor $R_n\colon \mathscr{A}_n\rightarrow Q(M,X)$ which sends an object
\begin{align*}
x_0\rightarrow x_1\rightarrow \cdots \rightarrow x_n
\end{align*}
to $x_0$, a morphism
\begin{align*}
(x_0,y_0,m_0,\phi_0)\rightarrow (x_1,y_1,m_1,\phi_1)\rightarrow\, \cdots\, \rightarrow (x_n,y_n,m_n,\phi_n)
\end{align*}
to $(x_0,y_0,m_0,\phi_0)$ and a 2-cell $(\operatorname{id}_{x_i},\operatorname{id}_{y_i}, \beta_i)_{0\leq i\leq n}$ to
\begin{align*}
(\operatorname{id}_{x_0},\operatorname{id}_{y_0},\beta_0)\colon (x_0,y_0,m_0,\phi_0)\rightarrow (x_0,y_0,m_0',\phi_0').
\end{align*}

The composite $R_n\circ \Upsilon_n$ is equal to the identity, and for the other composite, we define an oplax natural transformation $\epsilon\colon \Upsilon_n\circ R_n\Rightarrow \operatorname{id}$. Given an object $x_0\xrightarrow{f_1}x_1\xrightarrow{f_2}\cdots \xrightarrow{f_n}x_n$, consider the morphism
\begin{align*}
\epsilon_{f_i}\colon (x_0=x_0=\cdots =x_0) \ \longrightarrow \ (x_0\xrightarrow{f_1}x_1\xrightarrow{f_2}\cdots \xrightarrow{f_n}x_n)
\end{align*}
in $\mathscr{A}_n$ given by the sequence
\begin{align*}
\bigg[
(x_0,x_{i-1},e,f_{i-1}\circ \cdots \circ f_1)&\xrightarrow{(\operatorname{id}_{x_0},f_i,\operatorname{id}_e)}(x_0,x_i,e,f_i\circ \cdots \circ f_1)
\bigg]_{1\leq i\leq n}.
\end{align*}

For a morphism $F\colon (x_0\xrightarrow{f_1}x_1\xrightarrow{f_2}\cdots \xrightarrow{f_n}x_n)\longrightarrow (y_0\xrightarrow{g_1}y_1\xrightarrow{g_2}\cdots \xrightarrow{g_n}y_n)$ given by a sequence
\begin{align*}
\bigg[(x_{i-1},y_{i-1},m_{i-1},\phi_{i-1})\xrightarrow{(f_i,g_i,\alpha_i)}(x_i,y_i,m_i,\phi_i)
\bigg]_{1\leq i\leq n},
\end{align*}
consider the $2$-cell $A_F\colon \epsilon_{g_i}\circ (\Upsilon_n\circ R_n )(F)\rightarrow F\circ \epsilon_{f_i}$ given by
\vspace*{-1.5em}
\begin{center}
\[\left[\!\!
\begin{tikzpicture}[baseline=(current bounding box.center)]
\matrix (m) [matrix of math nodes,row sep=2em,column sep=3em, text height=1ex, text depth=0.25ex]
  {
\scriptstyle(x_0,y_{i-1},m_0,g_{i-1}\circ \cdots \circ g_1\circ \phi_0) & \scriptstyle(x_0,y_i,m_0,g_i\circ \cdots \circ g_1\circ\phi_0) \\
\scriptstyle(x_0,y_{i-1},m_{i-1},\phi_{i-1}\circ m_{i-1}(f_{i-1}\circ \cdots \circ f_1)) & \scriptstyle(x_0,y_i,m_i,\phi_i\circ m_i(f_i\circ \cdots \circ f_1))\\
  };
  \path[-stealth]
	(m-1-1) edge node[above]{$\scriptstyle(\operatorname{id}_{x_0},g_i,\operatorname{id}_{m_0})$} (m-1-2)   
   (m-2-1) edge node[below]{$\scriptstyle(\operatorname{id}_{x_0},g_i,\alpha_i)$} (m-2-2) 
   (m-1-1) edge node[left]{$\scriptstyle(\operatorname{id}_{x_0},\operatorname{id}_{y_{i-1}},\alpha_{i-1})$} (m-2-1)
   (m-1-2) edge node[right]{$\scriptstyle(\operatorname{id}_{x_0},\operatorname{id}_{y_i},\alpha_i)$} (m-2-2)
;
\end{tikzpicture}
\!\!\right]_{1\leq i\leq n}\]
\end{center}

The $2$-cells $A_F$ assemble to define natural transformations $(\epsilon_{g_i})_*\circ (\Upsilon_n\circ R_n)\Rightarrow (\epsilon_{f_i})^*$, and they respect identities and composition. Hence, we have an oplax natural transformation $\epsilon\colon \Upsilon_n\circ R_n\Rightarrow \operatorname{id}$. One can check that the triangle identities are satisfied.
\end{proof}

\subsection{Comparing geometric realisations}

Consider the double category $\mathscr{M}\ltimes X$, its vertical $2$-category $Q(M,X)$ and for all $n\geq 0$, the $2$-category $\mathscr{A}_n$ and the double category $\mathscr{B}_n$ as constructed in the previous sections. For the remainder of this section, we consider $\mathscr{A}_n$ as a double category with only identity horizontal morphisms.

\begin{construction}\label{simplicial double categories}
We can define two simplicial double categories $\mathscr{A}_\bullet$ and $\mathscr{B}_\bullet$ by defining the simplicial structure maps as below --- they are defined in the ``obvious'' way, but to be precise we write them out.

\medskip

For $\theta\colon [k]\rightarrow [n]$ in $\Delta$, the structure map $\theta^*\colon \mathscr{A}_n\rightarrow \mathscr{A}_k$ is given by the usual structure map $\theta^*\colon N_n(X)\rightarrow N_k(X)$ on object categories (recall that we interpret it as a double category with discrete object category), and on the morphism categories it is given on objects by the usual structure map
\begin{align*}
\theta^*\colon N_n(M. X)\rightarrow N_k(M. X)
\end{align*}
and on morphisms by removing or repeating the $\beta_i$'s accordingly.

\medskip

The structure map $\theta^*\colon \mathscr{B}_n\rightarrow \mathscr{B}_k$ is given by the identity $X\rightarrow X$ on object categories and on morphism categories by the usual structure map
\begin{align*}
\theta^*\colon \coprod_{x,y}N_n(M(x,y))\rightarrow \coprod_{x,y}N_k(M(x,y))
\end{align*}
on objects and on morphisms by removing or repeating the $\alpha_i$'s accordingly.

\medskip

Define two bisimplicial categories $\mathscr{A}_{\bullet\bullet}$ and $\mathscr{B}_{\bullet\bullet}$ by applying the horizontal nerve functor levelwise: $\mathscr{A}_{nk}=N^h_k(\mathscr{A}_n)$, and $\mathscr{B}_{nk}=N^h_k(\mathscr{B}_n)$.
\exend
\end{construction}

\begin{lemma}\label{same nerves up to transposition}
The bisimplicial category $\mathscr{A}_{\bullet\bullet}$ is isomorphic to the transpose of the bisimplicial category $\mathscr{B}_{\bullet\bullet}$.
\end{lemma}
\begin{proof}
This is a case of writing out the definitions. We verify that the objects and morphisms coincide and leave composition and simplicial structure maps to the reader. Let $n,k\geq 0$ and consider the categories $\mathscr{A}_{nk}=N^h_k(\mathscr{A}_n)$ and $\mathscr{B}_{kn}=N^h_n(\mathscr{B}_k)$.

\medskip

Since we only have identity horizontal morphisms in $\mathscr{A}_n$, the object set of $\mathscr{A}_{nk}$ is in bijection with the object set $N_n(X)$ of $\mathscr{A}_n$. The object set of $\mathscr{B}_{kn}$ is the set of $n$-simplices in the nerve of the object category of $\mathscr{B}_k$, i.e.~$N_n(X)$.

\medskip

A morphism in $\mathscr{A}_{nk}$ is a sequence of $k$ vertical morphisms in $\mathscr{A}_n$ connected by $2$-cells: the diagram below is a morphism from $x_0\xrightarrow{f_1}\cdots \xrightarrow{f_n}x_n$ to $y_0\xrightarrow{g_1}\cdots \xrightarrow{g_n}y_n$, where the horizontal sequences come from the morphisms and the vertical sequences from the $2$-cells. A morphism in $\mathscr{B}_{kn}$ is a sequence of $n$ vertical morphisms in $\mathscr{B}_k$ connected by $2$-cells: the diagram below is a morphism from $x_0\xrightarrow{f_1}\cdots \xrightarrow{f_n}x_n$ to $y_0\xrightarrow{g_1}\cdots \xrightarrow{g_n}y_n$, but now the vertical sequences come from the morphisms and the horizontal sequences from the $2$-cells.

\begin{center}
\begin{tikzpicture}
\matrix (m) [matrix of math nodes,row sep=2em,column sep=3em, text height=1.5ex, text depth=0.25ex]
  {
(x_0,y_0,m^0_0,\phi^0_0) & (x_1,y_1,m^0_1,\phi^0_1) & \cdots & (x_n,y_n,m^0_n,\phi^0_n) \\
(x_0,y_0,m^1_0,\phi^1_0)& (x_1,y_1,m^1_1,\phi^1_1)& \cdots &(x_n,y_n,m^1_n,\phi^1_n) \\
\vdots & \vdots & \ddots & \vdots \\
(x_0,y_0,m^k_0,\phi^k_0)& (x_1,y_1,m^k_1,\phi^k_1)& \cdots &(x_n,y_n,m^k_n,\phi^k_n) \\
  };
   \path[-stealth]
   (m-1-1) edge node[above]{$\scriptstyle(f_1,g_1,\alpha^0_1)$} (m-1-2) 
   (m-1-2) edge node[above]{$\scriptstyle(f_2,g_2,\alpha^0_2)$} (m-1-3)
   (m-1-3) edge node[above]{$\scriptstyle(f_n,g_n,\alpha^0_n)$} (m-1-4)
   (m-2-1) edge node[below]{$\scriptstyle(f_1,g_1,\alpha^1_1)$} (m-2-2) 
   (m-2-2) edge node[below]{$\scriptstyle(f_2,g_2,\alpha^1_2)$} (m-2-3)
   (m-2-3) edge node[below]{$\scriptstyle(f_n,g_n,\alpha^1_n)$} (m-2-4)
   (m-4-1) edge node[below]{$\scriptstyle(f_1,g_1,\alpha^k_1)$} (m-4-2) 
   (m-4-2) edge node[below]{$\scriptstyle(f_2,g_2,\alpha^k_2)$} (m-4-3)
   (m-4-3) edge node[below]{$\scriptstyle(f_n,g_n,\alpha^k_n)$} (m-4-4)
   (m-1-1) edge node[left]{$\scriptstyle(\operatorname{id}_{x_0},\operatorname{id}_{y_0},\beta^1_0)$} (m-2-1)
   (m-1-2) edge node[left]{$\scriptstyle(\operatorname{id}_{x_1},\operatorname{id}_{y_1},\beta^1_1)$} (m-2-2)
   (m-1-4) edge node[left]{$\scriptstyle(\operatorname{id}_{x_n},\operatorname{id}_{y_n},\beta^1_n)$} (m-2-4)
   (m-2-1) edge node[left]{$\scriptstyle(\operatorname{id}_{x_0},\operatorname{id}_{y_0},\beta^2_0)$} (m-3-1)
   (m-2-2) edge node[left]{$\scriptstyle(\operatorname{id}_{x_1},\operatorname{id}_{y_1},\beta^2_1)$} (m-3-2)
   (m-2-4) edge node[left]{$\scriptstyle(\operatorname{id}_{x_n},\operatorname{id}_{y_n},\beta^2_n)$} (m-3-4)
   (m-3-1) edge node[left]{$\scriptstyle(\operatorname{id}_{x_0},\operatorname{id}_{y_0},\beta^k_0)$} (m-4-1)
   (m-3-2) edge node[left]{$\scriptstyle(\operatorname{id}_{x_1},\operatorname{id}_{y_1},\beta^k_1)$} (m-4-2)
   (m-3-4) edge node[left]{$\scriptstyle(\operatorname{id}_{x_n},\operatorname{id}_{y_n},\beta^k_n)$} (m-4-4)	
	;
\end{tikzpicture}
\end{center}

In both cases, composition is given by vertical composition of morphisms and $2$-cells in $\mathscr{M}\ltimes X$, and one can verify that these coincide. Likewise, one can check that the simplicial structure maps can be identified.
\end{proof}

We can combine this with the homotopy equivalences of the previous sections to show that the geometric realisations of $\mathscr{M}\times X$, $\mathscr{M}\ltimes X$ and $Q(M,X)$ are homotopy equivalent.

\begin{theorem}\label{classifying spaces of lax action double cat and action 2-cat are htpy eq}
The inclusion $Q(M,X) \rightarrow \mathscr{M}\ltimes X$ induces a homotopy equivalence of geometric realisations.
\end{theorem}
\begin{proof}
By Lemmas \ref{iota_n induces homotopy equivalence}, \ref{I_n induces homotopy equivalence} and \ref{same nerves up to transposition}, we have a diagram as below, where the homotopy equivalences on the left and right are induced by maps of simplicial double categories whose sources are constant simplicial objects (in fact, simply the inclusion of the zero simplices of the target):
\begin{align*}
|\mathscr{M}\ltimes X|\xrightarrow{\ \simeq\ } |\mathscr{B}_\bullet|\cong |\mathscr{B}_{\bullet\bullet}|\cong |\mathscr{A}_{\bullet\bullet}|\cong |\mathscr{A}_\bullet|\xleftarrow{\ \simeq\ } |Q(M,X)|.
\end{align*}
To see that this homotopy equivalence is induced by the inclusion, we analyse the diagonal instead of collapsing to the horizontal and vertical axes. This leaves us with a zig-zag of simplicial categories, which levelwise fits into the diagram below, where the left vertical map is the one induced by $Q(M,X)\rightarrow \mathscr{M}\ltimes X$.
\begin{center}
\begin{tikzpicture}
\matrix (m) [matrix of math nodes,row sep=2em,column sep=3em, text height=1.5ex, text depth=0.25ex]
  {
	N_n^h(\mathscr{M}\ltimes X) & N^h_n(\mathscr{B}_n) \\
	N_n^h(Q(M,X)) & N^h_n(\mathscr{A}_n) \\
  };
   \path[-stealth]
	(m-1-1) edge (m-1-2)
	(m-2-1) edge (m-2-2)
	;
	\path[left hook-stealth]
(m-2-1) edge (m-1-1)	
	;
	\path[-]
	(m-1-2) edge[double distance=2pt] (m-2-2)
	;
\end{tikzpicture}
\end{center}
Tracing through the definitions, we see that this diagram commutes for all $n\geq 0$ and the claim follows.
\end{proof}

Combined with the homotopy equivalence of \Cref{Phi induces homotopy equivalence}, we have the following immediate corollary.

\begin{corollary}\label{classifying spaces of strict action double cat and action 2-cat are htpy eq}
The zig-zag of double functors
\begin{align*}
\mathscr{M}\times X \rightarrow \mathscr{M}\ltimes X \leftarrow Q(M,X)
\end{align*}
induces a homotopy equivalence of geometric realisations.
\end{corollary}

The geometric realisation $|M|$ is naturally a topological monoid and the action of $M$ on $X$ defines an action of $|M|$ on the geometric realisation $|X|$.

\begin{corollary}
The homotopy quotient $|X|_{|M|}$ of the action of $|M|$ on $|X|$ is homotopy equivalent to the geometric realisation of $Q(M,X)$.
\end{corollary}
\begin{proof}
The vertical nerve of the double category $\mathscr{M}\times X=[M\times X\rightrightarrows X]$ is the simplicial category whose category of $n$-simplices is $M^n\times X$ with the usual simplicial structure maps given by projection, action and product. Since geometric realisation commutes with finite products, the realisation of this is the homotopy quotient $|X|_{|M|}$.
\end{proof}

Finally, we consider the special case that we will need in following section.

\begin{observation}\label{M acting on itself}
Let $(M,\otimes)$ be a strict monoidal category and consider the monoidal category $M\times M^{\otimes^{\operatorname{op}}}$, where the second factor is the category $M$ with the opposite product: $a\otimes^{\operatorname{op}}b=b\otimes a$. The monoidal category $M\times M^{\otimes^{\operatorname{op}}}$ acts on the category $M$ by left and right multiplication: $(a,b).m=a\otimes m\otimes b$ for all objects $a,b,m$ in $M$. We see that the $2$-category
\begin{align*}
Q_2(M):=Q(M\times M^{\otimes^{\operatorname{op}}},M)
\end{align*}
is given as follows:
\begin{itemize}[label=$\ast$]
\item the objects are those of $M$, 
\item a morphism $m\rightarrow m'$ is a tuple $(a,b,\phi\colon amb\rightarrow m')$ where $a$ and $b$ are objects of $M$ and $\phi$ is a morphism in $M$,
\item  and a $2$-cell $(a,b,\phi)\rightarrow (a',b',\phi')$ is a pair of morphisms $\alpha\colon a\rightarrow a'$, $\beta\colon b\rightarrow b'$ in $M$ such that $\phi=\phi'\circ (\alpha \,\operatorname{id}_m\,\beta)$,
\end{itemize}

and the various composites are as follows:
\begin{itemize}[label=$\ast$]
\item for morphisms:
\begin{center}
\begin{tikzpicture}
\matrix (m) [matrix of math nodes,row sep=2em,column sep=3em,, text height=1.5ex, text depth=0.25ex]
  {
	m & \\   
	m' & m'' \\   
  };
  \path[-stealth] 
	(m-1-1) edge node[left]{$(c,d,\psi)$} (m-2-1) edge node[above right]{$(ac, db, \phi\circ (\operatorname{id}_a\,\psi\,\operatorname{id}_b))$} (m-2-2) 
	(m-2-1) edge node[below]{$(a,b,\phi)$} (m-2-2) 
  ;
\end{tikzpicture}
\end{center}

\item composition of $2$-cells along morphisms (i.e.~within hom-categories) is given by coordinatewise composition:
\begin{center}
\begin{tikzpicture}
\matrix (m) [matrix of math nodes,row sep=2em,column sep=3em, text height=1.5ex, text depth=0.25ex]
  {
	(a,b,\phi) & \\   
	(a',b',\phi') & (a'',b'',\phi'') \\   
  };
  \path[-stealth] 
	(m-1-1) edge node[left]{$(\alpha,\beta)$} (m-2-1) edge node[above right]{$(\alpha'\circ \alpha,\beta'\circ\beta)$} (m-2-2) 
	(m-2-1) edge node[below]{$(\alpha',\beta')$} (m-2-2) 
  ;
\end{tikzpicture}
\end{center}

\item composition of $2$-cells along objects is given by the monoidal product:  the composite of the following $2$-cells
\begin{center}
\begin{tikzcd}[row sep=2em,column sep=5em]
 m \arrow[bend left=35]{r}[name=u,below]{}[above]{(c,d,\psi)}
    \arrow[bend right=35]{r}[name=d]{}[below]{(c', d',\psi')}
    \arrow[Rightarrow,to path=(u) -- (d)\tikztonodes]{r}{\,(\gamma,\delta)}
    & 
    m' \arrow[bend left=35]{r}[name=u,below]{}[above]{(a,b,\phi)}
    \arrow[bend right=35]{r}[name=d]{}[below]{(a',b',\phi')}
    \arrow[Rightarrow,to path=(u) -- (d)\tikztonodes]{r}{\,(\alpha,\beta)}
    & 
    m''
\end{tikzcd}
\end{center}
is the $2$-cell
\begin{align*}
(\alpha\gamma,\delta\beta)\colon (ac, db, \phi\circ (\operatorname{id}_a\,\psi\,\operatorname{id}_b))\longrightarrow (a'c', d'b', \phi'\circ (\operatorname{id}_{a'}\,\psi'\,\operatorname{id}_{b'}))
\end{align*}
\end{itemize}
Note that this also makes sense if $M$ isn't small.\exend
\end{observation}

Let $M$ be a small monoidal category and $|M|$ the topological monoid given by the geometric realisation of $M$. Recall that the classifying space $B|M|$ of $|M|$ is the total geometric realisation of the standard bar construction $B_\bullet M$ whose category of $n$-simplices is $M^n$ and whose simplicial structure maps are given by the monoidal structure of $M$. If $M$ is an essentially small monoidal category, then we denote by $|M|$ and $B|M|$ the topological monoid and its classifying space defined as above for some equivalent small monoidal category.

\begin{corollary}\label{Q construction and monoidal category have homotopy equivalent classifying spaces}
Let $M$ be an essentially small strict monoidal category. The classifying space $B|M|$ of the topological monoid $|M|$ is homotopy equivalent to the geometric realisation of the $2$-category $Q_2(M)$ of \Cref{M acting on itself}.
\end{corollary}
\begin{proof}
We may assume $M$ to be small. The vertical nerve of the double category
\begin{align*}
[(M\times M^{\otimes^{\operatorname{op}}})\times M\rightrightarrows M]
\end{align*}
given by the action of $M\times M^{\otimes^{\operatorname{op}}}$ on $M$ is the edgewise subdivision of the bar construction $B_\bullet M$. The claim then follows from \Cref{classifying spaces of strict action double cat and action 2-cat are htpy eq} together with the fact that the geometric realisation of the edgewise subdivision is homeomorphic to the geometric realisation of the original simplicial space (\cite[Proposition A.1]{Segal73}).
\end{proof} 

\begin{remark}\label{Segal}
The above result should be compared with the classical result of Segal: the classifying space of a topological monoid $M$ is homeomorphic to the geometric realisation of a topological category $\mathscr{C}(M)$ with objects the objects of $M$ and morphisms $(a,b)\colon m\rightarrow m'$ where $a$ and $b$ are objects of $M$ such that $amb=m'$ (\cite[Proposition 2.5]{Segal73}).
\exend
\end{remark}

\section{Comparison with (stable) algebraic K-theory}\label{comparison with algebraic K-theory}

In this last section of the paper, we compare the categories $\operatorname{RBS}(M)$ with the (stable) algebraic K-theory space. We have already remarked on this at the beginning of the previous section in order to motivate the results proved there. To recap, we associate to any exact category $\mathscr{E}$, a strict monoidal category $M_{\mathscr{E}}$ of flags and associated gradeds, and when $\mathscr{E}=\mathcal{P}(A)$ is the exact category of finitely generated projective modules over an associative ring $A$, then the monoidal category $M_{\mathcal{P}(A)}$ decomposes into a disjoint union of $\operatorname{RBS}(M)$'s. We show that the monoidal category $M_{\mathscr{E}}$ defines a model for the algebraic K-theory space $K(\mathscr{E})$ by comparing with Quillen's Q-construction $Q(\mathscr{E})$. We find that $B|M_{\mathscr{E}}|\simeq |Q(\mathscr{E})|$, so that in particular, $K(\mathscr{E})\simeq \Omega B|M_{\mathscr{E}}|$.

\medskip

In fact, we will work in slightly greater generality, namely with categories with filtrations as introduced below. This is a category with a distinguished class of short exact sequences satisfying a set of axioms enabling us to merge and split filtrations. Any exact category is a category with filtrations, but we do not need the full power of exact categories for our constructions, so we choose to work in this broader setting. It is also clear that Quillen's Q-construction can be defined verbatim for categories with filtrations. They have the advantage of including for example the category of vector spaces of dimension at most $n$.

\subsection{Categories with filtrations}

\begin{definition}
Let $\mathscr{C}$ be a category with a zero object $0$ and a distinguished class $C$ of triples $a\rightarrow b\rightarrow c$ called \textit{short exact sequences}. If a morphism appears as the first morphism in a short exact sequence, we call it an \textit{admissible monomorphism} and denote it by $\rightarrowtail$; if it appears as the second, we call it an \textit{admissible epimorphism} and denote it by $\twoheadrightarrow$. We say that $\mathscr{C}$ is a \textit{category with filtrations} (with respect to the collection $C$) if it satisfies the following axioms:
\begin{enumerate}
\item $C$ is closed under isomorphisms,
\item the sequences $0\rightarrow a \xrightarrow{=} a$ and $a\xrightarrow{=} a\rightarrow 0$ are short exact sequences for all objects $a$,
\item the composite of two admissible monomorphisms (epimorphisms) is itself an admissible monomorphism (epimorphism),
\item admissible monomorphisms are kernels of their corresponding admissible epimorphisms, and admissible epimorphisms are cokernels of their corresponding admissible monomorphisms,
\item the pullback of an admissible epimorphism along an admissible monomorphism is an admissible epimorphism,
\item the pushout of an admissible monomorphism along an admissible epimorphism is an admissible monomorphism.
\defend
\end{enumerate}
\end{definition}

Existence of pullbacks and pushouts in axioms (5) and (6) comes for free, so we do not need to assume this --- see \Cref{pullbacks and pushouts} below.

\begin{example}
Let $\mathscr{A}$ be an abelian category, and let $\mathscr{C}$ be a full subcategory containing $0$ which is closed under isomorphisms. Let $C$ be the class of sequences $A\rightarrow B\rightarrow C$ in $\mathscr{C}$ which are exact in $\mathscr{A}$. Suppose the classes of admissible monomorphisms respectively admissible epimorphisms are closed under composition. Then $\mathscr{C}$ is a category with filtrations.
\exend
\end{example}

In view of this we have the following list of examples.

\begin{example}
\ 
\begin{enumerate}
\item Exact categories.
\item Consider the abelian category $\operatorname{Vect}(k)$ of finite dimensional vector spaces over a field $k$. Fix $n\in \N$ and let $\operatorname{Vect}(k)_{\leq n}$ denote the subcategory spanned by the vector spaces of dimension less than or equal to $n$. This is a category with filtrations.
\item Similarly, if $R$ is a ring such that the rank of projective modules is well-defined, then the category $\mathcal{P}(R)_{\leq n}$ of projective $R$-modules of rank at most $n$ is a category with filtrations.
\exend
\end{enumerate}
\end{example}

The following proposition is an immediate consequence of the axioms. Note that the roles of monomorphisms and epimorphisms are swapped when comparing with axioms (5) and (6) of the definition.

\begin{proposition}\label{pullbacks and pushouts}
The pullback of an admissible monomorphism along an admissible epimorphism exists and is an admissible monomorphism. The pushout of an admissible epimorphism along an admissible monomorphism exists and is an admissible epimorphism. Moreover, in both cases the squares are bicartesian.
\end{proposition}

\begin{remark}
We implicitly assume that all categories with filtrations are essentially small, that is, equivalent to a small category with filtrations.
\exend
\end{remark}

Let $\mathscr{C}$ be a category with filtrations. We now introduce the formalities of filtrations, flags and associated gradeds needed for our constructions.

\begin{definition}
Let $I=\{i_0<\cdots <i_k\}$ be a finite linearly ordered set, let $m$ be an object in $\mathscr{C}$ and let $(a_i)_{i\in I}$ be an $I$-graded object in $\mathscr{C}$. An \textit{$I$-indexed filtration in $m$ with associated graded $(a_i)_{i\in I}$} is an equivalence class $[x^I, (\rho_i)_{i\in I}]$ of diagrams as below satisfying that $x_{i-1}\rightarrowtail x_i\stackrel{\rho_i}{\twoheadrightarrow }a_i$ is a short exact sequence for all $i\in I$, where $x_{i_0-1}:=0$.

\begin{center}
\begin{tikzpicture}
\matrix (m) [matrix of math nodes,row sep=2em,column sep=2em, text height=1.5ex, text depth=0.25ex]
  {
x_{i_0} & x_{i_1} & \cdots & x_{i_{k-1}} & x_{i_k}=m \\
a_{i_0} & a_{i_1} & \cdots & a_{i_{k-1}} & a_{i_k} \\
  };
  \path[>->] 
	(m-1-1) edge (m-1-2)
	(m-1-2) edge (m-1-3)
	(m-1-3) edge (m-1-4)
	(m-1-4) edge (m-1-5)
  ;
  \path[->>]
  	(m-1-1) edge node[right]{$\rho_{i_0}$} (m-2-1)
  	(m-1-2) edge node[right]{$\rho_{i_1}$} (m-2-2)
  	(m-1-4) edge node[right]{$\rho_{i_{k-1}}$} (m-2-4)
  	(m-1-5) edge node[right]{$\rho_{i_k}$} (m-2-5)
  ;
\end{tikzpicture}
\end{center}

Two such diagrams are equivalent, if there is a commutative diagram

\begin{center}
\begin{tikzcd}[row sep=1em,column sep=0.4em,text height=1.5ex, text depth=0.25ex]
& y_{i_0} \arrow[>->,rr] \arrow[->>,dd] & & y_{i_1} \arrow[>->,rr] \arrow[->>,dd]  & & \cdots \arrow[>->,rr] & & y_{i_{k-1}} \arrow[>->,rr] \arrow[->>,dd] & & m \arrow[->>,dd] \\
x_{i_0} \arrow[>->,rr, crossing over] \arrow[->>,dr] \arrow[->,ur, "\cong" near start] & & x_{i_1} \arrow[>->,rr, crossing over] \arrow[->>,dr] \arrow[->,ur, "\cong" near start] & & \cdots \arrow[>->,rr] & & x_{i_{k-1}} \arrow[>->,rr, crossing over] \arrow[->>,dr] \arrow[->,ur, "\cong" near start] & & m \arrow[->>,dr] \arrow[-,double equal sign distance,ur] &\\
& a_{i_0} & & a_{i_1} & & \cdots & & a_{i_{k-1}} & & a_{i_k} \\
\end{tikzcd}
\end{center}
In that case the isomorphisms $x_i\rightarrow y_i$ are necessarily unique, so a representing diagram is unique up to unique isomorphism.

\medskip

An $I$-indexed filtration in $m$ with associated graded $(a_i)_{i\in I}$ is called an \textit{($I$-indexed) flag (with associated graded)} if $a_i\neq 0$ for all $i\in I$. Equivalently, some (and thus any) representative is a sequence of non-invertible monomorphisms.
\defend
\end{definition}

We observe that any filtration has an underlying flag given by composing all invertible admissible monomorphisms with the succeeding morphism.

\medskip

The existence of pullbacks of admissible monomorphisms along admissible epimorphisms and the fact that these are themselves admissible monomorphisms enable us to merge filtrations as in the definition below. The universal property of pullbacks implies that this is well-defined, that is, independent of the choice of representatives of the filtrations.

\begin{definition}\label{merging of filtrations}
Let $\theta\colon I\rightarrow J$ be a surjective order preserving map. Suppose we are given a $J$-indexed filtration $[x^J, (\pi_j)_{j\in J}]$ in $m$ with associated graded $(b_j)_{j\in J}$ and for every $j\in J$, a $\theta^{-1}(j)$-indexed filtration
\begin{align*}
[y^j, (\rho_i)]:=[y^{\theta^{-1}(j)}, (\rho_i)_{i\in \theta^{-1}(j)}],
\end{align*}
in $b_j$ with associated graded $(a_i)_{i\in \theta^{-1}(j)}$.

\medskip

The \textit{merging of (the collection) $[y^j,(\rho_i)]_{j\in J}$ into $[x^J, (\pi_j)]$} is the $I$-indexed filtration in $m$ with associated graded $(a_i)_{i\in I}$ represented by a sequence $(\hat{y}^I, (\hat{\rho}_i))$ satisfying that for all $j\in J$, the restriction
\begin{align*}
(\hat{y}^{\theta^{-1}(j)}, (\hat{\rho}_i)_{i\in \theta^{-1}(j)})
\end{align*}
to $\theta^{-1}(j)=\{i_0<\cdots <i_k\}$ factors through $(y^j, (\rho_i))$ as indicated by the commutative diagram of pullbacks below
\begin{center}
\begin{tikzpicture}
\matrix (m) [matrix of math nodes,row sep=2em,column sep=2em, text height=1.5ex, text depth=0.25ex]
  {
\hat{y}_{i_0} & \hat{y}_{i_1} & \cdots & \hat{y}_{i_{k-1}} & \hat{y}_{i_k}=x_j \\ 
y_{i_0} & y_{i_1} & \cdots & y_{i_{k-1}} & b_j \\
a_{i_0} & a_{i_1} & \cdots & a_{i_{k-1}} & a_{i_k} \\
  };
  \path[>->] 
	(m-1-1) edge (m-1-2)
	(m-1-2) edge (m-1-3)
	(m-1-3) edge (m-1-4)
	(m-1-4) edge (m-1-5)
	(m-2-1) edge (m-2-2)
	(m-2-2) edge (m-2-3)
	(m-2-3) edge (m-2-4)
	(m-2-4) edge (m-2-5)
  ;
  \path[->>]
  	(m-1-1) edge (m-2-1)
  	(m-1-2) edge (m-2-2)
  	(m-1-4) edge (m-2-4)
  	(m-1-5) edge node[right]{$\pi_j$} (m-2-5)
  	(m-2-1) edge node[right]{$\rho_{i_0}$} (m-3-1)
  	(m-2-2) edge node[right]{$\rho_{i_1}$} (m-3-2)
  	(m-2-4) edge node[right]{$\rho_{i_{k-1}}$} (m-3-4)
  	(m-2-5) edge node[right]{$\rho_{i_k}$} (m-3-5)
  ;
  \path[-]
  (m-1-1) edge[style=transparent] node[opaque,pos=0.1]{\scalebox{1.5}{$\lrcorner$}} (m-2-2)
  (m-1-2) edge[style=transparent] node[opaque,pos=0.1]{\scalebox{1.5}{$\lrcorner$}} (m-2-3)
  (m-1-4) edge[style=transparent] node[opaque,pos=0.1]{\scalebox{1.5}{$\lrcorner$}} (m-2-5)
  ;
\end{tikzpicture}
\end{center}

where $\hat{\rho_i}$ is the composite of $\rho_i$ with the admissible epimorphism $\hat{y}_i\twoheadrightarrow y_i$.

\medskip

We denote the merging by $[x^J, (\pi_j)]\circ \Big([y^j,(\rho_i)]_{j\in J}\Big):=[\hat{y}^I,(\hat{\rho}_i)]$.
\defend
\end{definition}

\begin{remark}\label{merging and splitting of filtrations in vector spaces}
The existence and uniqueness of a merging together with the observation that we can split flags should be interpreted as a generalisation of the following statement for vector spaces: for a surjective order preserving map $\theta\colon I\rightarrow J$ of finite linearly ordered sets, a $J$-indexed filtration $\{V_j\}_{j\in J}$ of $V$ together with a $\theta^{-1}(j)$-indexed filtration of the cokernel $V_j/V_{j-1}$ for all $j$ is equivalent to an $I$-indexed filtration of $V$.
\exend
\end{remark}

\subsection{A monoidal category of flags and associated gradeds}\label{M_C}

We now define a monoidal category encoding the data of flags with associated gradeds in a given category with filtrations. Intuitively, the objects should be thought of as associated gradeds, and the morphisms as those induced by flags where we allow refinement of flags. For example, a morphism from $(a,b,c)$ to $(m)$ is a $3$-step filtration of $m$ with associated graded $(a,b,c)$.

\medskip

Let $\mathscr{C}$ be a category with filtrations. The monoidal category $M_{\mathscr{C}}$ is defined in Constructions \ref{objects and morphisms}, \ref{composition} and \ref{monoidal product in M_C} (see also \Cref{non-symmetric multicategories} for a different perspective in terms of multicategories).

\begin{construction}[Objects and morphisms]\label{objects and morphisms}
The objects of $M_{\mathscr{C}}$ are tuples $(I, (m_i)_{i\in I})$, where $I$ is a finite linearly ordered set and $(m_i)_{i\in I}$ is an $I$-graded object in $\mathscr{C}$ with $m_i\neq 0$ for all $i\in I$. We just write $(m_i)_{i\in I}$ and call such an object an \textit{$I$-indexed list}, and we include the empty list $\emptyset$. A morphism $\phi\colon (m_i)_{i\in I}\rightarrow (n_j)_{j\in J}$ consists of the following data
\begin{enumerate}
\item a surjective order preserving map $\theta\colon I\rightarrow J$,
\item for every $j\in J$, a $\theta^{-1}(j)$-indexed flag in $n_j$ with associated graded $(m_i)_{i\in \theta^{-1}(j)}$:
\begin{align*}
[x^j, (\rho_i)]:=[x^{\theta^{-1}(j)},(\rho_i)_{i\in \theta^{-1}(j)}]
\end{align*}
\end{enumerate}
We write $\phi=(\theta, [x^j, (\rho_i)]_{j\in J})\colon (m_i)_{i\in I}\rightarrow (n_j)_{j\in J}$.
\exend
\end{construction}

\begin{remark}
Diagramatically, one can picture a morphism $\phi$ as specified above as follows. Writing out the list of objects of the source in the top line, and the list of objects of the target in the bottom line, we connect the objects as specified by the order preserving map and label the target objects by the appropriate flags. Of course, this can be more or less detailed in order to emphasise the relevant data or structure.

\begin{center}
\begin{tikzpicture}
\matrix (m) [matrix of math nodes,row sep=2em,column sep=0.5em]
  {
(m_i)_{i\in \theta^{-1}(j_0)} & (m_i)_{i\in \theta^{-1}(j_1)} & \quad\cdots\quad  & (m_i)_{i\in \theta^{-1}(j_k)} \\
 n_{j_0} & n_{j_1} & \cdots & n_{j_k} \\
  };
  \path[-]
	($(m-1-1.south west)!.30!(m-1-1.south east)$) edge (m-2-1) 
	($(m-1-1.south west)!.50!(m-1-1.south east)$) edge (m-2-1)
	($(m-1-1.south west)!.70!(m-1-1.south east)$) edge (m-2-1)
	($(m-1-2.south west)!.30!(m-1-2.south east)$) edge (m-2-2) 
	($(m-1-2.south west)!.50!(m-1-2.south east)$) edge (m-2-2)
	($(m-1-2.south west)!.70!(m-1-2.south east)$) edge (m-2-2)
	($(m-1-3.south west)!.30!(m-1-3.south east)$) edge (m-2-3) 
	($(m-1-3.south west)!.50!(m-1-3.south east)$) edge (m-2-3)
	($(m-1-3.south west)!.70!(m-1-3.south east)$) edge (m-2-3)
	($(m-1-4.south west)!.30!(m-1-4.south east)$) edge (m-2-4)
	($(m-1-4.south west)!.50!(m-1-4.south east)$) edge (m-2-4) 
	($(m-1-4.south west)!.70!(m-1-4.south east)$) edge (m-2-4) 
;
\node
also [label=below: {$\scriptstyle [x^{j_0},\, (\rho_i)]$}] (m-2-1)
;
\node
also [label=below: {$\scriptstyle [x^{j_1},\, (\rho_i)]$}] (m-2-2)
;
\node
also [label=below: {$\scriptstyle [x^{j_k},\, (\rho_i)]$}] (m-2-4)
;
\end{tikzpicture}
\end{center}

We will use diagrams like this to picture composition and an important decomposition below, but other than that, we only include this remark hoping that it might help the reader to detach themselves a little from the technical aspects and notation.
\exend
\end{remark}

Before defining composition, we observe that the concatenation operation on the objects of $M_{\mathscr{C}}$ can be extended to the morphisms. This will also be used to define a monoidal product in $M_{\mathscr{C}}$ (see \Cref{monoidal product in M_C}).

\begin{construction}[Concatenation]\label{concatenation}
We denote the concatenation of linearly ordered sets $I=\{i_0<\cdots < i_k\}$ and $J=\{j_0<\cdots <j_l\}$ by
\begin{align*}
I\circledast J=\{i_0<\cdots i_k<j_0<\cdots <j_l\}.
\end{align*}

Recall that the concatenation of graded objects $(m_i)_{i\in I}$ and $(n_j)_{j\in J}$ is the $(I\circledast J)$-graded object
\begin{align*}
(m_i)_{i\in I}\circledast(n_j)_{j\in J}=((m\circledast n)_i)_{i\in I\circledast J}=(m_{i_0},\ldots,m_{i_k},n_{j_0},\ldots,n_{j_l}).
\end{align*}

For morphisms, we can likewise concatenate the data: the concatenation of 
\begin{align*}
(\theta,[x^j,(\rho_i)])\colon (m_i)_{i\in I}\rightarrow (k_j)_{j\in J},\quad\text{and}\quad (\sigma,[y^j,(\pi_i)])\colon (n_i)_{i\in I'}\rightarrow (l_j)_{j\in J'},
\end{align*}
is the morphism
\begin{align*}
(\theta,[x^j,(\rho_i)])\circledast (\sigma,[y^j,(\pi_i)])\colon ((m\circledast n)_i)_{i\in I\circledast I'}\rightarrow ((k\circledast l)_j)_{j\in J\circledast J'},
\end{align*}
given by
\begin{enumerate}
\item the surjective order preserving map $\theta\circledast \sigma\colon I\circledast I' \rightarrow J\circledast J'$ defined on $I$ respectively $I'$ by $\theta$ respectively $\sigma$.
\item the flag $[x^j,(\rho_i)]$ for $j\in J$, and the flag $[y^j,(\pi_i)]$ for $j\in J'$.
\end{enumerate}
We also write $(\theta,[x^j,(\rho_i)])\circledast (\sigma,[y^j,(\pi_i)])=(\theta\circledast \sigma,[x^j,(\rho_i)]\circledast[y^j,(\pi_i)])$.
\exend
\end{construction}

\begin{definition}
Let $I$ be a finite linearly ordered set. An \textit{interval} $I'\subseteq I$ is a subset satisfying that if $i<j<l$ and $i,l\in I'$, then $j\in I'$. A \textit{partition} of $I$ is a decomposition $I=\circledast_{t\in T}I_t$ for ordered intervals $I_t\subseteq I$, $t\in T$, where $T$ is some linearly ordered set.
\defend
\end{definition}

We observe that any morphism in $M_{\mathscr{C}}$ can be completely decomposed as the concatenation of morphisms to one object lists.

\begin{observation}\label{complete decomposition}
Let $(\theta,[x^j,(\rho_i)]_{j\in J})\colon (m_i)_{i\in I}\rightarrow (n_j)_{j\in J}$ be a morphism in $M_{\mathscr{C}}$. Then
\begin{align*}
(\theta,[x^j,(\rho_i)]_{j\in J})=\circledast_{j\in J} \,(\theta^j,[x^j,(\rho_i)]_{j\in \{j\}})=(\theta, \circledast_{j\in J} [x^j,(\rho_i)]_{j\in \{j\}}),
\end{align*}
where $(\theta^j,[x^j,\rho_i]_{j\in \{j\}})\colon (m_i)_{i\in \theta^{-1}(j)}\rightarrow (n_j)_{j\in \{j\} }$ is the morphism given by
\begin{enumerate}
\item the unique map $\theta^j=\theta\vert_{\theta^{-1}(j)}\colon \theta^{-1}(j)\rightarrow \{j\}$,
\item the flag $[x^j,(\rho_i)]$ in $n_j$.\exend
\end{enumerate}
\end{observation}

We now define composition in $M_{\mathscr{C}}$.

\begin{construction}[Composition]\label{composition}
Let
\begin{align*}
&(\theta,[y^j,(\rho_i)]_{j\in J})\colon (m_i)_{i\in I}\rightarrow (n_j)_{j\in J} \quad\text{and}\quad (\sigma,[x^l,(\pi_j)]_{l\in L})\colon (n_j)_{j\in J}\rightarrow (k_l)_{l\in L}
\end{align*}
be morphisms in $M_{\mathscr{C}}$. The composite is defined by merging the flags of the given morphisms for each $l\in L$:
\begin{align*}
(\sigma,[x^l,(\pi_j)]_{l\in L})&\circ (\theta,[y^j,(\rho_i)]_{j\in J}) \\
&:=\bigg(\sigma\circ \theta, \circledast_{l\in L} \, [x^l,(\pi_j)]\circ\Big([y^j,(\rho_i)]_{j\in \sigma^{-1}(l)}\Big) \bigg),
\end{align*}
where $[x^l,(\pi_j)]\circ\Big([y^j,(\rho_i)]_{j\in \sigma^{-1}(l)}\Big)$ is the $(\sigma\circ \theta)^{-1}(l)$-indexed flag in $k_l$ with associated graded $(a_i)_{i\in (\sigma\circ \theta)^{-1}(l)}$ as defined in \Cref{merging of filtrations}.
\exend
\end{construction}

\begin{remark}
For each $l\in L$, the composition can be pictured by the diagram below, where we have omitted the objects and just denote the finite linearly ordered sets and write $\sigma^{-1}(l)=\{j_0,\cdots, j_k\}$.

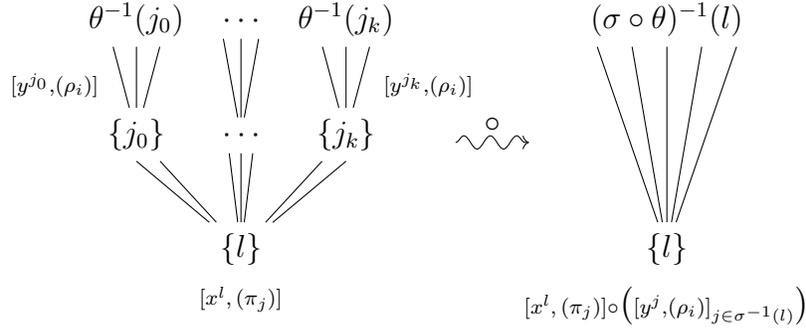
\begin{center}
\begin{tikzpicture}
\matrix (m) [matrix of math nodes,row sep=2em,column sep=0.5em]
  {
\theta^{-1}(j_0) & \cdots & \theta^{-1}(j_k) & \qquad\qquad & (\sigma\circ \theta)^{-1}(l) \\
\{j_0\} & \cdots & \{j_k\} & \ \quad\   & \\
& \{l\} & & & \{l\} \\
  };
  \path[-]
	($(m-1-1.south west)!.30!(m-1-1.south east)$) edge (m-2-1) 
	($(m-1-1.south west)!.50!(m-1-1.south east)$) edge (m-2-1)
	($(m-1-1.south west)!.70!(m-1-1.south east)$) edge (m-2-1)
	($(m-1-2.south west)!.20!(m-1-2.south east)$) edge (m-2-2) 
	($(m-1-2.south west)!.50!(m-1-2.south east)$) edge (m-2-2)
	($(m-1-2.south west)!.80!(m-1-2.south east)$) edge (m-2-2)
	($(m-1-3.south west)!.30!(m-1-3.south east)$) edge (m-2-3) 
	($(m-1-3.south west)!.50!(m-1-3.south east)$) edge (m-2-3)
	($(m-1-3.south west)!.70!(m-1-3.south east)$) edge (m-2-3)
	($(m-2-1.south west)!.50!(m-2-1.south east)$) edge (m-3-2)
	($(m-2-1.south west)!.80!(m-2-1.south east)$) edge (m-3-2) 
	($(m-2-2.south west)!.30!(m-2-2.south east)$) edge (m-3-2) 
	($(m-2-2.south west)!.50!(m-2-2.south east)$) edge (m-3-2)
	($(m-2-2.south west)!.70!(m-2-2.south east)$) edge (m-3-2)
	($(m-2-3.south west)!.20!(m-2-3.south east)$) edge (m-3-2) 
	($(m-2-3.south west)!.50!(m-2-3.south east)$) edge (m-3-2)
	($(m-1-5.south west)!.10!(m-1-5.south east)$) edge (m-3-5)
	($(m-1-5.south west)!.30!(m-1-5.south east)$) edge (m-3-5) 
	($(m-1-5.south west)!.50!(m-1-5.south east)$) edge (m-3-5)
	($(m-1-5.south west)!.70!(m-1-5.south east)$) edge (m-3-5)
	($(m-1-5.south west)!.90!(m-1-5.south east)$) edge (m-3-5)
;
\path[->]
	(m-2-4.west) edge[decorate, decoration={snake}] node[above]{$\circ$} (m-2-4.east)
;
\node
also [label=below: {$\scriptstyle [x^l,\, (\pi_j)]$}] (m-3-2)
;
\node
also [label=above left: {$\scriptstyle [y^{j_0}, (\rho_i)]$}] (m-2-1)
;
\node
also [label=above right: {$\scriptstyle [y^{j_k}, (\rho_i)]$}] (m-2-3)
;
\node
also [label=below: {$\scriptstyle [x^l,\,(\pi_j)]\circ\big([y^j,(\rho_i)]_{j\in \sigma^{-1}(l)}\big)$}] (m-3-5)
;
\end{tikzpicture}
\end{center}
\exend
\end{remark}

The concatenation operation defines a monoidal product in $M_{\mathscr{C}}$.

\begin{construction}[Monoidal product]\label{monoidal product in M_C}
Let
\begin{align*}
\circledast\colon M_{\mathscr{C}}\times M_{\mathscr{C}}\rightarrow M_{\mathscr{C}}
\end{align*}
be the functor given by:
\begin{align*}
(m_i)_{i\in I}\circledast (n_j)_{j\in J}=((m\circledast n)_i)_{i\in I\circledast J},
\end{align*}
and
\begin{align*}
(\theta,[x^j,(\rho_i)])\circledast (\sigma,[y^j,(\pi_i)])=(\theta\circledast \sigma,[x^j,(\rho_i)]\circledast [y^j,(\pi_i)]).
\end{align*}

One easily verifies that this defines a strict monoidal product with identity object the empty list $\emptyset$.
\exend
\end{construction}

\begin{remark}\label{non-symmetric multicategories}
The category $M_{\mathscr{C}}$ can also be interpreted as the strict monoidal category coming from a non-symmetric multicategory whose objects are those of $\mathscr{C}$ and where a morphism $(a_1,\ldots,a_n)\rightarrow (b)$ is a flag in $b$ with associated graded $(a_1,\ldots, a_n)$. See \cite[Definition 2.1.1 and \S 2.3]{Leinster04} or \cite[Definitions 3.1.6 and 3.1.7]{GepnerHaugseng15}.
\exend
\end{remark}

\begin{remark}
Let us relate this definition to the reductive Borel--Serre categories defined in \Cref{RBS(M)definition}. Let $A$ be an associative ring and let $\mathcal{P}(A)$ be the exact category of finitely generated projective $A$-modules. For $M\in \mathcal{P}(A)$, there is a fully faithful functor
\begin{align*}
F_M\colon \operatorname{RBS}(M)\rightarrow M_{\mathcal{P}(A)}
\end{align*}
given by sending a splittable flag to its associated graded. For a morphism $gU_{\mathcal{F}}\colon \mathcal{F}\rightarrow \mathcal{F}'$ in $\operatorname{RBS}(M)$, the corresponding morphism is most easily described as a composite
\begin{align*}
\operatorname{gr}(\mathcal{F})\xrightarrow{g} \operatorname{gr}(g\mathcal{F})\xrightarrow{\operatorname{incl}} \operatorname{gr}(\mathcal{F}'),
\end{align*}
where the first map is the map of associated gradeds induced by the isomorphism $g\colon \mathcal{F}\rightarrow g\mathcal{F}$ and the second is the one induced by the refinement $g\mathcal{F}\subset \mathcal{F}'$. More precisely, the refinement $g\mathcal{F}\subset \mathcal{F}'$ defines a morphism as follows, where we write
\begin{align*}
\mathcal{F}=(M_1 \subsetneq \cdots \subsetneq M_{d-1}), \quad\text{and}\quad \mathcal{F}'=(N_1 \subsetneq \cdots \subsetneq N_{e-1})
\end{align*}
and recall that there is an order-preserving injective map $f\colon \{1,\ldots, e-1\}\rightarrow \{1,\ldots, d-1\}$ such that $N_j=gM_{f(j)}$. The surjective order preserving map $\theta\colon \{1,\ldots, d\}\rightarrow \{1,\ldots, e\}$ is given by $i\mapsto  \operatorname{min}\{j\mid i\leq f(j)\}$. For $j\in \{1,\ldots, e\}$, we have $\theta^{-1}(j)=\{f(j-1)+1,\ldots, f(j)\}$ (setting $f(0)=0$) and we choose the flag in $N_{j}/N_{j-1}$ given by the image of the flag
\begin{align*}
[gM_{f(j-1)+1}\subset gM_{f(j-1)+2}\subset \cdots \subset gM_{f(j)}=N_{j}].
\end{align*}

For a set $\mathcal{M}$ of representatives of finitely generated projective $A$-modules, these functors provide an equivalence $M_{\mathcal{P}(A)}\simeq \coprod_{M\in \mathcal{M}} \operatorname{RBS}(M)$.
\exend
\end{remark}

We make the following useful observations.

\begin{proposition}\label{all morphism in M_C are monomorphisms}
All morphisms in $M_{\mathscr{C}}$ are monomorphisms.
\end{proposition}
\begin{proof}
It suffices to show that morphisms to one object lists are monomorphisms. This amounts to showing that for given any flag $[x^J,(\rho_j)]$ in $m$ with associated graded $(b_j)_{j\in J}$ and any given surjective order preserving maps $\theta,\sigma\colon I\rightarrow J$, and flags $[y^{\theta^{-1}(j)}, (\mu_i)]$ and $[z^{\sigma^{-1}(j)}, (\nu_i)]$ in $b_j$ for all $j\in J$, we have
\begin{align*}
[x^J, (\rho_j)]\circ \Big([y^{\theta^{-1}(j)}, (\mu_i)]_{j\in J}\Big) =[x^J, (\rho_j)]\circ \Big([z^{\sigma^{-1}(j)}, (\nu_i)]_{j\in J}\Big),
\end{align*}
if and only if $\theta=\sigma$ and $[y^{\theta^{-1}(j)}, (\mu_i)] =[z^{\sigma^{-1}(j)}, (\nu_i)]$ for all $j\in J$. The equality $\theta=\sigma$ follows directly from the fact that flags are defined by sequences of non-invertible admissible monomorphisms, and the equality of flags is then verified by the universal property of pullbacks.
\end{proof}

The following proposition gives a decomposition of morphisms which will be crucial to our proofs later on. We picture the decomposition diagrammatically and provide some intuition below.

\begin{proposition}\label{decomposition technical statement}
Let $\phi\colon (a_i)_{i\in I_1}\circledast (m_i)_{i\in I_2}\circledast (b_i)_{i\in I_3}\rightarrow (n_j)_{j\in J}$ be a morphism in $M_{\mathscr{C}}$ given by an order preserving map $\theta\colon I_1\circledast I_2\circledast I_3\rightarrow J$. Let $J=J_1\circledast J_2\circledast J_3$ be the partition given by $J_1=\theta(I_1)-\theta(I_2\cup I_3)$ and $J_3=\theta(I_3)-\theta(I_2\cup I_1)$. Then $\phi$ can be written on the form
\begin{align*}
\phi=\phi_A\circledast (f\circ (\phi_a\circledast \operatorname{id}\circledast \phi_b))\circledast \phi_B,
\end{align*}
for a morphism 
\begin{align*}
f\colon a\circledast (m_i)_{i\in I_2}\circledast b\rightarrow (n_j)_{j\in J_2} 
\end{align*}
with $a$ and $b$ one object lists or the empty list, and morphisms
\begin{align*}
&\phi_A\colon (a_i)_{i\in I_A}\rightarrow (n_j)_{j\in J_1},\quad\text{and}\quad \phi_a\colon (a_i)_{i\in I_a}\rightarrow a,\\
&\phi_B\colon (b_i)_{i\in I_B}\rightarrow (n_j)_{j\in J_3},\quad\text{and}\quad \phi_b\colon (b_i)_{i\in I_b}\rightarrow b,
\end{align*}
where $I_A\subset I_1$ is the preimage of $J_1$ and $I_a$ is its complement, and $I_B\subset I_3$ is the preimage of $J_3$ and $I_b$ its complement.

\medskip

Moreover, if $\phi=\phi'_A\circledast (f'\circ (\phi_{a'}\circledast \operatorname{id}\circledast \phi_{b'}))\circledast \phi'_B$ is another such decomposition, then $\phi_A=\phi_A'$, $\phi_B=\phi_B'$ and
\begin{align*}
\phi_a=\alpha\circ\phi_{a'},\quad \phi_b=\beta\circ \phi_{b'}, \quad \text{and}\quad f'=f\circ (\alpha \circledast \operatorname{id}_m\circledast \beta)
\end{align*}
for unique isomorphisms $\alpha\colon a'\rightarrow a$, $\beta\colon b'\rightarrow b$.
\end{proposition}

We do not provide a full proof of this, as it is a straightforward albeit technical observation. Instead we provide an example below in the case when $\mathscr{C}=\operatorname{Vect}(k)$ is the category of finite dimensional vector spaces over some field $k$, as this illustrates the intuition behind the decomposition better. The reader can readily verify that this generalises directly. The decomposition relies crucially on the fact that filtrations can be merged and split, mirroring the way filtrations of vector spaces behave (see \Cref{merging and splitting of filtrations in vector spaces}).  The idea is to ``collapse'' the outer tuples to one-object (or empty) lists so that we find a ``terminal decomposition'' of the morphism --- this will be vital to our arguments later on (see \Cref{terminal decomposition}). The point is to identify the ``inner'' part of the morphism emitting from the interval $I_2$, where it may be necessary to add an object on either side to actually get a morphism. Before explaining the example, we note that the decomposition can be illustrated by the diagram below (where we have replaced the objects by the finite linearly ordered sets for notational simplicity). The dashed arrows illustrate that the decompositions differ by (unique) isomorphisms of the one object lists that are interpolated.

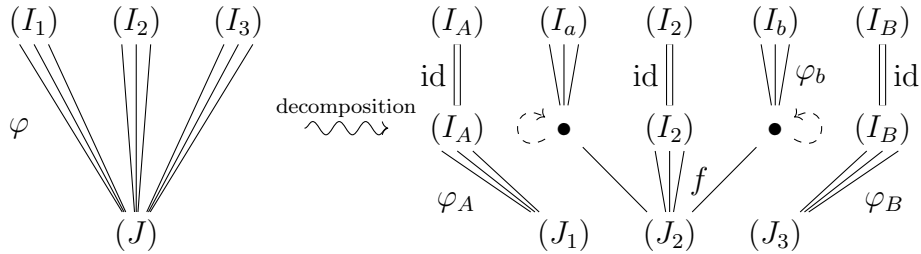
\begin{center}
\begin{tikzpicture}
\matrix (m) [matrix of math nodes,row sep=2em,column sep=1em, text height=1.5ex, text depth=0.25ex]
  {
(I_1) & (I_2) & (I_3) & \quad\quad & (I_A) & (I_a) & (I_2) & (I_b) & (I_B)\\
  && &  \quad\quad & (I_A) & \bullet & (I_2) & \bullet & (I_B)  \\
  & (J) &&& &(J_1) & (J_2) & (J_3) & \\
  };
  \path[-]
  ($(m-1-1.south west)!.30!(m-1-1.south east)$) edge node[left]{$\phi\quad$} (m-3-2) 
  ($(m-1-1.south west)!.50!(m-1-1.south east)$) edge (m-3-2)
  ($(m-1-1.south west)!.70!(m-1-1.south east)$) edge (m-3-2)
  ($(m-1-2.south west)!.30!(m-1-2.south east)$) edge (m-3-2) 
  ($(m-1-2.south west)!.50!(m-1-2.south east)$) edge (m-3-2)
  ($(m-1-2.south west)!.70!(m-1-2.south east)$) edge (m-3-2)
  ($(m-1-3.south west)!.30!(m-1-3.south east)$) edge (m-3-2) 
  ($(m-1-3.south west)!.50!(m-1-3.south east)$) edge (m-3-2)
  ($(m-1-3.south west)!.70!(m-1-3.south east)$) edge (m-3-2)
  ($(m-1-6.south west)!.30!(m-1-6.south east)$) edge (m-2-6) 
  ($(m-1-6.south west)!.50!(m-1-6.south east)$) edge (m-2-6)
  ($(m-1-6.south west)!.70!(m-1-6.south east)$) edge (m-2-6)
  ($(m-1-8.south west)!.30!(m-1-8.south east)$) edge (m-2-8) 
  ($(m-1-8.south west)!.50!(m-1-8.south east)$) edge (m-2-8)
  ($(m-1-8.south west)!.70!(m-1-8.south east)$) edge  node[right]{$\phi_b$} (m-2-8)
  ($(m-2-7.south west)!.30!(m-2-7.south east)$) edge (m-3-7) 
  ($(m-2-7.south west)!.50!(m-2-7.south east)$) edge (m-3-7)
  ($(m-2-7.south west)!.70!(m-2-7.south east)$) edge node[right]{$f$} (m-3-7)
  ($(m-2-5.south west)!.30!(m-2-5.south east)$) edge node[below left]{$\phi_A$} (m-3-6) 
  ($(m-2-5.south west)!.50!(m-2-5.south east)$) edge (m-3-6)
  ($(m-2-5.south west)!.70!(m-2-5.south east)$) edge (m-3-6)
  ($(m-2-9.south west)!.30!(m-2-9.south east)$) edge (m-3-8) 
  ($(m-2-9.south west)!.50!(m-2-9.south east)$) edge (m-3-8)
  ($(m-2-9.south west)!.70!(m-2-9.south east)$) edge node[below right]{$\phi_B$} (m-3-8)
(m-1-5) edge[double equal sign distance] node[left]{$\operatorname{id}$} (m-2-5)
(m-1-7) edge[double equal sign distance] node[left]{$\operatorname{id}$} (m-2-7)
(m-1-9) edge[double equal sign distance] node[right]{$\operatorname{id}$} (m-2-9)
(m-2-6) edge (m-3-7)
(m-2-8) edge (m-3-7)
  ;
  \path[->,dashed]
(m-2-8) edge[out= -30,in= 30,looseness=5] (m-2-8)
(m-2-6) edge[out= 210,in= 150,looseness=5] (m-2-6)
  ;
  \path[->]
	(m-2-4.west) edge[decorate, decoration={snake}] node[above]{$\scriptstyle \text{decomposition}$} (m-2-4.east)
;
\end{tikzpicture}
\end{center}

\begin{example}
Let $k$ be a field and $\mathscr{C}=\operatorname{Vect}(k)$, and consider a morphism
\begin{align*}
\phi\colon (A_0,A_1,A_2,M_0,M_1,B_0,B_1,B_2,B_3) \longrightarrow (N_0,N_1,N_2)
\end{align*}
in $M_{\mathscr{C}}$, which is given by the surjective order preserving map $[8]\cong [2]\circledast [1]\circledast [3]\rightarrow [2]$ which partitions $[8]$ into $\{0<1<2<3\}$, $\{4<5<6\}$ and $\{7<8\}$. Then $\phi$ is additionally given by three flags with associated gradeds:
\begin{itemize}[label=$\ast$]
\item a flag $F_0\subset F_1\subset F_2\subset F_3=N_0$ together with an identification
\begin{align*}
F_0\oplus F_1/F_0\oplus F_2/F_1\oplus F_3/F_2\cong A_0\oplus A_1\oplus A_2 \oplus M_0,
\end{align*}
respecting the grading;
\item a flag $E_0\subset E_1\subset E_2=N_1$ together with an identification
\begin{align*}
E_0\oplus E_1/E_0\oplus E_2/E_1\cong M_1\oplus B_0\oplus B_1,
\end{align*}
respecting the grading;
\item a flag $V_0\subset V_1=N_2$ together with an identification $V_0\oplus V_1/V_0\cong B_2\oplus B_3$, respecting the grading.
\end{itemize}

We illustrate this morphism by the following diagram.

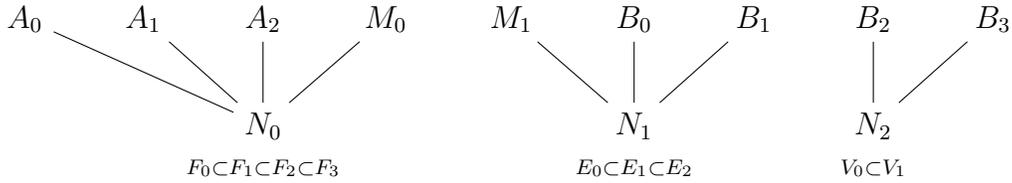
\begin{center}
\begin{tikzpicture}
\matrix (m) [matrix of math nodes,row sep=2em,column sep=2em, text height=1.5ex, text depth=0.25ex]
  {
 A_0 & A_1 & A_2 & M_0 & M_1 & B_0 & B_1 & B_2 & B_3 \\
   & & N_0 & & & N_1 & & N_2 & \\
  };
\path[-]
(m-1-1) edge (m-2-3)
(m-1-2) edge (m-2-3)
(m-1-3) edge (m-2-3)
(m-1-4) edge (m-2-3)
(m-1-5) edge (m-2-6)
(m-1-6) edge (m-2-6)
(m-1-7) edge (m-2-6)
(m-1-8) edge (m-2-8)
(m-1-9) edge (m-2-8)
;
\node
also [label=below:$\scriptstyle F_0\subset F_1\subset F_2\subset F_3$] (m-2-3)
;
\node
also [label=below:$\scriptstyle E_0\subset E_1\subset E_2$] (m-2-6)
;
\node
also [label=below:$\scriptstyle V_0\subset V_1$] (m-2-8)
;
\end{tikzpicture}
\end{center}

We obtain the desired decomposition by replacing the subset $(A_0,A_1,A_2)$ of the source tuple by $(F_2)$ and the subset $(B_0,B_1)$ by $(E_2/E_0)$ and splitting the flags accordingly. This is illustrated by the following diagram. The identification of the associated gradeds are the obvious ones inherited from the data defining $\phi$.

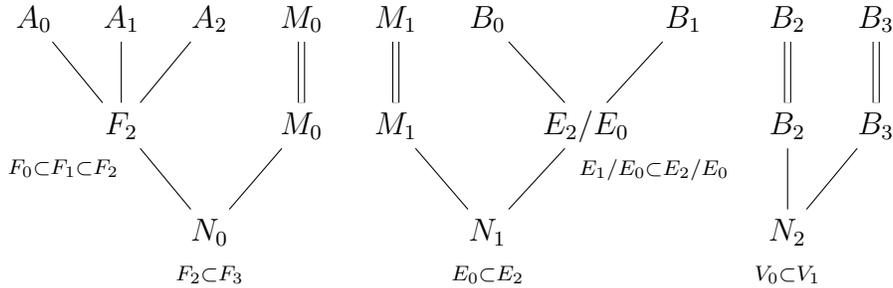
\begin{center}
\begin{tikzpicture}
\matrix (m) [matrix of math nodes,row sep=2em,column sep=0.5em, text height=1.5ex, text depth=0.25ex]
  {
 A_0 & & A_1 & & A_2 & & M_0 & & M_1 & & B_0 & & B_1 & & & B_2 & & B_3\\
   & & F_2 & & & & M_0 & & M_1 & & & E_2/E_0 & & & & B_2  & & B_3 \\
   & & & & N_0 & & & & & & N_1 & & & & & N_2 & & \\
  };
\path[-]
(m-1-1) edge (m-2-3)
(m-1-3) edge (m-2-3)
(m-1-5) edge (m-2-3)
(m-1-7) edge[double equal sign distance] (m-2-7)
(m-1-9) edge[double equal sign distance] (m-2-9)
(m-1-11) edge (m-2-12)
(m-1-13) edge (m-2-12)
(m-1-16) edge[double equal sign distance] (m-2-16)
(m-1-18) edge[double equal sign distance] (m-2-18)
(m-2-3) edge (m-3-5)
(m-2-7) edge (m-3-5)
(m-2-9) edge (m-3-11)
(m-2-12) edge (m-3-11)
(m-2-16) edge (m-3-16)
(m-2-18) edge (m-3-16)
;
\node
also [label=below left:$\scriptstyle F_0\subset F_1\subset F_2\hspace{-1em}$] (m-2-3)
;
\node
also [label=below right:$\hspace{-1.3em}\scriptstyle E_1/E_0\subset E_2/E_0$] (m-2-12)
;
\node
also [label=below:$\scriptstyle F_2\subset F_3$] (m-3-5)
;
\node
also [label=below:$\scriptstyle E_0\subset E_2$] (m-3-11)
;
\node
also [label=below:$\scriptstyle V_0\subset V_1$] (m-3-16)
;
\end{tikzpicture}
\end{center}

In the general case, where we do not have canonical choices of subspaces and quotients, one needs to make a choice of subobject $F_2$ and a choice of quotient object $E_2/E_0$. These are unique up to unique isomorphism, so decompositions given by different choices differ by unique isomorphisms of these one object lists.

\medskip

With the notation of the proposition, we have
\begin{itemize}[label=$\ast$]
\item $J=[2]=\emptyset\circledast \{0<1\}\circledast \{2\}=J_1\circledast J_2\circledast J_3$,
\item $\phi_A\colon \emptyset\rightarrow \emptyset$ is the unique morphism, and $\phi_B\colon (B_2,B_3)\rightarrow (N_2)$ is given by the flag $V_0\subset V_1=N_2$ and the identification of the associated graded coming from of $\phi$,
\item $\phi_a\colon (A_0,A_1,A_2)\rightarrow (F_2)$  is given by the flag $F_0\subset F_1\subset F_2$ and the identification of the associated graded coming from $\phi$,
\item $\phi_b\colon (B_0,B_1)\rightarrow (E_2/E_0)$ is given by the flag $E_1/E_0 \subset E_2/E_0$ and the identification of the associated graded coming from $\phi$,
\item $f\colon (F_2,M_0,M_1,E_2/E_0)\rightarrow (N_0,N_1)$ is the concatenation of two morphisms
\begin{align*}
f_1\colon (F_2,M_0)\rightarrow (N_0)\quad\text{and}\quad f_2\colon (M_1, E_2/E_0)\rightarrow (N_1)
\end{align*}
given by the flags $F_2\subset F_3=N_0$ respectively $E_0\subset E_2=N_1$ where the identifications of the associated gradeds are given by the identities on $F_2$ respectively $E_2/E_0$ and by the isomorphisms $F_3/F_2\cong M_0$ respectively $E_0\cong M_1$ given by $\phi$.
\end{itemize}
\exend
\end{example}

\subsection{A 2-categorical Q-construction}\label{Q-constructions}

In this section, we associate a $2$-category $Q_2(M)$ to any strict monoidal category $M$. See \Cref{appendix nerves and geometric realisations} for the basic notions of $2$-categories needed here.

\medskip

Let $M$ be a strict monoidal category. We denote the monoidal product by juxtaposition. The following construction is simply the action $2$-category $Q(M\times M^{\otimes^{\operatorname{op}}}, M)$ of \Cref{M acting on itself}, but we remind ourselves of the details for clarity.

\begin{construction}
We define a $2$-category $Q_2(M)$ as follows:
\begin{itemize}[label=$\ast$]
\item the objects are those of $M$, 
\item a morphism $m\rightarrow m'$ is a tuple $(a,b,\phi\colon amb\rightarrow m')$ where $a$ and $b$ are objects of $M$ and $\phi$ is a morphism in $M$,
\item  and a $2$-cell $(a,b,\phi)\rightarrow (a',b',\phi')$ is a pair of morphisms $\alpha\colon a\rightarrow a'$, $\beta\colon b\rightarrow b'$ in $M$ such that $\phi=\phi'\circ (\alpha \,\operatorname{id}_m\,\beta)$,
\end{itemize}

and the various composites are as follows:
\begin{itemize}[label=$\ast$]
\item for morphisms:
\begin{center}
\begin{tikzpicture}
\matrix (m) [matrix of math nodes,row sep=2em,column sep=3em,, text height=1.5ex, text depth=0.25ex]
  {
	m & \\   
	m' & m'' \\   
  };
  \path[-stealth] 
	(m-1-1) edge node[left]{$(c,d,\psi)$} (m-2-1) edge node[above right]{$(ac, db, \phi\circ (\operatorname{id}_a\,\psi\,\operatorname{id}_b))$} (m-2-2) 
	(m-2-1) edge node[below]{$(a,b,\phi)$} (m-2-2) 
  ;
\end{tikzpicture}
\end{center}

\item composition of $2$-cells along morphisms (i.e.~within hom-categories) is given by coordinatewise composition:
\begin{center}
\begin{tikzpicture}
\matrix (m) [matrix of math nodes,row sep=2em,column sep=3em, text height=1.5ex, text depth=0.25ex]
  {
	(a,b,\phi) & \\   
	(a',b',\phi') & (a'',b'',\phi'') \\   
  };
  \path[-stealth] 
	(m-1-1) edge node[left]{$(\alpha,\beta)$} (m-2-1) edge node[above right]{$(\alpha'\circ \alpha,\beta'\circ\beta)$} (m-2-2) 
	(m-2-1) edge node[below]{$(\alpha',\beta')$} (m-2-2) 
  ;
\end{tikzpicture}
\end{center}

\item composition of $2$-cells along objects is given by the monoidal product:  the composite of the following $2$-cells
\begin{center}
\begin{tikzcd}[row sep=2em,column sep=5em]
 m \arrow[bend left=35]{r}[name=u,below]{}[above]{(c,d,\psi)}
    \arrow[bend right=35]{r}[name=d]{}[below]{(c', d',\psi')}
    \arrow[Rightarrow,to path=(u) -- (d)\tikztonodes]{r}{\,(\gamma,\delta)}
    & 
    m' \arrow[bend left=35]{r}[name=u,below]{}[above]{(a,b,\phi)}
    \arrow[bend right=35]{r}[name=d]{}[below]{(a',b',\phi')}
    \arrow[Rightarrow,to path=(u) -- (d)\tikztonodes]{r}{\,(\alpha,\beta)}
    & 
    m''
\end{tikzcd}
\end{center}
is the $2$-cell
\begin{align*}
(\alpha\gamma,\delta\beta)\colon (ac, db, \phi\circ (\operatorname{id}_a\,\psi\,\operatorname{id}_b))\longrightarrow (a'c', d'b', \phi'\circ (\operatorname{id}_{a'}\,\psi'\,\operatorname{id}_{b'}))
\end{align*}
\end{itemize}

We also define a $1$-category $Q_1(M)$ by taking components of the hom-categories in $Q_2(M)$. More precisely, the objects of $Q_1(M)$ are those of $Q_2(M)$ and the hom-sets are given by taking equivalence classes of morphisms in $Q_2(M)$, where two morphisms are equivalent, if there is a zig-zag of $2$-cells between them.

\medskip

Interpreting $Q_1(M)$ as a $2$-category with only identity $2$-cells, there is a canonical (strict) pseudofunctor $\kappa_M\colon Q_2(M)\rightarrow Q_1(M)$.
\exend
\end{construction}

\begin{observation}
A strong monoidal functor $M\rightarrow N$ induces a pseudofunctor between the associated $2$-categories $Q_2(M)\rightarrow Q_2(N)$, and a monoidal natural transformation of strong monoidal functors induces an oplax natural transformation of pseudofunctors. In particular, if $M$ is essentially small, then $Q_2(M)$ and $Q_1(M)$ admit geometric realisations.
\exend
\end{observation}

\begin{remark}
This construction is related to Quillen's $S^{-1}S$-construction, an intermediary construction used to prove the ``$Q=+$'' Theorem (\cite{Grayson}). For a monoidal category $S$, if the hom-categories in $Q_2(S)$ are groupoids, then $Q_1(S)$ is the category $\langle S\times S, S\rangle$ as defined in \cite{Grayson}. In comparison, Quillen's $S^{-1}S$-construction is the category $S^{-1}S=\langle S, S\times S\rangle$. (see also Remark \ref{Q(M,X) vs Quillen's SS}).
\exend
\end{remark}

We observe that $Q_2(M)$ contracts onto $Q_1(M)$ in the following situation.

\begin{proposition}\label{contracts to 1-category}
Let $M$ be a strict monoidal category. If for all pairs of objects $m, m'$ in $M$, the hom-category $\operatorname{Hom}_2(m,m')$ is a disjoint union of categories with terminal objects, then the pseudofunctor $\kappa_M\colon Q_2(M)\rightarrow Q_1(M)$ admits a right $2$-adjoint. In particular, if $M$ is essentially small, then $\kappa_M$ induces a homotopy equivalence of geometric realisations.
\end{proposition}
\begin{proof}
The adjoint is given by fixing a choice of terminal object in each component of all hom-categories in $Q_2(M)$. The unique $2$-cells from morphisms to the terminal objects define a lax unit transformation.
\end{proof}

Let $\mathscr{C}$ be a category with filtrations, let $M_{\mathscr{C}}$ be the monoidal category of flags and associated gradeds as defined in \Cref{M_C}, and consider the $2$-categorical Q-construction $Q_2(M_{\mathscr{C}})$. Since we implicitly assume that $\mathscr{C}$ is essentially small, so is $M_{\mathscr{C}}$, and thus the Q-constructions admit geometric realisations. The following proposition exploits the decomposition of morphisms in $M_{\mathscr{C}}$ given in \Cref{decomposition technical statement}.

\begin{proposition}\label{terminal decomposition}
For any morphism
\begin{align*}
((a_i)_{i\in I_a}, (b_i)_{i\in I_b}, \phi)\colon (m_i)_{i\in I}\rightarrow (n_j)_{j\in J},
\end{align*}
in $Q_2(M_{\mathscr{C}})$ there is a unique $2$-cell to a morphism of the form
\begin{align*}
\bigg((n_j)_{j\in J_1},\ (n_j)_{j\in J_3},\ \operatorname{id}_{(n_j)_{j\in J}}\bigg)\ \circ\  (a,b,f),
\end{align*}
for some partition $J=J_1\circledast J_2\circledast J_3$ and some $(a,b,f)\colon (m_i)_{i\in I}\rightarrow (n_j)_{j\in J_2}$ in $Q_2(M_{\mathscr{C}})$ with $a$ and $b$ one object lists or the empty list. Moreover, this morphism is unique up to a change of $(a,b,f)$, and two such representatives given by $(a,b,f)$ respectively $(a',b',f')$ will be connected by a unique (necessarily invertible) $2$-cell $(\alpha,\beta)\colon (a,b,f)\rightarrow (a',b',f')$.
\end{proposition}
\begin{proof}
In view of \Cref{decomposition technical statement}, any morphism
\begin{align*}
\phi\colon (a_i)_{i\in I_a}\circledast (m_i)_{i\in I}\circledast (b_i)_{i\in I_b}\rightarrow (n_j)_{j\in J}
\end{align*}
in $M_{\mathscr{C}}$ can be written on the form
\begin{align*}
\phi=\phi_A\circledast (f\circ (\phi_a\circledast \operatorname{id}\circledast \phi_b))\circledast \phi_B.
\end{align*}

It follows that we have a $2$-cell in $Q_2(M_{\mathscr{C}})$
\begin{center}
\begin{tikzpicture}
\matrix (m) [matrix of math nodes,row sep=3em,column sep=1em, text height=1.5ex, text depth=0.25ex]
  {
  \bigg((a_i)_{i\in I_A\circledast I_a},\ (b_i)_{i\in I_b\circledast I_B},\ \phi_A\circledast (f\circ (\phi_a\circledast \operatorname{id}_{(m_i)_{i\in I}}\circledast \phi_b))\circledast \phi_B\bigg) \\
\bigg((n_j)_{j\in J_1}\circledast a,\ b\circledast(n_j)_{j\in J_3},\ \operatorname{id}_{(n_j)_{j\in J_1}} \circledast f\circledast \operatorname{id}_{(n_j)_{j\in J_3}}\bigg) \\
  };
  \path[-stealth] 
(m-1-1) edge node[right]{$(\phi_A\circledast \phi_a,\phi_b\circledast\phi_B)$} (m-2-1)
  ;
\end{tikzpicture}
\end{center}
which is unique by the uniqueness observation of \Cref{decomposition technical statement} and the fact that $f$ is a monomorphism (\Cref{all morphism in M_C are monomorphisms}). The final statement also follows directly from the uniqueness of the decomposition.
\end{proof}

As a direct consequence of this, we can apply \Cref{contracts to 1-category}.

\begin{corollary}\label{Q-construction contract to 1-categories}
The hom-categories of $Q_2(M_{\mathscr{C}})$ are disjoint unions of categories with terminal objects. In particular, the pseudofunctor $\kappa_{M_{\mathscr{C}}}\colon Q_2(M_{\mathscr{C}})\rightarrow Q_1(M_{\mathscr{C}})$ admits a right $2$-adjoint, and thus induces a homotopy equivalence of geometric realisations.
\end{corollary}

\begin{remark}
The decomposition in the proof of \Cref{terminal decomposition} will be referred to as the \textit{terminal decomposition} of a morphism.
\exend
\end{remark}

\subsection{Comparing with Quillen's Q-construction}

Let $\mathscr{C}$ be a category with filtrations and let $M_{\mathscr{C}}$ be the monoidal category of flags and associated gradeds defined in \Cref{M_C}. We want to compare the classifying space $B|M_{\mathscr{C}}|$ with Quillen's Q-construction $Q(\mathscr{C})$. To do this, we first of all compare $Q_1(M_{\mathscr{C}})$ with $Q(\mathscr{C})$ and then combine this with the results of the previous sections.

\medskip

Recall that Quillen's Q-construction $Q(\mathscr{C})$ is the category with objects those of $\mathscr{C}$, and where a morphism $x\rightarrow y$ is given by an isomorphism class of diagrams of the form $x\twoheadleftarrow z \rightarrowtail y$, where two such diagrams are isomorphic if there is an isomorphism between the middle objects which commutes with the morphisms to $x$ and $y$. Composition is given by pullbacks, that is, the composite of $[x_1\twoheadleftarrow z_1\rightarrowtail x_2]$ and $[x_2\twoheadleftarrow z_2\rightarrowtail x_3]$ is given by the sequence $x_1\twoheadleftarrow z_1\times_{x_2} z_2\rightarrowtail x_3$.

\medskip

We define a functor $\Psi\colon Q(\mathscr{C})\rightarrow Q_1(M_{\mathscr{C}})$. On objects, it is given by
\begin{align*}
\Psi(x)=(x),\quad\text{for }x\neq 0,\quad \text{and}\quad \Psi(0)=\emptyset.
\end{align*}

Defining $\Psi$ on morphisms requires a little more work. For a fixed representative $x\twoheadleftarrow z\rightarrowtail y$ of a morphism in $Q(\mathscr{C})$, fix additionally an admissible monomorphism $a\rightarrowtail z$ corresponding to the admissible epimorphism $z\twoheadrightarrow x$ and an admissible epimorphism $y\twoheadrightarrow b$ corresponding to the admissible monomorphism $z\rightarrowtail y$. Consider the morphism $\phi\colon \Psi(a)\circledast \Psi(x)\circledast \Psi(b)\rightarrow \Psi(y)$ in $M_{\mathscr{C}}$ given by the underlying flag of the filtration with associated graded represented by the diagram
\begin{center}
\begin{tikzpicture}
\matrix (m) [matrix of math nodes,row sep=2em,column sep=2em, text height=1.5ex, text depth=0.25ex]
  {
 	a & z & y \\
 	a & x & b \\
  };
  \path[>->] 
	(m-1-1) edge (m-1-2)
	(m-1-2) edge (m-1-3)
  ;
  \path[->>]
	(m-1-2) edge (m-2-2)
	(m-1-3) edge (m-2-3)
  ;
  \path[-]
	(m-1-1) edge[double equal sign distance] (m-2-1)
  ;
\end{tikzpicture}
\end{center}

The following lemma is easily verified by tracing through the definitions.

\begin{lemma}
The morphism $[\Psi(a),\Psi(b),\phi]$ in $Q_1(M_{\mathscr{C}})$ defined above is independent of the choice of representative of the morphism $[x\twoheadleftarrow z\rightarrowtail y]$ and of the choice of $a\rightarrowtail z$ and $y\twoheadrightarrow b$. Moreover, the morphism $(\Psi(a),\Psi(b),\phi)$ in $Q_2(M_{\mathscr{C}})$ is a terminal representative of this morphism.
\end{lemma}

In view of this, we set
\begin{align*}
\Psi([x\twoheadleftarrow z\rightarrowtail y])=[\Psi(a),\Psi(b),\phi].
\end{align*}

One can check that this preserves composition and is associative, so that we have indeed defined a functor
\begin{align*}
\Psi\colon Q(\mathscr{C})\rightarrow Q_1(M_{\mathscr{C}}).
\end{align*}

\begin{remark}
To see that it preserves composition, one needs to identify the composite, i.e.~the $5$-step filtration given by the diagram of pullbacks below, and then determine a terminal representative of the resulting morphism by identifying the terminal decomposition of it as in the proof of \Cref{terminal decomposition}. Doing this, we find that it is represented by the subfiltration with associated graded given by picking out the sequence $a'\rightarrowtail z_3 \rightarrowtail x_3$ below, which is easily seen to represent the image of the composite in $Q(\mathscr{C})$.
\begin{center}
\begin{tikzpicture}
\matrix (m) [matrix of math nodes,row sep=2em,column sep=2em, text height=1.5ex, text depth=0.25ex]
  {
 	c & a' & z_3 & z_2 & x_3 \\
 	c & a & z_1 & x_2 & d \\
 	 & a & x_1 & b & \\
  };
  \path[>->] 
	(m-1-1) edge (m-1-2)
	(m-1-2) edge (m-1-3)
	(m-1-3) edge (m-1-4)
	(m-1-4) edge (m-1-5)
	(m-1-2) edge (m-1-3)
	(m-1-3) edge (m-1-4)
	(m-2-2) edge (m-2-3)
	(m-2-3) edge (m-2-4)
  ;
  \path[->>]
	(m-1-2) edge (m-2-2)
	(m-1-3) edge (m-2-3)
	(m-1-4) edge (m-2-4)
	(m-1-5) edge (m-2-5)
	(m-2-3) edge (m-3-3)
	(m-2-4) edge (m-3-4)
  ;
  \path[-]
	(m-1-1) edge[double equal sign distance] (m-2-1)
	(m-2-2) edge[double equal sign distance] (m-3-2)
	(m-1-2) edge[style=transparent] node[opaque,pos=0.1]{\scalebox{1.5}{$\lrcorner$}} (m-2-3)
	(m-1-3) edge[style=transparent] node[opaque,pos=0.1]{\scalebox{1.5}{$\lrcorner$}} (m-2-4)  
  ;
\end{tikzpicture}
\end{center}
\exend
\end{remark}

We will apply Quillen's Theorem A to the functor $\Psi$ to show that it induces a homotopy equivalence of geometric realisations. First of all, we make the following observations.

\begin{lemma}\label{Psi fully faithful}
The functor $\Psi\colon Q(\mathscr{C})\rightarrow Q_1(M_{\mathscr{C}})$ is fully faithful.
\end{lemma}
\begin{proof}
It is easy to see that it is full, since a 3-step filtration in $y$ whose associated graded is $(a,x,b)$ identifies $x$ as a subquotient of $y$, which exactly corresponds to a morphism in $Q(\mathscr{C})$. To see that it is faithful, consider diagrams as below and assume that they define the same morphism in $Q_1(M_{\mathscr{C}})$.
\begin{center}
\begin{tikzpicture}
\matrix (m) [matrix of math nodes,row sep=2em,column sep=2em, text height=1.5ex, text depth=0.25ex]
  {
 	a & z & y & & c & z' & y\\
 	a & x & b & & c & x & d\\
  };
  \path[>->] 
	(m-1-1) edge (m-1-2)
	(m-1-2) edge (m-1-3)
	(m-1-5) edge (m-1-6)
	(m-1-6) edge (m-1-7)
  ;
  \path[->>]
	(m-1-2) edge (m-2-2)
	(m-1-3) edge (m-2-3)
	(m-1-6) edge (m-2-6)
	(m-1-7) edge (m-2-7)
  ;
  \path[-]
	(m-1-1) edge[double equal sign distance] (m-2-1)
	(m-1-5) edge[double equal sign distance] (m-2-5)
  ;
\end{tikzpicture}
\end{center}

Then there exist isomorphisms $\alpha\colon c\rightarrow a$ and $\beta\colon d\rightarrow b$ such that the composite of the associated graded of the right hand diagram with $(\alpha ,\operatorname{id}_x, \beta)$ defines the same filtration with associated graded as the left hand diagram. In particular, there exists a (unique) isomorphism $z\rightarrow z'$ which commutes with the morphisms to $x$ and $y$, i.e.~$[x\twoheadleftarrow z\rightarrowtail y]=[x\twoheadleftarrow z'\rightarrowtail y]$.
\end{proof}

The following proposition is an immediate consequence of \Cref{terminal decomposition} and \Cref{Psi fully faithful}.

\begin{proposition}\label{representatives of morphisms in Q_1 using Psi}
Let $x$ be an object in $Q(\mathscr{C})$ and $(n_j)_{j\in J}$ an object in $M_{\mathscr{C}}$. Any morphism in $Q_1(M_{\mathscr{C}})$ from $\Psi(x)$ to $(n_j)_{j\in J}$ can be written uniquely as a composite
\begin{align*}
[(n_j)_{j<j_0},\ (n_j)_{j>j_0},\ \operatorname{id}_{(n_j)_{j\in J}}]\circ \Psi([x\twoheadleftarrow z\rightarrowtail n_{j_0}])
\end{align*}
for some $j_0\in J$ and some morphism $[x\twoheadleftarrow z\rightarrowtail n_{j_0}]$ in $Q(\mathscr{C})$.
\end{proposition}

We now show that the comma category $\Psi\downarrow \alpha$ has contractible geometric realisation for any object $\alpha$ in $Q_1(M_{\mathscr{C}})$.

\medskip

Let $(m_i)_{i\in I}$ be an object in $M_{\mathscr{C}}$. The comma category $\Psi\downarrow (m_i)$ has objects
\begin{align*}
\bigg(x,\ \bigg[(a_i),\ (b_i),\ \phi\colon (a_i)\circledast \Psi(x)\circledast (b_i)\rightarrow (m_i)\bigg]\bigg)
\end{align*}
where $x$ is an object in $\mathscr{C}$, and $[(a_i),(b_i),\phi]\colon \Psi(x)\rightarrow (m_i)$ is a morphism in $Q_1(M_{\mathscr{C}})$. A morphism in $\Psi\downarrow (m_i)$ is of the form
\begin{align*}
(x,\,[(a_i),(b_i),\phi])\xrightarrow{[x\twoheadleftarrow z\rightarrowtail y]}(y,\,[(c_i),(d_i),\psi])
\end{align*}
where $[x\twoheadleftarrow z\rightarrowtail y]$ is a morphism in $Q(\mathscr{C})$ such that
\begin{align*}
[(a_i),(b_i),\phi]=[(c_i),(d_i),\psi]\circ \Psi([x\twoheadrightarrow z\rightarrowtail y]).
\end{align*}

For every $i_0\in I$, set
\begin{align*}
m_{\leq i_0}:=(m_i)_{i\in I_{\leq i_0}},\quad m_{< i_0}:=(m_i)_{i\in I_{< i_0}},\quad m_{\geq i_0}:=(m_i)_{i\in I_{\geq i_0}},\quad m_{> i_0}:=(m_i)_{i\in I_{> i_0}}.
\end{align*}

\medskip

For all $i_0\in I$, consider the full subcategory $C_{i_0}\subseteq \Psi\downarrow (m_i)$ spanned by the objects of the form
\begin{align*}
\bigg(x,\ [m_{<i_0},m_{>i_0},\operatorname{id}]\circ \Psi([x\twoheadleftarrow z\rightarrowtail m_{i_0}]\bigg).
\end{align*}

The following two lemmas are immediate consequences of \Cref{representatives of morphisms in Q_1 using Psi}.

\begin{lemma}\label{C_i has terminal object}
Let $i_0\in I$. The object $(m_{i_0},[m_{<i_0},m_{>i_0},\operatorname{id}])$ is a terminal object in $C_{i_0}$. In particular, $C_{i_0}$ has contractible geometric realisation.
\end{lemma}

\begin{lemma}\label{C_i's cover comma category}
The subcategories $C_{i_0}$, $i_0\in I$, cover $\Psi\downarrow (m_i)$.
\end{lemma}

With this, we can make the final observation.

\begin{proposition}\label{contractible comma category}
The comma category $\Psi\downarrow (m_i)$ has contractible geometric realisation.
\end{proposition}
\begin{proof}
Note that for any $i,j\in I$,
\begin{align*}
C_i\cap C_j=\begin{cases}
(0,[m_{<i_0},m_{\geq i_0},\operatorname{id}]) & \text{if } \{i,j\}=\{i_0,i_0+1\}, \\
C_i & \text{if } i=j, \\
\emptyset & \text{else},
\end{cases}
\end{align*}
where $(0,[m_{<i},m_{\geq i},\operatorname{id}])$ denotes the terminal category on this object. Hence, $|\Psi\downarrow (m_i)|$ is contractible by \Cref{closeddescentcat}, since the nerve of this cover is contractible (alternatively, one can use the Nerve Theorem of \cite[Page 234]{Borsuk48}, see also \cite[Exercise 4G.4 and Corollary 4G.3]{Hatcher}).
\end{proof}

Then by Quillen's Theorem A (\cite{Qfg}) and the proposition above, we have the following result.

\begin{proposition}\label{Quillen's Q-construction and our collapsed Q-construction}
The functor $\Psi\colon Q(\mathscr{C})\rightarrow Q_1(M_{\mathscr{C}})$ induces a homotopy equivalence of geometric realisations.
\end{proposition}

Combining this with \Cref{Q-construction contract to 1-categories}, we have the following.

\begin{corollary}\label{Q-constructions have homotopy equivalent classifying spaces}
The zig-zag
\begin{align*}
Q(\mathscr{C})\xrightarrow{\ \Psi\ } Q_1(M_{\mathscr{C}})\xleftarrow{\kappa_{M_{\mathscr{C}}}} Q_2(M_{\mathscr{C}})
\end{align*}
induces a homotopy  equivalence of geometric realisations, $|Q(\mathscr{C})|\simeq |Q_2(M_{\mathscr{C}})|$.
\end{corollary}

Now we can combine this with the results of \Cref{monoidal categories and actions} to show that the monoidal category of flags and associated gradeds produces a model for the algebraic K-theory space. More precisely, we apply \Cref{Q construction and monoidal category have homotopy equivalent classifying spaces} which says that for any strict monoidal category $M$, there is a homotopy equivalence $B|M|\simeq |Q_2(M)|$ between the classifying space of the topological monoid $|M|$ and the geometric realisation of the $2$-categorical Q-construction.

\begin{theorem}\label{the K-theory space}
For any category with filtrations $\mathscr{C}$, the geometric realisation of Quillen's Q-construction $Q(\mathscr{C})$ is homotopy equivalent to the classifying space $B|M_{\mathscr{C}}|$ of the topological monoid $|M_{\mathscr{C}}|$. In particular, for any exact category $\mathscr{E}$, the space $\Omega B|M_{\mathscr{E}}|$ is a model for the algebraic K-theory space $K(\mathscr{E})$.
\end{theorem}

\appendix
\section{Nerves and geometric realisations}\label{appendix nerves and geometric realisations}

We give a quick recap of the definitions of double categories and $2$-categories, their nerves and their geometric realisations. We only define the notions that we will need and we refrain from specifying the various coherency axioms; these can be found in any good source on the subject (see for example \cite{Leinster}).

\subsection{Double categories.}

Let $\operatorname{Cat}$ denote the category of small categories.

\begin{definition}
A \textit{double category} is a category internal to $\operatorname{Cat}$: it consists of an object category $C_0$ and a morphism category $C_1$ equipped with source and target maps $s,t\colon C_1\rightarrow C_0$, an identity section $e\colon C_0\rightarrow C_1$, and a vertical composition rule $c\colon C_1\times_{C_0}C_1\rightarrow C_1$ satisfying the necessary coherency axioms. We write $\mathscr{C}=[C_1\rightrightarrows C_0]$, omitting the identity and vertical composition functors from the notation.

\medskip

The objects of $C_0$ are called the \textit{objects} of $\mathscr{C}$, the morphisms of $C_0$ are call the \textit{horizontal morphisms}, the objects of $C_1$ are called the \textit{vertical morphisms}, and the morphisms of $C_1$ are called \textit{$2$-cells}.
\defend
\end{definition}

\begin{definition}
Let $\mathscr{C}=[C_1\rightrightarrows C_0]$ be a double category. The \textit{transpose} $\mathscr{C}^t$ of $\mathscr{C}$ is the double category obtained by interchanging vertical and horizontal morphisms.
\defend
\end{definition}

\begin{definition}
Let $\mathscr{C}=[C_1\rightrightarrows C_0]$ and $\mathscr{D}=[D_1\rightrightarrows D_0]$ be double categories. A \textit{double functor} $F\colon \mathscr{C}\rightarrow \mathscr{D}$ is a pair of functors $(F_1\colon C_1\rightarrow D_1,F_0\colon C_0\rightarrow D_0)$ which commute with the source, target, identity and vertical composition functors.
\defend
\end{definition}

Let us remark that in the following definition, when we say simplicial category, we mean a simplicial object in categories and \textit{not} a category enriched in simplicial sets.

\begin{definition}
Let $\mathscr{C}=[C_1\rightrightarrows C_0]$. The \textit{horizontal nerve} of $\mathscr{C}$ is the simplicial category $N^h_\bullet(\mathscr{C})$ defined as follows: the category $N^h_n(\mathscr{C})$ has object set $N_n(C_0)$ and morphism set $N_n(C_1)$ with the inherited source and target maps, i.e.~a morphism from $c_0\xrightarrow{f_1} \cdots \xrightarrow{f_n} c_n$ to $d_0\xrightarrow{g_1} \cdots \xrightarrow{g_n} d_n$ is a sequence
\begin{align*}
\phi_0\xrightarrow{\alpha_1}\cdots \xrightarrow{\alpha_n}\phi_n
\end{align*}
in $C_1$ with $s(\phi_i)=c_i$, $t(\phi_i)=d_i$, $s(\alpha_i)=f_i$ and $t(\alpha_i)=g_i$ for all $i$. Composition is given by vertical composition in $\mathscr{C}$.

\medskip

The \textit{vertical nerve} of $\mathscr{C}$ is the simplicial category $N^v_\bullet(\mathscr{C})=N^h_\bullet(\mathscr{C}^t)$ given by the horizontal nerve of the transpose double category. More precisely, the category $N^v_n(\mathscr{C})$ has as objects sequences of vertical morphisms
\begin{align*}
c_0\xrightarrow{\phi_1} c_1 \xrightarrow{\phi_2} \cdots \xrightarrow{\phi_n} c_n
\end{align*}
and a morphism from $c_0\xrightarrow{\phi_1} \cdots \xrightarrow{\phi_n} c_n$ to $d_0\xrightarrow{\psi_1} \cdots \xrightarrow{\psi_n} d_n$ is a collection of $2$-cells
\begin{align*}
\alpha_i\colon \phi_i\Rightarrow \psi_i
\end{align*}
satisfying $t(\alpha_i)=s(\alpha_{i+1})$ for all $i$.

\medskip

The \textit{double nerve} of $\mathscr{C}$ is the bisimplicial set $N_{\bullet\bullet}(\mathscr{C})$ obtained by applying the usual $1$-categorical nerve functor levelwise to the horizontal nerve of $\mathscr{C}$:
\begin{align*}
N_{nk}(\mathscr{C})=N_n^{}(N^h_k(\mathscr{C}))=N_k^{}(N^v_n(\mathscr{C})).
\end{align*}

\medskip

The \textit{geometric realisation} of $\mathscr{C}$, denoted $|\mathscr{C}|$, is the total geometric realisation of $N_{\bullet\bullet}(\mathscr{C})$.
\defend
\end{definition}

\begin{observation}
A double functor induces a continuous map of geometric realisations.
\exend
\end{observation}

\subsection{Strict \texorpdfstring{$2$}{2}-categories.}

We will need to work with $2$-categories which are not necessarily small, nor even locally small, that is, the hom-categories need not be small either. However, we may restrict our attention to strict $2$-categories.

\begin{definition}
A \textit{strict $2$-category} $Q$ consists of a collection $\operatorname{ob} Q$ of objects and for each pair of objects $a,b\in \operatorname{ob} Q$, a hom-category $Q(a,b)$. It is equipped with composition functors $Q(b,c)\times Q(a,b)\rightarrow Q(a,c)$ and identities $\operatorname{id}_a\in Q(a,a)$ for all $a,b,c\in \operatorname{ob} Q$, and these must satisfy the necessary (strict) coherency axioms.

\medskip

The objects of the hom categories are called \textit{morphisms} and the morphisms are called \textit{$2$-cells}. We denote composition of $2$-cells along morphisms (within the hom-categories) by $\circ_h$ ($h$ for horizontal), and composition of $2$-cells along objects (via the composition functors) by $\circ_v$ ($v$ for vertical). Composition of morphisms will be denoted by $\circ$ (note that this is also given by the composition functors, so these composites are compatible with the vertical composition of $2$-cells).

\medskip

Let $Q,R$ be strict $2$-categories. A \textit{pseudofunctor} $F\colon Q\rightarrow R$ consists of the following data:
\begin{enumerate}
\item an assignment $F\colon \operatorname{ob} Q\rightarrow \operatorname{ob} R$,
\item for every pair of objects $a,b\in \operatorname{ob} Q$, a functor $F_{a,b}\colon Q(a,b)\rightarrow R(F(a),F(b))$,
\item for any pair of composable morphism $f\colon a\rightarrow b$, $g\colon b\rightarrow c$ in $Q$, an invertible $2$-cell $\hat{F}_{g,f}\colon F_{b,c}(g)\circ F_{a,b}(f)\rightarrow F_{a,c}(g\circ f)$,
\item for all objects $a\in \operatorname{ob} Q$, an invertible $2$-cell $\hat{F}_a\colon \operatorname{id}_{F(a)}\rightarrow F_{a,a}(\operatorname{id}_a)$,
\end{enumerate}
subject to the necessary coherency axioms.

We say that a pseudofunctor is \textit{strict} (or a \textit{$2$-functor}), if the invertible $2$-cells above are the identity.
\defend
\end{definition}

\begin{definition}
Let $Q,R$ be strict $2$-categories, and let $F,G\colon Q\rightarrow R$ be pseudofunctors. An \textit{oplax natural transformation} $\alpha\colon F\Rightarrow G$ consists of the following data:
\begin{itemize}[label=$\ast$]
\item for each $a\in \operatorname{ob} Q$, a morphism $\alpha_a\colon F(a)\rightarrow G(a)$,
\item for all $a,b\in \operatorname{ob} Q$, a natural transformation
\begin{align*}
\hat{\alpha}\colon(\alpha_b)_*\circ F_{a,b}\Rightarrow (\alpha_a)^*\circ G_{a,b}
\end{align*}
of functors $Q(a,b)\rightarrow R(F(a),G(b))$,
\end{itemize}
such that the following diagrams of horizontal composites of $2$-cells commute:
\begin{enumerate}
\item for all $a\in \operatorname{ob} Q$,
\begin{center}
\begin{tikzpicture}
\matrix (m) [matrix of math nodes,row sep=2em,column sep=2em]
  {
	\alpha_a & G_{a,a}(\operatorname{id}_a)\circ \alpha_a \\
	\alpha_a \circ F_{a,a}(\operatorname{id}_a) &  \\
  };
  \path[-stealth] 
	(m-1-1) edge node[above]{$\scriptstyle\hat{G}_a\circ_v\operatorname{id}_{\alpha_a}$} (m-1-2) edge node[left]{$\scriptstyle\operatorname{id}_{\alpha_a}\circ_v\hat{F}_a$} (m-2-1)
	(m-2-1) edge node[below right]{$\scriptstyle\hat{\alpha}_{\operatorname{id}_a}$} (m-1-2)
  ;
\end{tikzpicture}
\end{center}

\item for all composable morphisms $g\colon a\rightarrow b$, $f\colon b\rightarrow c$ in $Q$,
\begin{center}
\begin{tikzpicture}
\matrix (m) [matrix of math nodes,row sep=2em,column sep=3em]
  {
	\alpha_c \circ F_{b,c}(g)\circ F_{a,b}(f) & \alpha_c\circ F_{a,c}( g\circ f) & G_{a,c}(g\circ f)\circ\alpha_a \\
	G_{b,c}(g)\circ\alpha_b\circ F_{a,b}(f) & G_{b,c}(g)\circ G_{a,b}(f) \circ \alpha_a &  \\
  };
  \path[-stealth] 
	(m-1-1) edge node[above]{$\scriptstyle\operatorname{id}_{\alpha_c}\circ_v \hat{F}_{g,f}$} (m-1-2) edge node[left]{$\scriptstyle\hat{\alpha}_g\circ_v \operatorname{id}_{F_{a,b}(f)}$} (m-2-1)
	(m-1-2) edge node[above]{$\scriptstyle\hat{\alpha}_{g\circ f}$} (m-1-3)
	(m-2-1) edge node[below]{$\scriptstyle\operatorname{id}_{G_{b,c}(g)}\circ_v \hat{\alpha}_f$} (m-2-2)
	(m-2-2) edge node[below right]{$\scriptstyle\hat{G}_{g,f} \circ_v \operatorname{id}_{\alpha_a}$} (m-1-3)
  ;
\end{tikzpicture}
\end{center}
\end{enumerate}
A \textit{lax natural transformation} is as above, but with the natural transformation $\hat{\alpha}$ reversed.
\defend
\end{definition}

\begin{definition}
A strict $2$-category $Q$ is \textit{small} if the hom-categories are small and $\operatorname{ob} Q$ is a set. It is \textit{essentially small} if it is equivalent to a small $2$-category.
\defend
\end{definition}

\begin{observation}
Any small strict $2$-category $Q$ can be viewed as a double category with only identity horizontal morphisms, i.e. with discrete object category:
\begin{align*}
\mathscr{Q}=\bigg[\coprod_{a,b\in \operatorname{ob} Q} Q(a,b)\ \rightrightarrows\  \operatorname{ob} Q\ \bigg]
\end{align*}
with the obvious structure maps.
\exend
\end{observation}

\begin{definition}
Let $Q$ be a small strict $2$-category. The \textit{geometric realisation} $|Q|$ of $Q$ is the geometric realisation of the double category $\mathscr{Q}$ defined above.
\defend
\end{definition}

\begin{remark}
This definition agrees with the usual definition of the geometric realisation of $Q$ via the double nerve (see for example \cite{BullejosCegarra}). There are, however, various options for defining the nerve of a small (strict) $2$-category. See \cite{CarrascoCegarraGarzon} for a comparison in which it is also established that the ten different nerve constructions (of small bicategories) that they consider all have homotopy equivalent geometric realisations. 
\exend
\end{remark}

The following proposition is most easily proved by exploiting the fact that there is a natural homotopy equivalence between the geometric realisation of a small $2$-category and the geometric realisation of its geometric nerve (\cite[Theorem 1]{BullejosCegarra}), see for example \cite[Proposition 7.1]{CarrascoCegarraGarzon}.
\begin{proposition}
Lax and oplax natural transformations induce homotopies between the induced maps of geometric realisations.
\end{proposition}

\begin{definition}
For essentially small strict $2$-category $Q$, we define the \textit{geometric realisation} $|Q|$ of $Q$ to be the geometric realisation of any equivalent small $2$-category.
\defend
\end{definition}

\bibliographystyle{alpha}
\bibliography{RBS}

\end{document}